\documentclass{amsart}

\usepackage[utf8]{inputenc}

\usepackage{soul}

\usepackage{graphicx}
\usepackage{comment}

\usepackage{todonotes}

\usepackage{amssymb,amsmath,amsthm,stmaryrd,mathrsfs,wasysym,bbm}
\usepackage{enumitem,mathtools,xspace, extarrows}
\usepackage{xstring} 
\usepackage{aliascnt}
\definecolor{greencolor}{rgb}{0,0.45,0}
\definecolor{ucphcolor}{rgb}{0.517,0.016,0.016}
\definecolor{OliveGreen}{rgb}{0,0.6,0}
\usepackage[colorlinks,
            citecolor=greencolor,
            linkcolor=ucphcolor,
            urlcolor=greencolor, unicode]{hyperref}
\usepackage[nameinlink,capitalise,noabbrev]{cleveref}

\usepackage[backend=biber, style=alphabetic, maxnames=10, maxalphanames=3, minalphanames=3]{biblatex}
\usepackage{tikz}
\usetikzlibrary{decorations.pathmorphing,arrows}

\usepackage{lineno}

\let\autoref\cref

\newcommand*\cocolon{%
        \nobreak
        \mskip6mu plus1mu
        \mathpunct{}%
        \nonscript
        \mkern-\thinmuskip
        {:}%
        \mskip2mu
        \relax
}

\usepackage{adjustbox}

\usepackage{contour}
\usepackage{ulem}

\contourlength{0.8pt}

\newcommand{\myuline}[1]{%
  \uline{\phantom{#1}}%
  \llap{\contour{white}{#1}}%
}

\makeatletter
\newcommand*{\saved@myuline}{}
\let\saved@myuline\myuline

\newcommand*{\mathuline}{%
  \mathpalette{\math@myuline\saved@myuline}%
}
\newcommand*{\math@myuline}[3]{%
  \mbox{#1{$#2#3\m@th$}}%
}

\renewcommand*{\myuline}{%
  \relax  
  \ifmmode
    \expandafter\mathuline
  \else
    \expandafter\saved@myuline
  \fi
}
\makeatother

\usepackage{tikz}
\usepackage{tikz-cd}
\usetikzlibrary{cd}
\usetikzlibrary{arrows, arrows.meta, positioning, calc}
\let \amsamp = &

\usepackage[backend=biber, style=alphabetic, maxnames=10, maxalphanames=3, minalphanames=3]{biblatex}
\Crefname{prop}{Proposition}{Propositions}
\Crefname{lem}{Lemma}{Lemmas}
\Crefname{cor}{Corollary}{Corollaries}
\Crefname{thm}{Theorem}{Theorems}
\Crefname{mainthm}{Theorem}{Main theorem}
\Crefname{alphThm}{Theorem}{Theorems}
\Crefname{alphProp}{Proposition}{Propositions}
\Crefname{mainthm}{Theorem}{Main theorem}

\Crefname{defn}{Definition}{Definitions}
\Crefname{notation}{Notation}{Notations}
\Crefname{cons}{Construction}{Constructions}
\Crefname{rmk}{Remark}{Remarks}
\Crefname{obs}{Observation}{Observations}
\Crefname{trick}{Trick}{Tricks}
\Crefname{warning}{Warning}{Warnings}
\Crefname{conj}{Conjecture}{Conjectures}
\Crefname{assump}{Assumption}{Assumptions}
\Crefname{recollect}{Recollection}{Recollections}
\Crefname{terminology}{Terminology}{Terminologies}
\Crefname{condition}{Condition}{Conditions}
\Crefname{setting}{Setting}{Settings}
\Crefname{hypothesis}{Hypothesis}{Hypotheses}

\Crefname{question}{Question}{Questions}
\Crefname{example}{Example}{Examples}

\Crefname{figure}{Figure}{Figures}

\crefformat{equation}{(#2#1#3)}
\crefformat{section}{\S#2#1#3}
\crefmultiformat{section}{\S\S#2#1#3}{and~#2#1#3}{, #2#1#3}{, and~#2#1#3}

\newtheorem{thm}[subsubsection]{Theorem}
\newtheorem*{mainthm}{Main Theorem}
\newtheorem{prop}[subsubsection]{Proposition}
\newtheorem{lem}[subsubsection]{Lemma}
\newtheorem{cor}[subsubsection]{Corollary}
\newtheorem{alphThm}{Theorem}

\newcommand{\neutralize}[1]{\expandafter\let\csname c@#1\endcsname\count@}
\makeatother

\theoremstyle{definition}
\newtheorem{defn}[subsubsection]{Definition}
\newtheorem{cons}[subsubsection]{Construction}
\newtheorem{nota}[subsubsection]{Notation}

\newtheorem{recollect}[subsubsection]{Recollections}

\newtheorem{warning}[subsubsection]{Warning}
\newtheorem{hypothesis}[subsubsection]{Hypothesis}
\newtheorem{rmk}[subsubsection]{Remark}
\newtheorem{obs}[subsubsection]{Observation}
\newtheorem{example}[subsubsection]{Example}

\newcommand{\yoneda}{\text{\usefont{U}{min}{m}{n}\symbol{'107}}}
\DeclareFontFamily{U}{min}{}
\DeclareFontShape{U}{min}{m}{n}{<-> dmjhira}{}

\newcommand{\operadsCat}{\mathrm{Opd}}
\newcommand{\dayConvolution}{\mathrm{Day}}
\newcommand{\wcat}{\mathrm{WCat}}
\newcommand{\presentableStable}[1]{\mathrm{Pr}^L_{{#1}-\mathrm{st}}}

\newcommand{\unitMapFromSpectra}[1]{\mathrm{U}_{#1}}

\newcommand{\largeCat}{\widehat{\mathrm{Cat}}}
\newcommand{\udl}[1]{\underline{#1}}

\newcommand{\proper}{\mathcal{P}}

\newcommand{\inflation}{\mathrm{infl}}
\newcommand{\epiCategory}{\mathrm{Epi}}
\newcommand{\completion}[1]{^{\wedge}_{#1}}

\newcommand{\elbowCat}[1]{\mathrm{Exc}^{\lrcorner}_{{#1}}}
\newcommand{\elbowCatUdl}[1]{\underline{\mathrm{Exc}}^{\lrcorner}_{{#1}}}
\newcommand{\excisiveApproximation}{P}

\newcommand{\baseCat}{\mathscr{T}}

\newcommand{\reduce}{^{\mathrm{red}}}
\newcommand{\crossEffect}{\mathrm{cr}}
\newcommand{\multiexcisive}{\mathrm{MultiExc}}
\newcommand{\excisive}{\mathrm{Exc}}

\newcommand{\terminalTCat}{\underline{\ast}}

\newcommand{\coind}{\mathrm{Coind}}

\newcommand{\surjectiveSet}{{\mathrm{Epi}}}

\newcommand{\inclusion}{\mathrm{incl}}

\newcommand{\susps}{\Sigma^{\infty}}

\newcommand{\hhom}{\mathrm{Hom}}
\newcommand{\calg}{\mathrm{CAlg}}
\newcommand{\spc}{\mathcal{S}}
\newcommand{\vop}{^{\myuline{\mathrm{op}}}}

\newcommand{\ind}{\mathrm{Ind}}

\DeclareMathOperator{\cmonoid}{\mathrm{CMon}}

\newcommand{\forget}{\mathrm{fgt}}

\newcommand{\loops}{\Omega^{\infty}}

\newcommand{\eval}{\mathrm{ev}}

\newcommand{\family}{\mathcal{F}}

\newcommand{\constant}{\operatorname{const}}
\newcommand{\cat}{\mathrm{Cat}}

\newcommand{\presentable}{\mathrm{Pr}}

\newcommand{\A}{\mathcal{A}}

\newcommand{\sC}{{\mathcal C}}

\newcommand{\D}{{\mathcal D}}

\newcommand{\op}{^{\mathrm{op}}}

\DeclareMathOperator{\mapsp}{\mathrm{map}}
\newcommand{\sphere}{\mathbb{S}}

\DeclareMathOperator{\res}{\mathrm{Res}}
\DeclareMathOperator{\mackey}{\mathrm{Mack}}
\DeclareMathOperator{\presheaf}{\mathrm{PSh}}
\newcommand{\orbit}{\mathcal{O}}

\newcommand{\exact}{^{\mathrm{ex}}}
\newcommand{\spectra}{\mathrm{Sp}}
\newcommand{\finite}{\mathrm{Fin}}

\DeclareMathOperator{\spanCat}{\mathrm{Span}}
\DeclareMathOperator{\im}{\mathrm{Im}}
\newcommand{\id}{\mathrm{id}}
\DeclareMathOperator{\map}{\mathrm{Map}}
\DeclareMathOperator{\fib}{\operatorname{fib}}
\DeclareMathOperator{\cofib}{\operatorname{cofib}}
\DeclareMathOperator{\func}{\mathrm{Fun}}

\newcommand{\module}{\mathrm{Mod}}
\newcommand{\unit}{\mathbbm{1}}

\def\colim{\qopname\relax m{colim}}

\newcommand{\arrdisp}{0.33ex}
\newcommand{\arrdisplacementsp}{0.72ex}

\newcommand{\ardis}{\ar@<\arrdisp>}
\newcommand{\ardissp}{\ar@<\arrdisplacementsp>}



\newcommand*{\Scale}[2][4]{\scalebox{#1}{$#2$}}

\title{On the universality of multiexcisive functors}
\author{{Tobias Barthel} \and {Kaif Hilman} \and {Nikolay Konovalov}}

\date{\today}

\makeatletter
\patchcmd{\@setaddresses}{\indent}{\noindent}{}{}
\patchcmd{\@setaddresses}{\indent}{\noindent}{}{}
\patchcmd{\@setaddresses}{\indent}{\noindent}{}{}
\patchcmd{\@setaddresses}{\indent}{\noindent}{}{}
\makeatother

\address{Tobias Barthel, Max Planck Institute for Mathematics, Vivatsgasse 7, 53111 Bonn, Germany}
\email{tbarthel@mpim-bonn.mpg.de}
\urladdr{https://sites.google.com/view/tobiasbarthel/}

\address{Kaif Hilman, Mathematik Zentrum der Universität Bonn, Endenicher Allee 60, 53115 Bonn, Germany}
\email{kaif@math.uni-bonn.de}
\urladdr{https://sites.google.com/view/kaif-hilman/}

\address{Nikolay Konovalov, University of Chicago, Eckhart Hall, 5734 S University Ave, Chicago IL 60637, USA}
\email{nkonovalov@uchicago.edu}
\urladdr{https://nkonovalov.gitlab.io/webpage/}

\begin{document}

\begin{abstract}
    We provide a multiplicative classification of polynomial endofunctors on spectra in terms of their Mackey functors of cross--effects. More precisely, we prove that various categories of multivariable excisive functors from spectra to spectra are symmetric monoidally equivalent to the corresponding variants  of spectral Mackey functors. The symmetric monoidal structures appearing here are the Day convolutions on both sides, and the Mackey functors we consider involve variations on the category of finite sets and surjections.  
    
    The method is first to  introduce certain multivariable functors we call subdiagonal functors. By considering them all at once using parametrised category theory, we prove inductively that they all admit  Mackey functor descriptions as symmetric monoidal categories, endowing them with a  universal property along the way.
    In particular, specialising this to univariate functors  gives a new proof and strengthening of Glasman's result about $d$--excisive endofunctors on spectra. As  application of our perspective, we prove a ``Segal conjecture'' in the context of  Goodwillie calculus when $d$ is a prime number.
\end{abstract}



\maketitle

\vspace{0.5cm}
\tableofcontents

\vspace{-1cm}
\section{Introduction}

The study of ``polynomial'' functors has a long and established history in algebra and topology. 
In the derived and homotopical setting, polynomiality is often phrased in terms of Goodwillie's far-reaching and fundamental theory of  calculus via his notion of \textit{excisive functors}, see \cite{GoodwillieCalculus1,GoodwillieCalculus2,GoodwillieCalculus3}. As such, the category $\excisive_d$ of reduced $d$--excisive endofunctors on spectra is of central importance in Goodwillie calculus. In a series of deep work \cite{aroneChingChainRule,aroneChingClassification,aroneChingCrossEffects}, Arone--Ching developed a presentation of this category using Goodwillie's derivatives.

In a different direction, it has been observed that functor calculus on spectra exhibits many features reminiscent of equivariant stable homotopy theory. For example, the excisive approximation fracture squares of Kuhn looks like the standard Tate fracture squares in the equivariant realm. Prompted by this, Arone--Ching and McCarthy have shown that there is in fact an equivalence between the category of reduced $(d=2)$--excisive endofunctors on spectra with the category of genuine $C_2$--spectra, induced by taking \textit{cross--effects} as opposed to derivatives. 

This equivalence does not generalise to higher degrees $d>2$, due to a bifurcation: While $G$--spectra can be described as Mackey functors on the $G$--orbit category (\cite{guillouMay,barwick1,nardinExposeIV,CMMN2}), $\excisive_d$ is controlled by the category $\epiCategory_d$ of finite sets up to size $d$ and surjections. For $d=2$, we see that $\epiCategory_d$ is isomorphic to the orbit category for $C_2$, explaining the aforementioned result; in contrast, $\epiCategory_d$ is not realised as the orbit category of a finite group for $d>2$.

The main purpose of this paper is to systematically and completely develop a multiplicative Mackey  classification  of excisive functors in terms of cross--effects:

\begin{mainthm}
    The category $\excisive_d$ can be canonically enhanced to an $\epiCategory_{d}$--parametrised symmetric monoidal stable category $\elbowCatUdl{d}$ of multiexcisive functors and it is the initial such category. In particular, taking cross--effects induces a canonical symmetric monoidal equivalence 
        \[
            \excisive_d\simeq \mackey(\epiCategory_{d};\spectra)
        \]
    between the category of reduced $d$--excisive functors and the category of spectral $\epiCategory_d$--Mackey functors  equipped with the Day convolution structures.
\end{mainthm}
 
This result is couched in the language of parametrised category theory. We will review this and explain our construction of $\elbowCatUdl{d}$ below. Importantly, however, let us point out already at this juncture that the extra layer of structure exhibited here is not for abstract generality's sake, but rather is fundamental even for our approach to classifying the unstructured equivalence $\excisive_d\simeq \mackey(\epiCategory_{d};\spectra)$.

In order to unpack some of the information that is contained in our main theorem, we spell out some of its consequences more explicitly:
    \begin{itemize}
        \item (Universality) We exhibit a universal property of excisive functors on spectra when viewed as a parametrised category, in parallel to the known universal property of equivariant spectra developed in \cite{nardinExposeIV,CLLSpans}.
        \item (Intermediate categories) The parametrised point of view provides a canonical solution to the previously open problem of constructing the correct targets of the analogues of restriction functors from equivariant homotopy theory. For this purpose, we introduce the  categories $\excisive_{d,r}^{\lrcorner}$ of \textit{subdiagonal functors} which might be of independent interest. 
        \item (Glasman's theorem) As a special case, we obtain a symmetric monoidal enhancement of unpublished work of Glasman \cite{glasmanGoodwillie}, who established an equivalence  $\excisive_d \simeq \mackey(\epiCategory_d;\spectra)$, understanding which was the initial inspiration and impetus for the present work. Since his equivalence was obtained by chaining together six equivalences, it is not clear which step is compatible with the multiplicative structures. The essential point here is that neither $\excisive_d$ nor $\mackey(\epiCategory_d;\spectra)$ satisfies any obvious multiplicative universal property which lends itself to a transparent construction of a symmetric monoidal functor one way or another.
    \end{itemize}
Our work may also viewed as a spectral analogue and extension of the classification of single variable functors for abelian groups and abelian categories by \cite{bauesDreckmannFranjouPirashvili}.

Throughout, we carefully consider the multiplicative features of $\excisive_d$ under the Day convolution structure, which recently have gained some attention. For example, building upon the insights of \cite{ChingDayConvolution}, Arone--Barthel--Heard--Sanders \cite{ABHS24} computed  the Balmer spectrum of $\excisive_d$; on the other hand, in work in progress, Marius Nielsen investigates the Picard group of said category and also studies a different approach to a symmetric monoidal version of Glasman's theorem. Finally, as a byproduct of phrasing everything in this more structured fashion, we are able to make a mathematically precise dictionary between Goodwillie calculus and equivariant stable homotopy theory, which we collect in \cref{table:dictionary} after proving the main theorem. This extends the partial dictionary  as codified in \cite{ABHS24}. \\

\noindent \textbf{Outline of approach.} Before giving a more in--depth overview, we summarise the general strategy  in this paper. In general, it is difficult to build functors into or out of Mackey functors due to the subtle coherences in span categories. In contrast to that, it is very easy to build symmetric monoidal colimit--preserving functors out of the category $\spectra$ of spectra since it is the initial presentably symmetric monoidal stable category. Nevertheless, various spectral Mackey functors associated to the category $\epiCategory_d$ do assemble to an object $\myuline{\spectra}$ in $\cat_{\epiCategory_d}\coloneqq\func(\epiCategory_d\op,\cat)$ with a similar universal property by  \cite{nardinExposeIV,CLLSpans}. Our strategy is to exploit this universal property by locating the category $\excisive_d$ as part of a system of categories comparable to $\myuline{\spectra}$.  We achieve this by introducing the main definition of the article, namely that of \textit{subdiagonality} for multivariable functors (see \cref{cons:intermediate_categories}). Our approach also shows  that  $\excisive_d$ \textit{does} acquire a universal property akin to that of ordinary spectra not by itself, but when taken  as an assemblage of categories together with the multivariable subdiagonal functors (see \cref{thm:excisive_functors_as_spectral_mackey}). \\

\noindent \textbf{Subdiagonal functors.} To set the context, let $\sC$ be a small stable category, $ \D$ be a presentable stable category,  $1\leq r\leq s\leq  d$, and $f\colon [s]\twoheadrightarrow[r]$ a surjection from a finite set of size $s$ to one of size $r$ whose preimage of $i\in[r]$ has size $k_i$. First of all, recall that Goodwillie's cross--effect functors improve excisiveness in that, if $F\colon \sC^{\times r}\rightarrow \D$ is  reduced $d_i$--excisive in the $i$--th variable, then its $f$--cross--effect $\crossEffect_fF\colon \sC^{\times s}\rightarrow \D$ is reduced $(d_i-k_i+1)$--excisive in each of the  variables coming from the $i$--th variable of $F$. Conversely, excisiveness is additive in that, if $G\colon \sC^{\times s}\rightarrow\D$ were a functor which is reduced $e_j$--excisive in the $j$--th variable, then restricting along the diagonal $\Delta_f\colon \sC^{\times r}\rightarrow\sC^{\times s}$ gives a functor $\Delta_f^*G\colon \sC^{\times r}\rightarrow\D$ which is reduced $\sum_{j\in f^{-1}(i)}e_j$--excisive in the $i$--th variable.

Now, we write $\excisive_{d,r}(\sC,\D)$ for the full subcategory of $\func(\sC^{\times r},\D)$ of functors which are reduced and $(d-r+1)$--excisive in each variable. We write $\excisiveApproximation_{k_1,\ldots,k_r}$ for the endofunctor on $\func(\sC^{\times r},\D)$ given by taking the $k_i$--th excisive approximation in the $i$--th variable. By taking the Goodwillie tower in each of the $r$ variables, every object $F$ in $\excisive_{d,r}(\sC,\D)$ fits into a canonical subdivided cube whose terms are of the form $\excisiveApproximation_{k_1,\ldots,k_r}F$. We define the full subcategory of {subdiagonal functors} $\elbowCat{d,r}(\sC,\D)\subseteq \excisive_{d,r}(\sC,\D)$ to be those functors $F$ whose canonical subdivided cube from above is right Kan extended from the subdiagram consisting of those terms $P_{k_1,\ldots,k_r}F$  for which $k_1+\cdots+k_r\leq d$. Note that $\elbowCat{d,1}(\sC,\D)=\excisive_d(\sC,\D)$.

To explain the raison d'etre for these subdiagonal functors, note in general that the functor $\crossEffect_f\colon \excisive_{d,r}(\sC,\D)\rightarrow\func(\sC^{\times s},\D)$ (resp. $\Delta_f^*\colon \excisive_{d,s}(\sC,\D)\rightarrow\func(\sC^{\times r},\D)$) does not land in the full subcategory $\excisive_{d,s}(\sC,\D)$ (resp. $\excisive_{d,r}(\sC,\D)$). The reason here is that, for $F\in\excisive_{d,r}(\sC,\D)$ and $G\in\excisive_{d,s}(\sC,\D)$, the degrees of excisiveness in the variables of $\crossEffect_fF$ and $\Delta_f^*G$ depend on the sizes of the preimages of $f\colon[s]\twoheadrightarrow[r]$, thus precluding the possibility of having a uniform degree bound across the variables. The key observation now is the following:

\begin{alphThm}[{\cref{thm:key_pigeonhole_principle}}]\label{alphThm:key_pigeonhole_principle}
    We have the dashed factorisations
    \[
        \begin{tikzcd}
                \excisive_{d,r}^{\lrcorner}(\sC,\D)\dar[hook]\rar["{\crossEffect_{f}}",dashed]& \excisive_{d,s}^{\lrcorner}(\sC,\D)\dar[hook]\\
                \excisive_{d,r}(\sC,\D)\rar["{\crossEffect_{f}}"']& \func(\sC^{\times s},\D)
        \end{tikzcd}
        \hspace{1cm}
        \begin{tikzcd}
                \excisive_{d,r}^{\lrcorner}(\sC,\D)\dar[hook]& \excisive_{d,s}^{\lrcorner}(\sC,\D)\dar[hook]\lar["\Delta_f^*"',dashed]\\
                \func(\sC^{\times r},\D)& \excisive_{d,s}(\sC,\D). \lar["\Delta_f^*"]
        \end{tikzcd}
    \]
\end{alphThm}

\noindent For instance, in the special case of $r=1$, the result says that $\crossEffect_fF=\crossEffect_sF\colon\sC^{\times s}\rightarrow \D$ is not just $(d-s+1)$--excisive in each variable, but even satisfies the subdiagonality descent condition. Even more explicitly, it is worth pausing to contemplate the often enlightening case of $d=3$, $r=1$, and $s=2$. The result then says that while $\crossEffect_2F\colon\sC^{\times 2}\rightarrow\D$ is neither $(2,1)$--excisive nor $(1,2)$--excisive,  it is nonetheless built  by gluing these pieces in a canonical manner. 

We prove this theorem by appealing to the pigeonhole principle several times. To our eyes, this highlights that the subdiagonality condition occupies a certain numerical ``Goldilocks zone'' which could be of independent interest. With these said, we next explain how to exploit the peculiarities of these subdiagonal functors to obtain a Mackey description of multiexcisive functors using the theory of parametrised higher categories as introduced in \cite{parametrisedIntroduction,expose1Elements} and more specifically, the aspects thereof as developed  in \cite{nardinExposeIV,CLLSpans}.\\

\noindent \textbf{Parametrised categories and the main theorem.} 
In \cite{nardinThesis}, Nardin showed that, for so--called \textit{atomic orbital} categories $\baseCat$ (see \cref{defn:types of orbits}), the Mackey functor categories $\{\mackey(\baseCat_{/v};\spectra)\}_{v\in\baseCat}$ assemble to an object $\myuline{\spectra}$ in $\cat_{\baseCat}\coloneqq \func(\baseCat\op,\cat)$, usually called the \textit{$\baseCat$--category of genuine $\baseCat$--spectra}. Importantly, by introducing the concept of \textit{$\baseCat$--stability} and \textit{$\baseCat$--presentability} as $\baseCat$--parametrised versions of stability and presentability, he showed that $\myuline{\spectra}$ satisfies a universal property in $\cat_{\baseCat}$ analogous to the universal  property of $\spectra$ in $\cat$.

It turns out that for each $1\leq r\leq d$, the categories $\epiCategory_{d,r}\coloneqq (\epiCategory_d)_{/[r]}$ are all atomic orbital. Leveraging on \cref{alphThm:key_pigeonhole_principle}, we show that, for a fixed $d$, the categories $\{\elbowCat{d,r}(\sC,\D)\}_{1\leq r\leq d}$ assemble to an object $\elbowCatUdl{d}(\sC,\D)\in\cat_{\epiCategory_d}$.  But more is true, and this is distilled in the next theorem.

\begin{alphThm}[{\cref{thm:basic_properties_of_parametrised_excisive_category,thm:parametrised_excisive_category_is_symmetric_monoidal}}]\label{alphThm:parametrised_excisive_category_is_symmetric_monoidal}
Let $\sC\in \cat\exact$ and $\D\in \presentable^L_{\mathrm{st}}$. The object $\elbowCatUdl{d}(\sC,\D)\in\cat_{\epiCategory_d}$ is $\epiCategory_{d}$--presentable--stable. Moreover, in the presence of symmetric monoidal structures, i.e., if  $\sC\in\calg(\cat\exact)$ and $\D\in \calg(\presentable^L_{\mathrm{st}})$, then  $\elbowCatUdl{d}(\sC,\D)$ admits an $\epiCategory_{d}$--presentably symmetric monoidal structure coming from the Day convolutions.  
\end{alphThm}

In the special case of $\sC=\spectra^{\omega}$ and $\D=\spectra$, we write $\elbowCat{d,r}$ and $\elbowCatUdl{d}$ for $\elbowCat{d,r}(\sC,\D)$ and $\elbowCatUdl{d}(\sC,\D)$, respectively. By virtue of \cref{alphThm:parametrised_excisive_category_is_symmetric_monoidal} and Nardin's universal property for $\myuline{\spectra}$, we obtain a map $\myuline{\spectra}\rightarrow \elbowCatUdl{d}$ in $\calg(\cat_{\epiCategory_d})$. As part of the theory of subdiagonal functors, we are able to enumerate the generators in the categories $\elbowCat{d,r}$. Using this as an input for a ``geometric fixed points'' stratification induction, which we codify axiomatically as \cref{thm:abstractSpectraRecognition},  we obtain the  main theorem of the article.

\begin{alphThm}[{\cref{thm:excisive_functors_as_spectral_mackey}}]\label{alphThm:excisive_functors_as_spectral_mackey}
    Let $d$ be a positive integer and $r\leq d$. Then the  symmetric monoidal functor $\myuline{\spectra}\rightarrow \elbowCatUdl{d}$ is an equivalence. Evaluating at level $r$, this yields a symmetric monoidal equivalence $\elbowCat{d,r}\simeq \mackey(\epiCategory_{d,r};\spectra)$ with the Day convolution symmetric monoidal structures on both sides. 
\end{alphThm}

\noindent This strengthens Glasman's result in at least two ways: firstly, it gives a multiplicative enhancement of the Mackey description; secondly, along the way in the inductive procedure to recover his equivalence in the case of $r=1$, we also  obtain similar equivalences for multivariable functors. In fact, because we obtain an equivalence of $\epiCategory_d$--parametrised categories, $\elbowCatUdl{d}$ thus enjoys the  universal properties of $\myuline{\spectra}$, and  all these equivalences are compatible across different values of $d$ and $r$. We explain one consequence of this compatibility next. 

\vspace{1mm}

For $r\geq 1$, we will write $\epiCategory_{\infty,r}$ for the category $\epiCategory_{/[r]}$, so that the objects are full labelled partitions $(k_1,\ldots,k_r)$ of a set into $r$ nonempty bins.  To a stable presentable category $\D$, we may now define its associated category $\mackey(\epiCategory_{\infty,r};\D)$ of $\epiCategory_{\infty,r}$--Mackey functors and  we  define  $\mackey_{\mathrm{fs}}(\epiCategory_{\infty,r};\D)$ of $\epiCategory_{\infty,r}$ to be the full subcategory of \textit{finitely supported} Mackey functors, i.e., those that vanish away from finitely many objects of $\epiCategory_{\infty,r}$.  For a small stable category $\sC$, we will denote by $\multiexcisive(\sC^{\times r},\D)$ the full subcategory of $\func(\sC^{\times r},\D)$ consisting of those functors which are multireduced and multiexcisive, i.e., those which are reduced and $d_i$--excisive for some $d_i\geq 0$ in the $i$--th variable. We then have:

\begin{alphThm}[{\cref{cor:finitely_supported_equivalence}}]\label{alphThm:finitely_supported_equivalence}
    Let $\D$ be a presentable stable category and $r\geq 1$. The functor  which assigns to a multireduced and multiexcisive functor $F\colon (\spectra^{\omega})^{\times r}\rightarrow\D$  the Mackey functor which takes $(k_1,\ldots,k_r)$ to $(\crossEffect_{k_1,\ldots,k_r}F)(\sphere,\ldots,\sphere)$ participates in an equivalence
    \[\multiexcisive((\spectra^{\omega})^{\times r},\D)\simeq \mackey_{\mathrm{fs}}(\epiCategory_{\infty,r};\D).\]
\end{alphThm}

\noindent Observe that for $r\geq 2$, this is a priori \textit{not} the description  from applying the identification for $r=1$. We explain this point in \cref{rmk:multivariable_classification_not_currying}. As far as we are aware, this is the first such classification for multivariable functors.\\

\noindent\textbf{An application.} To close off this introduction, let us illustrate a potential use of these Mackey descriptions in Goodwillie calculus which highlights the technical as well as suggestive powers of the equivariant--to--calculus dictionary in \cref{table:dictionary}.

One of the deepest results in equivariant homotopy theory is the Segal conjecture, which is now a theorem. It says that there is an equivalence $\big((\sphere_G)^G\big)\completion{I(G)}\simeq \sphere^{BG} $ of spectra
where $G$ is a finite group, $(\sphere_G)^G$ is the genuine $G$--fixed points of the genuine $G$--spectrum, and $I(G)$ is the augmentation ideal of the Burnside ring $A(G)$. This was first proved by Lin \cite{lin} for $G=C_2$  and  by Adams--Gunawardena--Miller \cite{adamsGunawardenaMiller} for all elementary abelian $p$--groups by means of brilliant homological algebra computations, and finally by Carlsson \cite{carlssonSegalConjecture} for all finite groups in a celebrated tour de force of equivariant stable homotopy theory.

Under the dictionary in \cref{table:dictionary}, one may fairly ask if there might be an analogue in calculus. After a suitable translation, this asks whether or not the canonical map
\begin{equation}
    \label{eqn:naive_segal_conjecture}
    P_d\susps\loops(-)\completion{I(d)}\longrightarrow \big((-)^{\otimes d}\big)^{h\Sigma_d}
\end{equation}
of $d$--excisive functors $\spectra^{\omega}\rightarrow\spectra$ is an equivalence. Here, $I(d)$ is the augmentation ideal (cf.~\cref{cons:completion_at_ideals}) of the Goodwillie--Burnside ring $A(d)$ introduced in \cite{ABHS24}, and $(-)\completion{I(d)}$ is the corresponding $I(d)$--adic completion. From the perspective of calculus, the divided power functor $\big((-)^{\otimes d}\big)^{h\Sigma_d}$ is a complicated object, and \cref{eqn:naive_segal_conjecture} being equivalence would imply that we have algebraically decompleted it into the simpler object $P_d\susps\loops(-)$.

In the case $d=2$, since we have $\excisive_2\simeq \mackey(\epiCategory_2;\spectra)\simeq\spectra_{C_2}$, this map is indeed an equivalence by Lin's theorem. Disappointingly, the map \cref{eqn:naive_segal_conjecture} already fails to be an equivalence when $d=3$. However, we are still able to prove the following:

\begin{alphThm}[{\cref{thm:segal_conjecture_calculus}}]\label{alphThm:segal_conjecture}
    Let $d=p$ be a prime number. Then the map \cref{eqn:naive_segal_conjecture} is a $p$--adic equivalence.
\end{alphThm}

It is worth mentioning that, unlike the key case of $p$--groups in the original equivariant setting where the augmentation completion may be ignored after $p$--completing both sides, here it plays a crucial role in conspiring for the correct combinatorics to happen, and the result is false without it. For this reason, we think that the role of the augmentation ideal in calculus deserves further investigation.

We prove this by reducing it to the equivariant Segal conjecture using a stratification procedure via Goodwillie's derivatives. The key  here is  showing that the derivatives ``commute'' with the $I(d)$--completion under connectivity assumptions. This uses a connectivity argument by transferring the standard $t$--structure on $ \mackey(\epiCategory_d;\spectra)$ to a ``sensible'' one on $\excisive_d$ afforded to us by \cref{alphThm:excisive_functors_as_spectral_mackey}, which is furthermore compatible with the various stratifications, as well as a ``cellular model'' for the excisive approximations. It would be interesting to investigate   other values of $d$ for which a semblance of the theorem above can be arranged to hold.

\addtocontents{toc}{\protect\setcounter{tocdepth}{1}} 

\subsection{Organisation}\label{subsection:organisation}
This article is divided into two parts. We first develop all the Goodwillie calculus that we need in \cref{part:calculus} which we then package using the language of parametrised categories in \cref{part:parametrised}. We work out various properties of the cross--effects in \cref{section:multivariable_calculus}: we collect known facts about the cross--effects in \cref{subsection: cross-effect yoga}. We then introduce the subdiagonal functors and prove the key \cref{alphThm:key_pigeonhole_principle} in \cref{subsection:subdiagonal_functors}; we then work out the interaction between the cross--effects and restrictions along the diagonal on the subdiagonal functors in \cref{subsection:beckChevalley_squares}. Next, we develop the categorical structures and properties of these subdiagonal functors in \cref{section:categorical_properties_multiplicative}. We  show that subdiagonality is a smashing local subcategory in \cref{subsection:multiplicative_structures}, thus endowing multiplicative structures on these objects. In \cref{subsec:generators_and_localisations}, we enumerate the generators of the subdiagonal categories and describe the kernels of excisive approximations in terms of these generators.

We switch gears by working out various aspects of parametrised category theory in \cref{section:parametrised_categories} over arbitrary base categories $\baseCat$, which might be of independent interest. After recalling the setup in \cref{subsection:recollection_parametrised}, we introduce the universal spaces and their associated categorified ``geometric fixed point functors'' in \cref{subsection:universal_spaces_and_stratifications}. We then  axiomatise the relevant stratification arguments in a recognition principle for  spectra in \cref{subsection:recognition_principle}. 

Finally, we combine all the foregoing elements in \cref{section:parametrised_Goodwillie}: we assemble the subdiagonal functor categories into a parametrised category and prove \cref{alphThm:parametrised_excisive_category_is_symmetric_monoidal} in \cref{subsection:subdiagonal_parametrised}. Combining this with our knowledge of the generators and the recognition principle, we prove the main \cref{alphThm:excisive_functors_as_spectral_mackey} in \cref{subsection:main_theorem}. We also summarise here the comparisons with equivariant homotopy theory in a dictionary \cref{table:dictionary}.  We then deduce \cref{alphThm:finitely_supported_equivalence} in \cref{subsection:mackeYy_model_multiexcisive} and prove \cref{alphThm:segal_conjecture} in \cref{subsection:segal}.

\subsection{Acknowledgements} We thank Greg Arone, Thomas Blom, Drew Heard, Sil Linskens, Connor Malin,   Marius Nielsen, and Vignesh Subramanian for helpful conversations surrounding this work. TB and KH were supported by the European Research Council (ERC) under Horizon Europe (GeoCats, grant No.~101042990). KH is also supported by the European Research Council (ERC) under Horizon Europe (BorSym, ID: 101163408). All authors thank the Max Planck Institute for Mathematics (MPIM) in Bonn and the University of Bonn for their hospitality and conducive working environments.

\addtocontents{toc}{\protect\setcounter{tocdepth}{2}} 

\newpage
\part{Calculus}\label{part:calculus}

\section{Multivariable calculus}\label{section:multivariable_calculus}
\subsection{Recollections}\label{subsection: cross-effect yoga}

We recall  Goodwillie's cross--effects as presented in \cite[\textsection 6.1.3]{lurieHA} and record some of its properties that will be relevant for us. We then record some basic interactions of Day convolutions with the cross--effects and excisive approximations. We do not claim any originality in these recollections and they are recorded merely for the convenience of the reader.

Let $\sC,\D$ be stable categories and $\func_*(\sC,\D)$ denote the category of reduced functors and $\func_{\underline{\ast}}(\sC^{\times r},\D)$ the category of $r$-variable functors which are reduced in each variable. For $f\colon [s]\twoheadrightarrow [r]$ a surjection from a finite set of size~$s$ to one of size~$r$, we denote by $\Delta_f\colon \sC^{\times r}\rightarrow \sC^{\times s}$ the associated diagonal functor. By for instance \cite[Lem. B.1]{heutsCategorifiedGoodwillie},  there is a biadjunction
    \begin{equation}\label{recollect:crossEffectBiadjunction}
        \begin{tikzcd}
            \func_{\udl{*}}(\sC^{\times r},\D)  \rar[bend left = 30, "\crossEffect_f" description] \rar[bend right = 30, "\crossEffect_f" description]& \func_{\underline{\ast}}(\sC^{\times s},\D) \lar["\Delta^*_f"' description]
        \end{tikzcd}
    \end{equation}
    where the biadjoint of $\Delta_f^*$, denoted as $\crossEffect_f$, is called the \textit{cross--effect}.

Suppose we have a pair of composable surjections $f\colon [s]\twoheadrightarrow [r]$ and $g\colon [r]\twoheadrightarrow [t]$. Since $\Delta_f\circ \Delta_g\simeq \Delta_{gf}$, by the property of the cross--effect as being an adjoint of restriction along the diagonal, we obtain that the cross--effects are functorial in surjections, i.e., 
    \begin{equation}\label{lem:iterabilityCrossEffects}
            \crossEffect_{gf} \simeq \crossEffect_{g}\crossEffect_{f}.
    \end{equation}


Next, we recall how excisive approximations and cross--effects interact. 

\begin{obs}\label{obs:calculus_of_cross-effects_and_excisive_approximations}
    Let $\sC$ be a category with finite colimits and a final object, and $\D$ a pointed differentiable category. Let $n\in\mathbb{N}$,  $\underline{k} =(k_1,\ldots,k_n)$ and $k\coloneqq \sum\underline{k}$. Then,  for $F\in\func(\sC^{\times n},\D)$, by applying \cite[Rmk. 6.1.3.23]{lurieHA} in each of the $n$ variable, we have  an equivalence
    \[\crossEffect_{\underline{k}}P_{\underline{k}}F\simeq P_{{\underline{1}}}\crossEffect_{\underline{k}}F\in \func(\sC^{\times k},\D).\]
\end{obs}

\begin{warning}
    Be warned that \cref{obs:calculus_of_cross-effects_and_excisive_approximations} is a sharp result in that the equivalence $\crossEffect_kP_rF\simeq P_{\underline{r-k+1}}\crossEffect_kF$ only works  when $r=k$ (and of course also in the uninteresting case of $r<k$ when both sides are zero for the case of reduced functors $F$). Indeed, let $k=2$, $r=3$, and let $F\colon \spectra \to \spectra$ be a homogeneous functor of degree~$4$ given by $F(X)=X^{\otimes 4}$. Then $P_3 F \simeq 0$. However, $\crossEffect_2F$ has a direct summand $(X,Y) \mapsto X^{\otimes 2}\otimes Y^{\otimes 2}$. Therefore, $P_{2,2}\crossEffect_2F$ is non-trivial.
\end{warning}

Next, we record the multiplicativity enjoyed by the cross--effects. Fix a small stable biexact symmetric monoidal category $\sC=(\sC, \otimes_{\sC})$ and a presentably symmetric monoidal stable category $\D=(\D,\otimes_{\D})$, i.e., a presentable stable category such that the tensor product $-\otimes_{\D}-$ commutes with colimits in each variable. We will often omit the subscript in the tensor products $\otimes_{\sC}, \otimes_{\D}$ when there is no danger of confusion. 

\begin{cons}\label{construction: day convolution}
We endow the functor category $\func(\sC,\D)$ with a symmetric monoidal structure 
$-\otimes-\colon \func(\sC,\D) \times \func(\sC,\D) \to \func(\sC,\D)$ given by the \textit{Day convolution}, \cite[Rmk.~4.8.1.13]{lurieHA}. Informally, we have
$$(F\otimes G)(X) \simeq \colim_{X_1\otimes_{\sC} X_2 \to X} F(X_1)\otimes_{\D} G(X_2)$$
for functors $F,G\in \func(\sC,\D)$ and $X\in\sC$. 

Note in particular that when $X=0$, the colimit is indexed simply by the diagram $\sC\times \sC$. Thus since the inclusion of the zero object is cofinal, the colimit in the formula above simplifies to give $(F\otimes G)(0)\simeq F(0)\otimes G(0)$.
\end{cons}

\begin{recollect}\label{remark: propetries of day convolution}
We remind a few properties of the Day convolution symmetric monoidal product. 
\begin{enumerate}
    \item \textit{(Functoriality)} Let $f\colon \sC_0\to \sC_1$ be a  symmetric monoidal functor between small  symmetric monoidal categories. Then the left Kan extension
    $f_!\colon \func(\sC_0,\D) \to \func(\sC_1,\D)$
    refines canonically to  a { symmetric monoidal} functor with respect to the Day convolution products.
    \item \textit{(Universal case)} We have an equivalence of presentably monoidal categories
    $\func(\sC,\D) \simeq \func(\sC,\spectra)\otimes \D \in \calg(\presentable^L)$.
    \item \textit{(co-Yoneda map)} The co-Yoneda map
    $$\yoneda\colon \sC\op \to \func(\sC,\spectra), \;\; c \mapsto \susps_+\map_{\sC}(c,-)$$
    refines canonically to a { symmetric monoidal} functor. Here, $\map_{\sC}(-,-)$ is the unpointed mapping space. In particular, the unit of the symmetric monoidal structure on $\func(\sC,\spectra)$ defined by the Day convolution is the representable functor $\susps_+\map_{\sC}(\unit_\sC,-)$, where $\unit_\sC$ is the unit in the symmetric monoidal category $\sC$.
\end{enumerate}
\end{recollect}

Recall that the inclusion $\func_*(\sC,\D)\subseteq \func(\sC,\D)$ of reduced functors admits a left adjoint $(-)\reduce$ called the reduction functor.

\begin{lem}\label{lemma: reduction is a smashing localization}
The reduction functor $(-)\reduce$ is a smashing localisation on $\func(\sC,\D)$ compatible with the Day convolution. 
\end{lem}

\begin{proof}

We show first that the kernel of the functor $(-)\reduce$ is a tensor ideal in $\func(\sC,\D)$, i.e., for any $F, G\in \func(\sC,\D)$ such that $F\reduce\simeq 0$, we have $(F\otimes G)\reduce\simeq 0$, cf.~\cite[Lemma~5.3.4]{nineAuthorsI}.
By \cref{remark: propetries of day convolution}, we can assume that $\D=\spectra$ is the category of spectra. Then any functor $G\in \func(\sC,\spectra)$ is a colimit of representable ones and $F\reduce \simeq 0$ if and only if the natural morphism $F(0)\to F(X)$ is an equivalence for any $X\in \sC$, i.e., $F$ is a constant functor. If $F$ is a constant functor, then $F$ is equivalent to the tensor product $Y\otimes \yoneda(0)$ for some $Y\in \spectra$. Therefore, it is enough to show that $\yoneda(0)\otimes \yoneda(c)\simeq \yoneda(0)$ for any $c\in \sC$. The latter is clear since the co-Yoneda embedding is symmetric monoidal, see \cref{remark: propetries of day convolution}. 

We showed that the localisation $(-)\reduce$ is compatible with the Day convolution. Next, we will show that the  transformation (from the lax symmetric monoidality of the inclusion $\func_*(\sC,\D)\subseteq \func(\sC,\D)$)
$$\eta\colon F\otimes \yoneda(\unit_\sC)\reduce \to F\reduce $$
is an equivalence for any $F\in \func(\sC,\spectra)$, i.e., $(-)\reduce$ is a smashing localisation. Note that the induced transformation $\eta\reduce$ between reduced functors is always an equivalence because $(-)\reduce$ is a symmetric monoidal functor. But by \cref{construction: day convolution}, reduced functors form a tensor ideal, and so in fact $F\otimes\yoneda(\unit_{\sC})\reduce$ is already reduced. Hence, $\eta\simeq \eta\reduce$ is an equivalence.
\end{proof}

By currying the variables, this shows that the multi-reduction functor also refines to a smashing localisation.

\begin{cor}\label{corollary: multireduction is a smashing localization}
The multi-reduction functor $$(-)\reduce\colon \func(\sC^{\times r},\D) \to \func_{\udl{*}}(\sC^{\times r},\D)$$ is a smashing localisation on $\func(\sC^{\times r},\D)$ compatible with the Day convolution. \qed
\end{cor}

\begin{cor}\label{corollary: cross-effect is symmetric monoidal}
Let $f\colon [k] \twoheadrightarrow [m]$ be a surjective map of finite sets. Then the (co)cross--effect functor $$\crossEffect_{f}\colon \func_{\udl{*}}(\sC^{\times m},\D) \to \func_{\udl{*}}(\sC^{\times k},\D)$$ refines canonically to a  symmetric monoidal functor.
\end{cor}

\begin{proof}
By the definition, the (co)cross--effect $\crossEffect_f$ is the composite of three functors
$$\crossEffect_{f} \colon \func_{\udl{*}}(\sC^{\times m},\D) \hookrightarrow \func(\sC^{\times m},\D)\xrightarrow{\Delta_{f!}} \func(\sC^{\times k},\D) \xrightarrow{(-)\reduce}\func_{\udl{*}}(\sC^{\times k},\D),$$
where $\Delta_{f!}$ is the left Kan extension along the functor $\Delta_f\colon \sC^{\times m} \to \sC^{\times k}$ and $(-)\reduce$ is the multi-reduction functor. By \cref{remark: propetries of day convolution} and \cref{corollary: multireduction is a smashing localization}, the functor $\crossEffect_{f}$ is canonically a lax symmetric monoidal which preserves the tensor product. However, it does not necessarily preserve the monoidal unit, since the inclusion 
$\func_{\udl{*}}(\sC^{\times m},\D) \hookrightarrow \func(\sC^{\times m},\D)$
does not. Therefore, it suffices to show that the natural map
\begin{equation}\label{equation: monoidal unit multireduction}
    \unit_{\func_{\udl{*}}(\sC^{\times k},\D)} \to \crossEffect_f(\unit_{{\func_{\udl{*}}(\sC^{\times m},\D)}})
\end{equation}
is an equivalence. 


Consider the fibre sequence
$$F \to \unit_{\func(\sC^{\times m},\D)} \to (\unit_{\func(\sC^{\times m},\D)})\reduce\simeq \unit_{\func_{\udl{\ast}}(\sC^{\times m},\D)}$$
in the category $\func(\sC^{\times m},\D)$. In particular, $F\reduce \simeq 0$. Note that the map ~\eqref{equation: monoidal unit multireduction} is obtained by applying  $(\Delta_{f!}-)\reduce\simeq \crossEffect_f(-)$ to the map $\unit_{\func(\sC^{\times m},\D)}\rightarrow \unit_{\func_{\udl{\ast}}(\sC^{\times m},\D)}$ because the functors $\Delta_{f!}$ and $(-)\reduce$ are symmetric monoidal. Consequently, the fibre of \cref{equation: monoidal unit multireduction} is $(\Delta_{f!}F)\reduce$. Since $f\colon [k] \twoheadrightarrow [m]$ is surjective, the functor 
$$\Delta_f^*\colon \func(\sC^{\times k},\D)\to \func(\sC^{\times m},\D)$$
preserves the full subcategory spanned by reduced functors. Therefore, for any reduced functor $G\in \func_{\udl{\ast}}(\sC^{\times},\D)$, we have
$$\map((\Delta_{f!}F)\reduce,G)\simeq \map(\Delta_{f!}F,G) \simeq \map(F,\Delta_f^*G)\simeq \map(F\reduce,G)\simeq 0.$$
By the Yoneda lemma, $(\Delta_{f!}F)\reduce$ is trivial.  
\end{proof}

Finally, let us record the interaction of the Day convolution with excisive approximations.

\begin{lem}\label{lemma: day with excisive is excisive}
    Let $\sC\in\calg(\cat\exact)$ and $\D\in\calg(\presentable^L_{\mathrm{st}})$. Then the Day convolution $F\otimes G$ is $d$-excisive if $F,G \in \func(\sC,\D)$ and $G$ is $d$-excisive.
\end{lem}

\begin{proof}
    Consider the endofunctor $T_d\colon \func(\sC,\D) \to \func(\sC,\D)$ given by
    \begin{align*}
    F&\mapsto \lim_{\emptyset\neq S \subset \{0,\ldots, d\}} F(C_S(-))
    \end{align*}
    from~\cite[Cons.~6.1.1.12]{lurieHA}, where $C_S\colon \sC \to \sC$ is the $S$-pointed cone, see~\cite[Cons.~6.1.1.18]{lurieHA}. Since $G$ is $d$-excisive, the natural map $G\to T_dG$ is an equivalence and it suffices to show that
    $$F\otimes G \to T_d(F\otimes G) $$
    is an equivalence as well, see \cite[Cons.~6.1.1.27]{lurieHA} for the construction of the $d$-excisive approximation.

    Since the pointed cone $C_S\colon \sC \to \sC$ is a left Kan extension, there are natural equivalences 
    $X\otimes C_S(Y) \xrightarrow{\simeq} C_S(X\otimes Y) $    for all $X,Y\in\sC$. Therefore, the functor $C_S$ is equivalent to the tensor product with $C_S(\unit_\sC)$ and since $C_S$ is a finite colimit, $C_S(\unit_\sC)$ is a dualisable object. Hence, since the Day convolution is constructed as a left Kan extension, 
    the functor $C_S^*(F\otimes G)$ is a left Kan extension along the composite 
    $$\sC \times \sC \xrightarrow{-\otimes -} \sC \xrightarrow{C_S(\unit_{\sC})^\vee\otimes -} \sC $$
    which is naturally equivalent to
    $$\sC \times \sC \xrightarrow{\id \times (C_S(\unit_{\sC})^\vee\otimes -)} \sC \times \sC \xrightarrow{-\otimes -} \sC.$$
    This implies that the diagram $S\mapsto (F\otimes G)(C_S(-))$ is equivalent to the diagram $S \mapsto F \otimes G(C_S(-))$. Finally, since the Day convolution commutes with finite products, we have
    \begin{align*}
        T_d(F\otimes G) &\simeq \lim_{\emptyset\neq S \subset \{0,\ldots, d\}} (F\otimes G)(C_S(-)) \simeq \lim_{\emptyset\neq S \subset \{0,\ldots, d\}} F \otimes G(C_S(-)) \\
        &\simeq F\otimes \lim_{\emptyset\neq S \subset \{0,\ldots, d\}} G(C_S(-)) \simeq F \otimes T_d(G) \simeq F \otimes G
    \end{align*} as was to be shown.
\end{proof}

Let $\sC\in\calg(\cat\exact)$, $\D\in\calg(\presentable^L_{\mathrm{st}})$, and let $\udl{d}=(d_1,\ldots,d_r)$ be a multi-index.  We write $\excisive_{\udl{d}}(\sC^{\times r},\D)$ for the full subcategory of $\func(\sC^{\times r},\D)$ spanned by the functors which are reduced $d_i$--excisive in the $i$-th variable.

\begin{cor}\label{corollary: excisive approximation is smashing}
    The category $\excisive_{\udl{d}}(\sC^{\times r},\D)$ admits a presentably symmetric monoidal structure rendering a refinement of the left adjoint $\excisiveApproximation_{\udl{d}} \colon \func(\sC^{\times r},\D)\rightarrow \excisive_{\udl{d}}(\sC^{\times r},\D)$ to a smashing localisation.
\end{cor}

\begin{proof}
    Without loss of generality, we can assume that $r=1$. Then $P_d$ is a localisation compatible with the Day convolution by~\cite[Lem. ~5.3.4]{nineAuthorsI}, see also \cite[Prop. ~2.38, Thm. 2.39]{ABHS24}. We only have to show that the transformation
    $$\eta \colon F \otimes P_d(\unit_{\func(\sC^{\times r},\D)}) \to P_d(F)$$
    is an equivalence for any $F\in \func(\sC^{\times r},\D)$, i.e., $P_d$ is a smashing localisation. Note that $P_d(\eta)$ is always an equivalence because $P_d$ is a symmetric monoidal functor. Finally, by \cref{lemma: day with excisive is excisive}, both sides $\eta$ are $d$-excisive, so $\eta\simeq P_d(\eta)$ is an equivalence.
\end{proof}

\subsection{Subdiagonal functors}\label{subsection:subdiagonal_functors}

Fix a small stable category $\sC$, a presentable stable category $\D$, and a number $d$ throughout.

\begin{nota}
    Let $r\leq d$. We will write $\excisive_{d,r}(\sC,\D) $ for the full subcategory of $\func(\sC^{\times r},\D)$ consisting of those functors which are reduced and $(d-r+1)$--excisive in each variable.
    In the special case of $\sC=\spectra^{\omega}, \D=\spectra$, we will  just write $\excisive_{d,r}$.
\end{nota}

The following observation and definition are the key insights which form the starting point for our work.

\begin{nota}\label{obs:canonical_subdivided_cubes}
    Let $r\leq d$ and $m \coloneqq d-r+1$.  We denote by $Q_{d,r}$ for the category $(\Delta^{m-1})^{\times r}$ which we think of as a subdivided cube. For these cubical objects $Q_{d,r}$, it will be convenient to adopt the convention that $\Delta^{m-1}$ be viewed as the linearly ordered category $\{1\rightarrow 2\rightarrow\cdots\rightarrow m\}$. Under this convention, we write $Q^{\lrcorner}_{d,r}\subseteq Q_{d,r}$  for the full subcategory of those tuples $(k_1,\ldots,k_r)$ such that $k_1+\cdots + k_r\leq d$, which we think of as truncated subdivided cubes.
    
    Then, by considering the Goodwillie tower of excisive approximations in each variable, we obtain a functor
    \[P_{\udl{\bullet}}\colon \excisive_{d,r}(\sC,\D)\longrightarrow \func(Q_{d,r},\excisive_{d,r}(\sC,\D))\] where for $F\in\excisive_{d,r}(\sC,\D)$, $P_{\udl{\bullet}}F$ sends $(k_1,\ldots,k_r)\in Q_{d,r}$ to $P_{k_1,\ldots,k_r}F$. 
\end{nota}

\begin{defn}\label{cons:intermediate_categories}
    Let $r\leq d$. We write $\excisive^{\lrcorner}_{d,r}(\sC,\D)\subseteq \excisive_{d,r}(\sC,\D)$ for the full subcategory of the \textit{subdiagonal functors}, i.e., those functors $F$ whose associated $Q_{d,r}$--cube $P_{\udl{\bullet}}F$ is right Kan extended from its restriction to $Q_{d,r}^{\lrcorner}\subseteq Q_{d,r}$.

\end{defn}

\begin{example}
We illustrate  here some examples of  the $Q_{d,r}$--cubes from \cref{obs:canonical_subdivided_cubes} and the condition of subdiagonality. In the following diagrams, the subdiagonal functors are those which are right Kan extended  from the blue portion.
    \begin{enumerate}
        \item For $d=3$, $r=2$, and $m=2$, every object in $\excisive_{3,2}$ fits in the $Q_{3,2}$--cube
        \begin{center}
            \begin{tikzcd}
                P_{2,2} \rar \dar & \mathcolor{blue}{P_{2,1}}\dar[color=blue]\\
                \mathcolor{blue}{P_{1,2}} \rar[color=blue] & \mathcolor{blue}{P_{1,1}}.
            \end{tikzcd}
        \end{center}
        \item For $d=4$, $r=2$, and  $m=3$, every object in $\excisive_{4,2}$ fits in the $Q_{4,2}$--cube
        \begin{center}
            \begin{tikzcd}
                P_{3,3} \rar\dar & P_{3,2} \rar \dar& \mathcolor{blue}{P_{3,1}}\dar[color=blue]\\
                P_{2,3} \rar\dar & \mathcolor{blue}{P_{2,2}} \rar[color=blue] \dar[color=blue]& \mathcolor{blue}{P_{2,1}}\dar[color=blue]\\
                \mathcolor{blue}{P_{1,3}} \rar[color=blue] & \mathcolor{blue}{P_{1,2}} \rar[color=blue] & \mathcolor{blue}{P_{1,1}}.
            \end{tikzcd}
        \end{center}
        \item For $d=4$, $r=3$, and  $m=2$, every object in $\excisive_{4,3}$ fits in the $Q_{4,3}$--cube
        \begin{center}
            \begin{tikzcd}
                P_{2,2,2}\ar[rr]\ar[dd]\ar[dr] && P_{2,2,1}\ar[dd]\ar[dr]\\
                & P_{2,1,2}\ar[rr, crossing over]&& \mathcolor{blue}{P_{2,1,1}}\ar[dd,color=blue]\\
                P_{1,2,2}\ar[rr]\ar[rd] && \mathcolor{blue}{P_{1,2,1}}\ar[rd,color=blue]\\
                & \mathcolor{blue}{P_{1,1,2}} \ar[rr,color=blue] \ar[uu, leftarrow, crossing over] && \mathcolor{blue}{P_{1,1,1}}.
            \end{tikzcd}
        \end{center}
    \end{enumerate}
\end{example}

\begin{obs}\label{obs:elbow_categories_are_stable_subcats}
    Note that $\excisive_{d,r}^{\lrcorner}(\sC,\D)\subseteq \excisive_{d,r}(\sC,\D)\subseteq\func_{\udl{\ast}}(\sC^{\times r},\D)$ are stable subcategories. This is because we are in the stable context, and so $P_{\underline{k}}$ preserves finite (co)limits. In fact, $\elbowCat{d,r}(\sC,\D)$ is  also closed under arbitrary limits in $\excisive_{d,r}(\sC,\D)\subseteq\func_{\udl{\ast}}(\sC^{\times r},\D)$ as we shall see in \cref{prop:reflection_from_all_excisive_into_elbow_excisive}: this is not immediate a priori since subdiagonality is defined in terms of the functors $P_{\udl{k}}$, which do not commute with infinite limits.
\end{obs}

We are now ready to state  the fundamental insight about  subdiagonlity, i.e., \cref{alphThm:key_pigeonhole_principle},  which affords our approach to obtaining the  Mackey description in \cref{alphThm:excisive_functors_as_spectral_mackey}.

\begin{thm}[Key pigeonhole principle]\label{thm:key_pigeonhole_principle}
    Let $\sC\in\cat\exact$ and $\D\in\presentable^L_{\mathrm{st}}$. Let $r\leq s\leq d$ and $f\colon [s]\twoheadrightarrow [r]$ be a surjection. We have the dashed factorisations
    \[
        \begin{tikzcd}
                \excisive_{d,r}^{\lrcorner}(\sC,\D)\dar[hook]\rar["{\crossEffect_{f}}",dashed]\ar[dr,phantom,"\mathrm{(1)}"]& \excisive_{d,s}^{\lrcorner}(\sC,\D)\dar[hook]\\
                \excisive_{d,r}(\sC,\D)\rar["{\crossEffect_{f}}"']& \func(\sC^{\times s},\D)
        \end{tikzcd}
        \hspace{1cm}
        \begin{tikzcd}
                \excisive_{d,r}^{\lrcorner}(\sC,\D)\dar[hook]\ar[dr,phantom,"\mathrm{(2)}"]& \excisive_{d,s}^{\lrcorner}(\sC,\D)\dar[hook]\lar["\Delta_f^*"',dashed]\\
                \func(\sC^{\times r},\D)& \excisive_{d,s}(\sC,\D). \lar["\Delta_f^*"]
        \end{tikzcd}
    \]
\end{thm}

\begin{rmk}
    As pointed out in the introduction, neither dashed lifts are obvious. Without the subdiagonality condition, it is not true that $\crossEffect_f$ on $\excisive_{d,r}(\sC,\D)$ lands in $\excisive_{d,s}(\sC,\D)$, much less $\elbowCat{d,s}(\sC,\D)$. Similarly in the case of $\Delta_f^*$. The subdiagonality condition is what guarantees that the images of these functors land in multivariable functors that have a uniform excisiveness bounds  across the variables.
\end{rmk}

As an immediate consequence of \cref{thm:key_pigeonhole_principle} and the biadjunction \cref{recollect:crossEffectBiadjunction}, we obtain the following:

\begin{cor}\label{cor:omnibus_multivariable_cross-effects}
        Let $r\leq s\leq d$ and $f\colon [s]\twoheadrightarrow [r]$ be a surjection.
        Then the biadjunction in the top row of 
            \begin{center}
        \begin{tikzcd}
            \func_{\udl{\ast}}(\sC^{\times r},\D)  \ar[rr,bend left = 20, "\crossEffect_{f}" description] \ar[rr,bend right = 20, "\crossEffect_{f}" description]&& \func_{\underline{\ast}}(\sC^{\times s},\D) \ar[ll, "\Delta_f^*"' description]\\
            \\
            \excisive^{\lrcorner}_{d,r}(\sC,\D) \ar[uu,hook] \ar[rr,bend left = 20, "\crossEffect_{f}" description,dashed] \ar[rr,bend right = 20, "\crossEffect_{f}" description,dashed]&& \excisive^{\lrcorner}_{d,s}(\sC,\D) \ar[uu,hook] \ar[ll,"\Delta_f^*"' description,dashed]
        \end{tikzcd}
    \end{center}
    restricts to a dashed biadjunction in the bottom row.
\end{cor}

The remainder of this subsection will be occupied with the proof of the theorem. To this end, we will need the following two limit cofinality preliminaries.

\begin{cons}\label{cons:filtered_cube}
    Let $k_1,\ldots,k_r$ be some positive integers, $k\coloneqq \sum_ik_i$, and let $\sum\colon \Delta^{k_1}\times\cdots\times \Delta^{k_r}\rightarrow \Delta^k$ be the functor which assigns to $(a_1,\ldots,a_r)\in\Delta^{k_1}\times\cdots\times \Delta^{k_r}$ the object $\sum_ia_i\in\Delta^k$ (recall also our convention for $\Delta^k$ from \cref{obs:canonical_subdivided_cubes}). Write $q\colon B_n\hookrightarrow \Delta^{k_1}\times\cdots\times \Delta^{k_r}$ for the pullback
    \begin{center}
        \begin{tikzcd}
            B_n \rar[hook, "q"] \dar\ar[dr,phantom,  "\lrcorner"]& \Delta^{k_1}\times\cdots\times \Delta^{k_r}\dar["\sum"]\\
            \langle s \in \Delta^r\: | \: s\geq n\rangle\rar[hook] & \Delta^k.
        \end{tikzcd}
    \end{center}
    That is, $B_n$ is the full subcategory spanned by those tuples $(a_i)_i\in\Delta^{k_1}\times\cdots\times \Delta^{k_r}$ such that $\sum_ia_i\geq n$. 
\end{cons}

\begin{lem}\label{lem:pointwise_cubical_induction}
    Let $\sC$ be a category with finite limits and $B_n$ be as in \cref{cons:filtered_cube}. A functor $F\in\func(\Delta^{k_1}\times\cdots\times \Delta^{k_r},\sC)$ is in the image of the right Kan extension $\func(B_n,\sC)\xhookrightarrow{q_*}\func(\Delta^{k_1}\times\cdots\times \Delta^{k_r},\sC)$ if and only if for all tuples $\udl{a}\in \Delta^{k_1}\times\cdots\times \Delta^{k_r}\backslash B_n$, writing $\udl{\delta}_i = (0,\ldots,0,1,0,\ldots,0)$ for the tuple with $1$ in component $i$ and zero elsewhere, the $r$--cube
    \[\Big\{F(\udl{a}+\udl{\delta}_{i_1}+\cdots+\udl{\delta}_{i_t})\Big\}_{(i_1,\ldots,i_t)\in \mathrm{Subsets \: of\:} [r]}\]
    is a limit diagram.
\end{lem}
\begin{proof}
    Note that since we can factorise the inclusion $q\colon B_n\subseteq \Delta^{k_1}\times\cdots\times \Delta^{k_r}$ as the sequence of fully faithful functors
    \[B_n \hookrightarrow B_{n-1} \hookrightarrow \cdots \hookrightarrow B_1 \hookrightarrow B_0 = \Delta^{k_1}\times\cdots\times \Delta^{k_r},\] it suffices by induction to prove the statement for the inclusion $B_n \hookrightarrow B_{n-1}$. By the pointwise right Kan extension formula, a functor $F \colon B_{n-1}\rightarrow\sC$ is right Kan extended along $B_n\subseteq B_{n-1}$ if and only if for all $\udl{a}\in B_{n-1}\backslash B_n$, i.e., $\sum \udl{a}=n-1$, the  canonical map $F(\udl{a}) \rightarrow \lim_{(\udl{a}\downarrow B_n)} F$ is an equivalence.
    But, since everything in sight is a poset,  the comma category $(\udl{a}\downarrow B_n)$ is equivalent to the punctured subdivided cube $(\prod_i\Delta^{l_i})\backslash \udl{1}$ where $l_i\coloneqq k_i-a_i$. 
    
    The main point now is that the inclusion $c\colon (\Delta^1)^{\times r}=\prod_i\Delta^{l_i}_{\leq 2}\subseteq\prod_i\Delta^{l_i}$ admits a right adjoint since for any number $k$, the inclusion $\Delta^1 = \{1\rightarrow 2\} \subseteq \Delta^k$ admits a right adjoint given by sending all numbers in $\Delta^k=\{1\rightarrow 2\rightarrow \cdots \rightarrow k+1\}$ greater than $1$ to $2$. Hence, the inclusion  of the punctured subcube $(\Delta^1)^{\times r}\backslash\udl{1}\subseteq \prod_i\Delta^{l_i}\backslash\udl{1}$ is limit cofinal, and so $\lim_{(\udl{a}\downarrow B_n)} F$ may equivalently be computed as the limit of the punctured $r$--subcube. This completes the proof.
\end{proof}

The following simple observation will play a basic role throughout the article as it allows us to reduce many proofs about the category $\elbowCat{d,r}$ to checking on a smaller class of objects which are amenable to pigeonhole arguments.

\begin{lem}\label{lem:generators_of_truncated_subdivided_cubes}
    Let $F\in\excisive_{d,r}(\sC,\D)$. If it is $(k_1,\ldots,k_r)$--excisive with $k_i\leq d-r+1$ for all $i$ such that $k_1+\cdots+k_r\leq d$, then in fact  $F\in\excisive_{d,r}^{\lrcorner}(\sC,\D)$.
\end{lem}
\begin{proof}
    By \cref{lem:pointwise_cubical_induction}, it suffices to argue that for every tuple $(\ell_1,\ldots,\ell_r)$ such that $\ell_1+\cdots+\ell_r>d$, writing $\udl{\delta}_i = (0,\ldots,0,1,0,\ldots,0)$ for the tuple with $1$ in component $i$ and zero elsewhere, the $r$--cube
    \[\Big\{P_{\underline{\ell}-\udl{\delta}_{i_1}-\cdots-\udl{\delta}_{i_a}}F\Big\}_{(i_1,\ldots,i_a)\in \mathrm{Subsets \: of\:} [r]}\]
    is a limit diagram. Our goal is to argue that this cube is degenerate in one direction and thus  conclude that it is cartesian. For this, we simply observe that by the pigeonhole principle, there exists some $j$ such that $k_j \leq \ell_j-1$, since otherwise $k_1+\cdots+k_r\geq \ell_1+\cdots+\ell_r > d$ contradicting our hypothesis. Therefore, all the maps in the $j$--th direction in the cube $\Big\{P_{\underline{\ell}-\udl{\delta}_{i_1}-\cdots-\udl{\delta}_{i_a}}F\Big\}_{(i_1,\ldots,i_a)\in \mathrm{Subsets \: of\:} [r]}$ will be equivalences, and so the cube is degenerate as wanted.
\end{proof}

We are now ready to give the proof of the theorem.

\begin{proof}[Proof of \cref{thm:key_pigeonhole_principle}.]
    Any surjection may be factorised into a sequence of surjections which is injective everywhere except at two points. Thus, by the functoriality of cross--effects from \cref{lem:iterabilityCrossEffects}, to prove both of the factorisation squares, it suffices to prove it for $s=r+1$ and the surjection $f\colon [s]\twoheadrightarrow [r]$ which collapses the first two points of $[s]$.  For this case, it will be helpful to adopt the notation $m\coloneqq d-r+1$ as well as $\crossEffect_{2,\udl{1}}\coloneqq\crossEffect_f$ and $\Delta_{2,\udl{1}}\coloneqq\Delta_f$ to specify to which variables we are applying the cross--effects and diagonal restriction.
    
    We prove each part in turn. For the square~(1), let $F\in\excisive_{d,r}^{\lrcorner}(\sC,\D)$.  We need to show two things, namely that: (a) $\crossEffect_{2,\underline{1}}F$ is $(m-1)$--excisive in each variable; (b) $\crossEffect_{2,\underline{1}}F$ is right Kan extended from the truncated subdivided cube in \cref{obs:canonical_subdivided_cubes}.

    For point (a), using that cross--effects detect excisiveness by \cite[Prop. 6.1.4.10]{lurieHA}, we need to show, without loss of generality, that $\crossEffect_{2,m,\underline{1}}F\simeq\crossEffect_{1,1,m,\underline{1}}\crossEffect_{2,\underline{1}}F\simeq 0$ where the first equivalence is by \cref{lem:iterabilityCrossEffects}. This will imply that the $m$--excisiveness in the second variable is actually $(m-1)$--excisive. But since $F\in \excisive^{\lrcorner}_{d,r}(\sC,\D)$, it suffices to show that
    \[\crossEffect_{2,m,\underline{1}}P_{\underline{k}}F\simeq 0\] for $\underline{k}=(k_1,\ldots,k_r) \leq \underline{m}$ satisfying $\sum\underline{k}\leq d$ since $F$ is a finite limit over these $P_{\underline{k}}F$'s. Without loss of generality, we might as well assume that $F$ is $\underline{k}$--excisive. This will involve a case analysis: we know first of all that $k_i\leq m$ for all $i$. Suppose $k_2<m$. Then of course $\crossEffect_{2,m,\underline{1}}F\simeq 0$. Now suppose $k_2=m$. Then
    \[k_1+k_3+\cdots+k_r \leq d-m = r-1.\] If one of the $k_i's$ were $0$, then $F$ was already the zero functor since it is reduced in each variable. If not, then since we must have $r-1\leq k_1+k_3+\cdots+k_r\leq r-1$, we thus have $k_1=1$, in which case we also get $\crossEffect_{2,m,\underline{1}}F\simeq 0$. This completes the proof of (a). 

    For point (b), by \cref{lem:pointwise_cubical_induction}, 
    it suffices to show the following claim: for $\underline{k}=(k_1,\ldots,k_{r+1})$ such that $\sum\underline{k}> d$, and writing $\udl{\delta}_i = (0,\ldots,0,1,0,\ldots,0)$ for the tuple with $1$ in component $i$ and zero elsewhere, the $(r+1)$--cube
    \[\Big\{P_{\underline{k}-\udl{\delta}_{i_1}-\cdots-\udl{\delta}_{i_a}}\crossEffect_{2,\underline{1}}F\Big\}_{(i_1,\ldots,i_a)\in \mathrm{Subsets \: of\:} [r+1]}\]
    is a limit. Furthermore, by virtue  of \cref{obs:elbow_categories_are_stable_subcats} and \cref{lem:generators_of_truncated_subdivided_cubes},  $F\in\excisive_{d,r}^{\lrcorner}(\sC,\D)$ is built as a finite limit out of the $\underline{m}$--excisive functors which are $(\ell_1,\ldots,\ell_r)$--excisive where $\ell_1+\cdots+\ell_r\leq d$. Hence, without loss of generality, we may suppose $F$ was $(\ell_1,\ldots,\ell_r)$--excisive where $\ell_1+\cdots+\ell_r\leq d$. Now to see the claim, it suffices to show that the total fibre of the cube is zero, and for this cube the total fibre is computed as the homogeneous layers in each variable. And so to show it vanishes, it suffices to show that
    \[\crossEffect_{\underline{k}}P_{\underline{k}}\crossEffect_{2,\underline{1}}F\simeq 0.\]
    But using \cref{obs:calculus_of_cross-effects_and_excisive_approximations} and \cref{lem:iterabilityCrossEffects}, we see
    \[\crossEffect_{\underline{k}}P_{\underline{k}}\crossEffect_{2,\underline{1}}F\simeq P_{{\underline{1}}}\crossEffect_{\underline{k}}\crossEffect_{2,\underline{1}}F\simeq P_{{\underline{1}}}\crossEffect_{k_1+k_2, k_3,\ldots,k_{r+1}}F.\]
    Since $F$ was $(\ell_1,\ldots,\ell_r)$--excisive, we see that $\crossEffect_{k_1+k_2, k_3,\ldots,k_{r+1}}F$ is $(\underline{\ell_1-(k_1+k_2)+1}, \underline{\ell_2-k_3+1},\ldots,\underline{\ell_r-k_{r+1}+1})$--excisive by \cite[Prop. 6.1.3.22]{lurieHA}. Since we have both $\ell_1+\cdots+\ell_r\leq d$ and $k_1+\cdots +k_{r+1}>d$, we must have that $\ell_1-(k_1+k_2) +1< 1$ or $\ell_j-k_{j+1} +1< 1$ for some $j\geq 2$, since otherwise we would have $\sum\underline{\ell}\geq \sum\underline{k} > d$ contradicting our hypothesis. Therefore, one of the variables of $\crossEffect_{k_1+k_2, k_3,\ldots,k_{r+1}}F$ is $0$--excisive, which implies that $\crossEffect_{k_1+k_2, k_3,\ldots,k_{r+1}}F\simeq 0 $ since all functors in sight are reduced in each variable. This proves the claim, and so also point (b).

    Now for the square (2), again as in the proof of point (1)(b) above, since $\excisive_{d,r}^{\lrcorner}(\sC,\D)\subseteq \func_{\udl{\ast}}(\sC^{\times r},\D)$ is closed under finite limits by \cref{obs:elbow_categories_are_stable_subcats}, it suffices to show that $\Delta_{2,\udl{1}}^*F\in\excisive_{d,r}^{\lrcorner}(\sC,\D)$ for $F\in\excisive_{d,r+1}^{\lrcorner}(\sC,\D)$ such that $F$ is $(\ell_1,\ldots, \ell_{r+1})$--excisive with $\sum\underline{\ell}\leq d$. By \cite[Cor. 6.1.3.5]{lurieHA}, we know that $\Delta_{2,\udl{1}}^*F$ is $(\ell_1+\ell_2,\ell_3,\ldots, \ell_{r+1})$--excisive. 
    
    We claim that if $F$ were not the zero functor, then $\ell_1+\ell_2 \leq m$. To see this, suppose $\ell_1+\ell_2> m$. Then $d\geq \sum\underline{\ell} = (\ell_1+\ell_2)+ (\ell_3+\cdots+\ell_{r+1}) > m +(\ell_3+\cdots+\ell_{r+1})$. Hence, since by definition $m=d-r+1$, we obtain $r-1 = d-m > (\ell_3+\cdots+\ell_{r+1})$. But then, there are $r-1$ terms in the set $\{\ell_3,\ldots,\ell_{r+1}\}$, hence we must have $\ell_i =0$ for some $i$ by pigeonhole. Since $F$ is reduced in each variable, this means that $F\simeq 0$ as claimed. 
    
    Therefore, by virtue of the claim we know at least that $\Delta_{2,\udl{1}}^*F\in\excisive_{d,r}(\sC,\D)$. On the other hand, since $(\ell_1+\ell_2)+\ell_3+\cdots+\ell_{r+1} \leq d$, we see that $\Delta_{2,\udl{1}}^*F$ is even in $\excisive_{d,r}^{\lrcorner}$ by \cref{lem:generators_of_truncated_subdivided_cubes}. This completes the proof of part (2) and  the theorem.
\end{proof}

\subsection{Beck--Chevalley squares}\label{subsection:beckChevalley_squares}
Our proof of the main theorem via parametrised higher category theory will leverage on an understanding of the interaction between the functors $\crossEffect_f$ and $\Delta^*_f$. This interaction, which is of a ``double--coset'' flavour, will be distilled as \cref{prop:doubcle_coset_truncated_to_d}, and the goal of this subsection is to prove this. To this end, we will need some preliminaries in order to be able even to state the result. Let $\sC$ and $\D$ be stable categories throughout.

The following considerations are dictated by the theory of parametrised categorie (see \cref{corollary: Exc admits indexed colimits} and \cref{thm:basic_properties_of_parametrised_excisive_category}). However, since we defer all explanations of this theory to \cref{subsection:recollection_parametrised} and \cref{subsection:subdiagonal_parametrised}, we shall not dwell on motivating the following construction and shall only mention what is needed to prove the desired result. 

\begin{cons}\label{construction: back-chevalley local systems}
    Consider a commuting square of finite sets
    \begin{equation}\label{equation: back-chevalley finite sets}
    \begin{tikzcd}
        U \rar["f_U"] \dar["g_U"] 
        & B\dar["g"] \\
        A\rar["f"]
        & E.
    \end{tikzcd}
    \end{equation}
    This square induces two more commutative diagrams
    \[
    \begin{tikzcd}
        \sC^E \rar["\Delta_f"] \dar["\Delta_g"']
        & \sC^A\dar["\Delta_{g_u}"] \\
        \sC^B\rar["\Delta_{f_u}"]
        & \sC^U,
    \end{tikzcd}
    \hspace{1cm}
    \begin{tikzcd}
    \func(\sC^{U},\D) \rar[ "\Delta_{ f_u}^*"]\dar[ "\Delta_{ g_u}^*"']& \func(\sC^{B},\D)\dar["\Delta_g^*"]\\
    \func(\sC^{A},\D) \rar["\Delta_f^*"] & \func(\sC^{E},\D).
    \end{tikzcd}
    \]
    By taking the right adjoints $ \Delta_g \dashv \Pi_g$ and $\Delta_{g_u} \dashv \Pi_{g_u}$ of the vertical arrows in the left square above, we obtain via the Beck--Chevalley transformation the following lax squares
    \[
    \begin{tikzcd}
        \sC^E \rar["\Delta_f"]  \ar[dr, Rightarrow]
        & \sC^A \\
        \sC^B\rar["\Delta_{f_u}"] \uar["\Pi_g"]
        & \sC^U \uar["\Pi_{g_u}"'],
    \end{tikzcd}
    \hspace{1cm}
    \begin{tikzcd}
    \func(\sC^{U},\D) \rar[ "\Delta_{ f_u}^*"] 
    & \func(\sC^{B},\D)\\
    \func(\sC^{A},\D) \rar["\Delta_f^*"] \uar[ "\Pi_{g_u}^*"]
    & \func(\sC^{E},\D)\uar["\Pi_g^*"'] \ar[ul, Rightarrow].
    \end{tikzcd}
    \]
    We note that $\Pi^*_g$ and $\Pi^*_{g_u}$ are \textit{left} adjoints to $\Delta^*_{g}$ and $\Delta^*_{g_u}$, respectively, i.e., $\Pi^*_g\simeq\Delta_{g!}$ and $\Pi^{*}_{g_u}\simeq \Delta_{g_u!}$. Moreover, the natural transformations
    \begin{equation*} 
    \Delta_f \circ \Pi_g \to \Pi_{g_u} \circ \Delta_{f_u},
    \end{equation*}
    \begin{equation} \label{equation: beck-chevalley non-reduced functors}
    \eta_{U}\colon \Pi^*_g \circ \Delta^*_f\to \Delta^*_{f_u} \circ \Pi^*_{g_u}
    \end{equation}
    are equivalences if~\eqref{equation: back-chevalley finite sets} is a pullback diagram, see e.g. \cite[Prop.~4.3.3]{hopkinslurie}.
\end{cons}

\begin{rmk}\label{remark: beck-chevalley is natural}
    The natural transformations~\eqref{equation: beck-chevalley non-reduced functors} are also functorial in $U$. Indeed, let $i\colon V\to U$ be a map of finite sets over $A\times_E B$. Then the following diagram
    \[
    \begin{tikzcd}
        \Pi^*_g\circ \Delta^*_f \rar["\eta_U"] \dar["\eta_V"']
        & \Delta_{f_u}^* \circ \Pi^*_{f_u} \dar \\
        \Delta_{f_v}^* \circ \Pi^*_{f_v} \rar["\simeq"]
        & \Delta_{f_u}^* \circ \Delta_i^* \circ \Pi^*_i \circ  \Pi^*_{f_u}.
    \end{tikzcd}
    \]
    commutes, where the right vertical morphism is induced by the unit of the adjoint pair $\Pi^*_i \dashv \Delta^*_i$. In particular, if $U=A\times_E B$, then the natural transformation $\eta_V$ is induced by the unit of the adjoint pair $\Pi^*_i \dashv \Delta^*_i$.
\end{rmk}

\begin{nota}
    We write $\epiCategory$ for the category of nonempty finite sets and surjections and $\epiCategory_d$ for the full subcategory of finite sets of cardinality at most $d$. We also write $\finite_{\epiCategory}$ and $\finite_{\epiCategory_d}$ for their finite coproduct cocompletion, respectively.
\end{nota}

\begin{cons}\label{construction: beck-chevalley for reduced functors}
    It is an elementary check that $\finite_{\epiCategory}$ and $\finite_{\epiCategory_d}$ admit pullbacks, cf.~\cite[Lem. 3.1]{glasmanGoodwillie}. These may be described as follows: let $f\colon A\twoheadrightarrow E$ and $g\colon B\twoheadrightarrow E$ be surjections such that $|A|,|B|\leq d$. Then their pullback in $\finite_{\epiCategory}$ (resp. in $\finite_{\epiCategory_d}$) is computed as the left square in
    \[
    \begin{tikzcd}
    \displaystyle{\coprod_{\substack{U \subseteq {A\times_{E}B}  \\ U \text{ is good}}}}U \rar[ "\sqcup f_U"]\dar[ "\sqcup g_U"']\ar[dr,phantom, "\lrcorner", near start] & B\dar[two heads, "g"]\\
    A \rar[two heads, "f"] & E
    \end{tikzcd}
    \hspace{1cm}
    \begin{tikzcd}
    \displaystyle{\bigoplus_{\substack{U \subseteq {A\times_{E}B}  \\ U \text{ is good}}}}\func_{\udl{\ast}}(\sC^{U},\D) \rar[ "\oplus \Delta_{ f_u}^*"]\dar[ "\oplus\Delta_{ g_u}^*"']& \func_{\udl{\ast}}(\sC^{B},\D)\dar["\Delta_g^*"]\\
    \func_{\udl{\ast}}(\sC^{A},\D) \rar["\Delta_f^*"] & \func_{\udl{\ast}}(\sC^{E},\D)
    \end{tikzcd}
    \]
    where a subset $U\subseteq A\times_{E}B$ is said to be \textit{good} in $\epiCategory$ if it surjects onto $A$ and $B$ under the projections (resp. \textit{good} in $\epiCategory_d$ if additionally, $|U|\leq d$). To distinguish these notions, we will also say $\epiCategory$--goodness and $\epiCategory_d$--goodness, respectively. Applying $\func_{\udl{\ast}}(\sC^{(-)},\D)$ to this square for different $U$ separately and taking their sum results in the  commuting square on the right above.

    By taking the left adjoints of the vertical arrows in the right square above, we obtain via the Beck--Chevalley transformation the lax square
    \begin{equation}\label{eqn:beckChevalleyLaxSquareALL}
    \begin{tikzcd}
    \displaystyle{\bigoplus_{\substack{U \subseteq {A\times_{E}B}  \\ U \text{ is $\epiCategory$-good}}}}\func_{\udl{\ast}}(\sC^{U},\D) \rar[ "\oplus \Delta_{ f_u}^*"]& \func_{\udl{\ast}}(\sC^{B},\D)\\
    \func_{\udl{\ast}}(\sC^{A},\D) \rar["\Delta_f^*"]\uar[ "\oplus\crossEffect_{ g_u}"] & \func_{\udl{\ast}}(\sC^{E},\D). \uar["\crossEffect_g"']\ar[ul, Rightarrow, shorten >=-3.5em]
    \end{tikzcd}
    \end{equation}
    Concretely, this natural transformation is given by 
    \begin{equation}\label{equation: base change}
    \eta=\bigoplus \eta_U\colon \crossEffect_g \Delta^*_f \simeq (\Pi^*_g\Delta^*_f){\reduce}\to \bigoplus_{\text{$U$ is good}} (\Delta^*_{f_u}\Pi^*_{g_u}){\reduce} \to \bigoplus_{\text{$U$ is good}} \Delta_{f_u}^*\crossEffect_{g_u}
    \end{equation}
    where the summands $\eta_U$ are induced by the natural transformations~\eqref{equation: beck-chevalley non-reduced functors}.
\end{cons}

We prove that the lax square \cref{eqn:beckChevalleyLaxSquareALL} commutes in \cref{proposition: double-coset formula}, and the strategy is to analyse the left adjoint $\Pi_f^*\colon \func(\sC^E,\D) \to \func(\sC^A,\D)$.
First, we show that the projection $\Pi_f^* \to \crossEffect_f $ splits as a direct summand and we calculate the fibre.

\begin{lem}\label{lemma: spilliting of the left adjoint}
Let $g\colon B \twoheadrightarrow E$ be a surjective map of finite sets. Then the natural transformation
\[\Pi^*_g F \to  \bigoplus_{\substack{i\colon U \hookrightarrow B  \\ g(U)=E}} \Delta^*_i \crossEffect_i\Pi^*_g F=  \bigoplus_{\substack{i\colon U \hookrightarrow B  \\ g(U)=E}} \Delta^*_i (\Pi^*_i \Pi^*_g F){\reduce}\simeq \bigoplus_{\substack{i\colon U \hookrightarrow B  \\ g(U)=E}} \Delta^*_i \crossEffect_{g\circ i} F\]
is an equivalence for any multireduced functor $F\in \func_{\udl{\ast}}(\sC^E,\D)$.
\end{lem}

\begin{proof}
    We will prove the claim by induction on the difference $|B|-|E|$. If $|B|=|E|$, then $g$ is a bijection and there is nothing to prove. If $|B|=|E|+1$, then we can assume that $B=\{0,\ldots,n\}$, $E=\{1,\ldots,n\}$, and $g(0)=g(1)=0$. Let $(X_0,\ldots,X_n) \in \sC^B$. Then,
    $(\Pi^*_g F)(X_0,X_1,\ldots,X_n) \simeq F(X_0\oplus X_1,\ldots, X_n) $
    and the projection
    $ (\Pi^*_g F)(X_0,X_1,\ldots,X_n) \to \crossEffect_g F(X_0,X_1,\ldots, X_n)$
    splits with the fibre 
    $F(X_0,X_2,\ldots, X_n) \oplus F(X_1,X_2,\ldots, X_n)$,     see e.g. \cite[Lemma~B.1]{heutsCategorifiedGoodwillie}. This proves the base case of the induction.

    Suppose that $|B|\geq |E|+1$. Then there are $b,b' \in B$ such that $g(b)=g(b')$ and we let $q\colon B \to B_0=B/(b\sim b')$ denote the quotient map. Hence the map $g\colon B\twoheadrightarrow E$ factors as     $f\colon B\xrightarrow{q} B_0 \xrightarrow{h} E $     and $|B_0|-|E| < |B| -|E|$. Therefore, by the inductive assumption, we have
    $\Pi^*_h F \simeq \bigoplus_{\substack{j\colon V \hookrightarrow B_0  \\ h(V)=E}} \Delta^*_j \crossEffect_{h\circ j} F.$
    By applying the left adjoint $\Pi_q^*$ to both sides, we obtain
    \begin{align*}
        \Pi^*_{g}F \simeq \Pi^*_{q}\Pi^*_h F \simeq \bigoplus_{\substack{j\colon V \hookrightarrow B_0  \\ h(V)=E}} \Pi^*_q \Delta^*_j \crossEffect_{h\circ j} F \simeq \bigoplus_{\substack{j\colon V \hookrightarrow B_0  \\ h(V)=E}} \Delta^*_{j_v} \Pi^*_{q_v}\crossEffect_{h\circ j} F
    \end{align*}
    where $j_v\colon V\times_{B_0} B \hookrightarrow B $ and $q_v\colon V\times_{B_0} B \to V$ are the projections. By the construction, $|V\times_{B_0} B|-|V|\leq 1$ for any subset $V\subset B_0$. Hence, by induction, we have
    \begin{align*}
        \Pi^*_{g}F &\simeq \bigoplus_{\substack{j\colon V \hookrightarrow B_0  \\ h(V)=E}} \Delta^*_{j_v} \Pi^*_{q_v}\crossEffect_{h\circ j} F \\
        &\simeq \bigoplus_{\substack{j\colon V \hookrightarrow B_0  \\ h(V)=E}} \Delta^*_{j_v}\left(\bigoplus_{\substack{i_u\colon U \hookrightarrow V\times_{B_0} B  \\ q_v(U)=V}}\Delta^*_{i_u}\crossEffect_{q_v\circ i_u}\crossEffect_{h\circ j} F\right) \\
        &\simeq \bigoplus_{\substack{j\colon V \hookrightarrow B_0  \\ h(V)=E}} \bigoplus_{\substack{i_u\colon U \hookrightarrow V\times_{B_0} B  \\ q_v(U)=V}}\Delta^*_{j_v \circ i_u}\crossEffect_{h\circ j \circ q_v \circ i_u} F \\
        &\simeq \bigoplus_{\substack{j\colon V \hookrightarrow B_0 \\ i_u\colon U \hookrightarrow V\times_{B_0} B  \\ h(V)=E \\ q_v(U)=V}}\Delta^*_{j_v \circ i_u}\crossEffect_{g \circ j_v \circ i_u} F \simeq \bigoplus_{\substack{i\colon U \hookrightarrow B \\ g(U)=E}}\Delta^*_{i}\crossEffect_{g \circ i} F.
    \end{align*}
    Here we obtained the last equivalence because the map between the indexing sets for the direct sums is given by $(j,i_u) \mapsto j_v\circ i_u$, which is a bijection.
\end{proof}

Next, we will use the splitting of \cref{lemma: spilliting of the left adjoint} to show that the lax square~\eqref{eqn:beckChevalleyLaxSquareALL} is actually strict.

\begin{prop}[Double--coset formula]\label{proposition: double-coset formula}
    Let 
    \begin{center}
    \begin{tikzcd}
        \displaystyle{\coprod_{\substack{U \subseteq {A\times_{E}B}  \\ U \text{ is $\epiCategory$-good}}}}U \rar[ "\sqcup f_U"]\dar[ "\sqcup g_U"']\ar[dr,phantom, "\lrcorner", near start]
        & B \arrow[two heads, "g"]{d} \\
        A \arrow[two heads, "f"]{r} 
        & E
    \end{tikzcd}
    \end{center}
    be a pullback square in the category $\finite_{\epiCategory}$. Then the natural transformation~\eqref{equation: base change}
    $$
    \eta\colon \crossEffect_g \Delta^*_f F\longrightarrow \bigoplus_{\substack{U \subseteq {A\times_{E}B}  \\ U \text{ is $\epiCategory$-good}}} \Delta_{f_u}^*\crossEffect_{g_u} F
    $$
    is an equivalence for any multireduced functor $\func_{\udl{\ast}}(\sC^A,\D).$
\end{prop}

\begin{proof}
Let $W=A\times_E B$. Then, by the basechange formula, we have
$$\crossEffect_g \Delta^*_f \simeq (\Pi^*_g \Delta_f^*)\reduce \simeq (\Delta^*_{f_w}\Pi^*_{g_w})\reduce. $$
By the splitting of \cref{lemma: spilliting of the left adjoint}, we obtain
\[(\Delta^*_{f_w}\Pi^*_{g_w}F)\reduce \simeq\left(\Delta^*_{f_w}\left(\bigoplus_{\substack{i\colon U \hookrightarrow W  \\ g_w(U)=A}} \Delta^*_i \crossEffect_{g_w\circ i} F\right) \right)\reduce \simeq \bigoplus_{\substack{i\colon U \hookrightarrow W  \\ g_w(U)=A}} (\Delta^*_{f_w\circ i} \crossEffect_{g_w\circ i} F)\reduce. \]

But then, we have the equivalences
$$
(\Delta^*_{f_w\circ i} \crossEffect_{g_w\circ i} F)\reduce \simeq
\begin{cases}
    \Delta^*_{f_w\circ i} \crossEffect_{g_w\circ i} F & \text{ if $f_w\circ i$ is surjective}\\
            0 & \text{ if $f_w \circ i$ is not surjective}.
\end{cases}
$$
where the first equivalence is since $\Delta^*_{f_w\circ i}\crossEffect_{g_w\circ i} F$ is already multireduced, and the second one is since we are implementing the multireduction  $(-)^{\mathrm{red}}$ on a functor which is constant in one of the variables.
Hence, the only possible non-trivial summands in the direct sum above are indexed by $\epiCategory$-good subsets $U$. Therefore,
$$\crossEffect_g \Delta^*_f F \simeq (\Pi^*_g \Delta_f F)\reduce \simeq (\Delta^*_{f_w}\Pi^*_{g_w} F)\reduce \simeq \bigoplus_{\substack{U \subseteq {A\times_{E}B}  \\ U \text{ is $\epiCategory$-good}}} \Delta_{f_u}^*\crossEffect_{g_u} F.$$
By \cref{remark: beck-chevalley is natural}, this equivalence is given by the natural transformation~\eqref{equation: base change}.
\end{proof}

By \cref{thm:key_pigeonhole_principle} and restricting to $\epiCategory_d$--good subsets, we may restrict the lax square~\eqref{eqn:beckChevalleyLaxSquareALL} to obtain the lax square
    \begin{equation}\label{eqn:beckChevalleyLaxSquareTRUNCATED}
    \begin{tikzcd}
    \displaystyle{\bigoplus_{\substack{U \subseteq {A\times_{E}B}  \\ U \text{ is $\epiCategory_d$-good}}}}\elbowCat{d,|U|}(\sC,\D) \rar[ "\oplus \Delta_{ f_u}^*"]& \elbowCat{d,|B|}(\sC,\D)\\
    \elbowCat{d,|A|}(\sC,\D) \rar["\Delta_f^*"]\uar[ "\oplus\crossEffect_{ g_u}"] & \elbowCat{d,|E|}(\sC,\D). \uar["\crossEffect_g"']\ar[ul, Rightarrow, shorten >=-3.5em]
    \end{tikzcd}
    \end{equation}

Our goal now is to prove the following:

\begin{prop}\label{prop:doubcle_coset_truncated_to_d}
The lax square \cref{eqn:beckChevalleyLaxSquareTRUNCATED} commutes.
\end{prop}


\begin{proof}
    Since \cref{eqn:beckChevalleyLaxSquareALL} commutes, for all $F\in\elbowCat{d,|A|}$, the map
    \[\crossEffect_g\Delta_f^*F\longrightarrow \bigoplus_{U\subseteq A\times_EB\:\: \epiCategory\text{-good }}\Delta_{f_u}^*\crossEffect_{g_u}F\] is an equivalence. We claim that for any $U\subseteq A\times_EB$ which is not $\epiCategory_d$--good, i.e., $|U|>d$, we have $\crossEffect_{g_u}F\simeq 0$. Given this, we will thus see that the Beck--Chevalley map in \cref{eqn:beckChevalleyLaxSquareTRUNCATED} is an equivalence as required.

    To see the claim, write $a\coloneqq |A|$. Since $F\in\elbowCat{d,a}$, it is a finite limit of functors $G$ which are $(\ell_1,\ldots,\ell_a)$--excisive such that $\ell_1+\cdots+\ell_a\leq d$. Thus, it suffices to deal with such $G\in\elbowCat{d,a}$. We write $K_i$ for the fibres of $g_u\colon U\twoheadrightarrow A$ at $i\in A$ and $k_i\coloneqq |K_i|$. By pigeonhole, we must have that $k_i > \ell_i$ for some $i\in A$ because we have assumed that $|U|>d$. Thus $\crossEffect_{g_u}G\simeq 0$, as claimed.
\end{proof}

\section{Categorical properties and multiplicative structures}\label{section:categorical_properties_multiplicative}

\subsection{Multiplicative structures}\label{subsection:multiplicative_structures}
Let $d$ be a positive integer, $r\leq d$, and $m\coloneqq d-r+1$. Our goal now is to construct a presentably symmetric monoidal structure on $\elbowCat{d,r}$ and enhance the cross--effect functors with symmetric monoidal structures. For this, we will need the following alternative characterisation of the category $\elbowCat{d,r}$ as a full subcategory of $\excisive_{d,r}$. Recall first the notation $Q_{d,r}^{\lrcorner}$ for the truncated subdivided cubes  from \cref{obs:canonical_subdivided_cubes}.

\begin{prop}\label{prop:reflection_from_all_excisive_into_elbow_excisive}
    Let $\sC\in\cat\exact$ and $\D\in\presentable^L_{\mathrm{st}}$.  The  inclusion $\elbowCat{d,r}(\sC,\D)\subseteq\excisive_{d,r}(\sC,\D)$ preserves  colimits and admits a left adjoint $\excisiveApproximation_{d,r}^{\lrcorner}\colon \excisive_{d,r}(\sC,\D)\rightarrow \elbowCat{d,r}(\sC,\D)$ given by $\lim_{Q_{d,r}^{\lrcorner}}\excisiveApproximation_{\udl{\bullet}}(-)$. Hence, $\elbowCat{d,r}(\sC,\D)$ is a presentable category and the inclusion $\elbowCat{d,r}(\sC,\D)\subseteq \excisive_{d,r}(\sC,\D)$ preserves limits. 
\end{prop}
\begin{proof}
    The inclusion preserves colimits since the excisive approximations $\excisiveApproximation_{\udl{\bullet}}$ are smashing localisations and the subdiagonality limits $\lim_{Q_{d,r}^{\lrcorner}}$ are finite by \cref{cons:intermediate_categories}. Now, for $A\in\excisive_{d,r}(\sC,\D)$, the key observation is that $\excisiveApproximation_{\udl{\bullet}}\excisiveApproximation_{d,r}^{\lrcorner}A\simeq \excisiveApproximation_{d,r}^{\lrcorner}\excisiveApproximation_{\udl{\bullet}}A\simeq \excisiveApproximation_{\udl{\bullet}}A$ by \cref{lem:generators_of_truncated_subdivided_cubes}, where here $\udl{\bullet}$ is given by the tuples appearing in the truncated subdivided cube $Q_{d,r}^{\lrcorner}$. Let $F\in\elbowCat{d,r}(\sC,\D)$ so that $F\rightarrow \lim_{Q_{d,r}^{\lrcorner}}\excisiveApproximation_{\udl{\bullet}}F$ is an equivalence. Then consider the computation:
    \begin{equation*}
        \begin{split}
            \map(A,F) & \simeq \lim_{Q_{d,r}^{\lrcorner}}\map(A,\excisiveApproximation_{\udl{\bullet}}F)\\
            &\simeq \lim_{Q_{d,r}^{\lrcorner}}\map(\excisiveApproximation_{\udl{\bullet}}A,\excisiveApproximation_{\udl{\bullet}}F)\\
            &\simeq \lim_{Q_{d,r}^{\lrcorner}}\map(\excisiveApproximation_{\udl{\bullet}}\excisiveApproximation_{d,r}^{\lrcorner}A,\excisiveApproximation_{\udl{\bullet}}F)\\
            &\simeq \lim_{Q_{d,r}^{\lrcorner}}\map(\excisiveApproximation_{d,r}^{\lrcorner}A,\excisiveApproximation_{\udl{\bullet}}F)\\
            &\simeq \map(\excisiveApproximation_{d,r}^{\lrcorner}A,F).
        \end{split}
    \end{equation*}
    whence the statement about the left adjoint.  The statements about presentability and limit--preservation are now immediate.
\end{proof}



\begin{cor}\label{cor:symmetric_monoidal_structure_on_elbow_cats}
    Let $\sC\in\calg(\cat\exact)$ and $\D\in\calg(\presentable^L_{\mathrm{st}})$. The category $\elbowCat{d,r}(\sC,\D)$ admits a presentably symmetric monoidal structure rendering a refinement of the left adjoint $\excisiveApproximation_{d,r}^{\lrcorner}\colon \excisive_{d,r}(\sC,\D)\rightarrow \elbowCat{d,r}(\sC,\D)$   to a symmetric monoidal functor which is even a smashing localisation.
\end{cor}
\begin{proof}

    We only have to show that there is an equivalence $-\otimes \excisiveApproximation_{d,r}^{\lrcorner}(\unit)\simeq \excisiveApproximation_{d,r}^{\lrcorner}(-)$ of endofunctors on $\excisive_{d,r}(\sC,\D)$. To this end, simply consider the  equivalences
    $$F\otimes \excisiveApproximation_{d,r}^{\lrcorner}(\unit) \simeq F \otimes \lim_{Q_{d,r}^{\lrcorner}}\excisiveApproximation_{\udl{\bullet}}(\unit) \simeq \lim_{Q_{d,r}^{\lrcorner}}F\otimes \excisiveApproximation_{\udl{\bullet}}(\unit) \simeq \lim_{Q_{d,r}^{\lrcorner}}\excisiveApproximation_{\udl{\bullet}}(F) \simeq \excisiveApproximation_{d,r}^{\lrcorner}(F)$$ which are natural in $F\in\excisive_{d,r}(\sC,\D)$, where the second equivalence uses that $Q^{\lrcorner}_{d,r}$ is a finite diagram and that the tensor product commutes with finite limits in both variables, and the third equivalence is since $P_{\udl{\bullet}}$ is a smashing localisation by \cref{corollary: excisive approximation is smashing}.
\end{proof}

\begin{nota}\label{nota:tensor_unit_for_elbow_cats}
    We write $\unit_{d,r}$ for the tensor unit of $\elbowCat{d,r}(\sC,\D)$ under the symmetric monoidal structure constructed in \cref{cor:symmetric_monoidal_structure_on_elbow_cats}. In other words, writing $\unit_r$ for the tensor unit of $\func_{\udl{\ast}}(\sC^{\times r},\D)$, we have $\unit_{d,r}\simeq \excisiveApproximation^{\lrcorner}_{d,r}P_{\udl{m}}\unit_r$ where $m=d-r+1$. 
\end{nota}

Having a presentably symmetric monoidal structure on $\elbowCat{d,r}(\sC,\D)$ in hand, our goal next is to formulate the symmetric monoidal refinement of the cross--effects.

\begin{prop}\label{prop:symmetric_monoidal_refinement_of_elbow_cross_effects}
    Let $d$ be a positive integer, $k\leq r\leq d$, $f\colon [r]\twoheadrightarrow[k]$ a surjective function. Then the symmetric monoidal right vertical map in the commuting diagram from \cref{thm:key_pigeonhole_principle}
    \begin{center}
        \begin{tikzcd}
            \elbowCat{d,k}(\sC,\D) \dar["\crossEffect_f"']\ar[r,hook]& \func_{\udl{\ast}}(\sC^{\times k},\D)\dar["\crossEffect_f"]\\
            \elbowCat{d,r}(\sC,\D) \ar[r,hook]& \func_{\udl{\ast}}(\sC^{\times r},\D).
        \end{tikzcd}
    \end{center}
    restricts to a symmetric monoidal structure on the left map $$\crossEffect_f\colon \elbowCat{d,k}(\sC,\D) \rightarrow \elbowCat{d,r}(\sC,\D).$$
\end{prop}
\begin{proof}
    Let $m\coloneqq d-r+1$, and $n\coloneqq d-k+1$. Consider the solid commuting diagram obtained from \cref{cor:omnibus_multivariable_cross-effects}
    \begin{center}
        \begin{tikzcd}
            \elbowCat{d,k}(\sC,\D) \dar["\crossEffect_f"']\ar[r,hook,shift right = 1.5]&\excisive_{d,k}(\sC,\D)\ar[r,hook,shift right = 1.5]\lar[shift right = 1.5, dashed, "\excisiveApproximation_{{d,k}}^{\lrcorner}"']& \func_{\udl{\ast}}(\sC^{\times k},\D)\dar["\crossEffect_f"]\lar[shift right = 1.5, dashed, "\excisiveApproximation_{\udl{n}}"']\\
            \elbowCat{d,r}(\sC,\D) \ar[r,hook,shift right = 1.5]&\excisive_{d,r}(\sC,\D)\ar[r,hook,shift right = 1.5]\lar[shift right = 1.5, dashed, "\excisiveApproximation_{{d,r}}^{\lrcorner}"']& \func_{\udl{\ast}}(\sC^{\times r},\D).\lar[shift right = 1.5, dashed, "\excisiveApproximation_{\udl{m}}"']
        \end{tikzcd}
    \end{center}
    Here, all the dashed arrows are smashing localisations, where the left horizontal dashed left adjoints were obtained from \cref{cor:symmetric_monoidal_structure_on_elbow_cats}. In particular, the solid horizontal inclusions all obtain an induced lax symmetric monoidal structure which preserve tensor products. By \cref{corollary: cross-effect is symmetric monoidal}, the right vertical $\crossEffect_f$ also has a symmetric monoidal structure. 
    
    Our task is to argue that the left vertical $\crossEffect_f$ has a symmetric monoidal refinement. By the previous paragraph, we already know that the left vertical $\crossEffect_f$ preserves tensor products, and so it only remains to show that it preserves the tensor unit. That is, we have to show that the natural 
    \begin{equation}\label{equation: unit map elbow localisation}
    \unit_{d,r}\rightarrow \crossEffect_f\unit_{d,k}
    \end{equation}
    is an equivalence. Consider the fibre sequence
    $$F \to \unit_{\func_{\udl{\ast}}(\sC^{\times k},\D)} \to \excisiveApproximation_{{d,k}}^{\lrcorner}\excisiveApproximation_{\udl{n}}(\unit_{\func_{\ast}(\sC^{\times k},\D)})\simeq \unit_{d,k}$$
    in the category $\func_{\udl{\ast}}(\sC^{\times k},\D)$. In particular, $\excisiveApproximation_{{d,k}}^{\lrcorner}\excisiveApproximation_{\udl{n}}F \simeq 0$. Note that the  map~\eqref{equation: unit map elbow localisation} is obtained by applying  $\excisiveApproximation_{{d,r}}^{\lrcorner}\excisiveApproximation_{\udl{m}}(\crossEffect_{f}-)$ to the map $\unit_{\func_{\udl{\ast}}(\sC^{\times k},\D)} \to  \unit_{d,k}$ because the functors $\crossEffect_{f}$ and $\excisiveApproximation_{{d,k}}^{\lrcorner}\excisiveApproximation_{\udl{n}}$ are symmetric monoidal. Hence, the fibre of \cref{equation: unit map elbow localisation} is $\excisiveApproximation_{{d,r}}^{\lrcorner}\excisiveApproximation_{\udl{m}}(\crossEffect_{f}F)$. Finally, by taking the left adjoints in the second commutative square in \cref{thm:key_pigeonhole_principle}, we obtain that $\excisiveApproximation_{{d,r}}^{\lrcorner}\excisiveApproximation_{\udl{m}}(\crossEffect_{f}F) \simeq 0$. 
\end{proof}

With symmetric monoidal structures in place, we now prove that the cross--effects satisfy the projection formula. Let $f\colon A\to E$ be a map of finite sets. As in \cref{construction: back-chevalley local systems}, we have an adjoint pair
\[
    \begin{tikzcd}
        \sC^E  \ar[rrr,shift left = 2, "\Delta_f"] &&& \sC^A  \ar[lll,shift left = 1, "\Pi_f"]
    \end{tikzcd}
\]
such that $\Delta_f$ is symmetric monoidal and $\Pi_f$ is $\sC^E$-linear, i.e., the canonical morphism
\begin{equation}\label{equation: projection local systems}
\Pi_f(c) \otimes d \to \Pi_f(c \otimes \Delta_f(d))
\end{equation}
is an equivalence for all $c \in \sC^A$ and $d \in \sC^E$. The adjoint pair $\Delta_f \dashv \Pi_f$ induces the adjoint pair 
\[
    \begin{tikzcd}
        \func(\sC^A,\D)  \ar[rrr,shift left = 2, "\Pi_{f!}"] &&& \func(\sC^E,\D)  \ar[lll,shift left = 1, "\Delta_{f!}"]
    \end{tikzcd}
\]
such that $\Delta_{f!}$ is a symmetric monoidal functor compatible with the Day convolution, see \cref{remark: propetries of day convolution}. In particular, since the category $\sC$ is stable, $\Pi_f$ is also a left adjoint to $\Delta_f$ and $\Pi_{f!}\simeq \Delta_f^*$.

\begin{prop}\label{proposition: projection formula day}
    The canonical morphism 
    \begin{equation}\label{equation: projection formula day}
    \Pi_{f!}(F\otimes\Delta_{f!}G)\to \Pi_{f!}(F)\otimes G
    \end{equation}
    is an equivalence for all $F \in \func(\sC^A,\D)$ and $G \in \func(\sC^E,\D)$.
\end{prop}

\begin{proof}
    By \cref{remark: propetries of day convolution}, we can assume that $\D=\spectra$ is the category of spectra. Then the functors $F,G$ are colimits of $\susps_+$ of representable ones and the desuspensions thereof. Hence, we can assume that $F\simeq \yoneda(c)$ and $G\simeq \yoneda(d)$ where $c\in \sC^A$ and $d\in \sC^E$. Again, by \cref{remark: propetries of day convolution}, we have equivalences
    \begin{align*}
        \Pi_{f!}(F\otimes\Delta_{f!}G) &\simeq \Pi_{f!}(\yoneda(c)\otimes\Delta_{f!}\yoneda(d)) \simeq \Pi_{f!}(\yoneda(c)\otimes\yoneda(\Delta_f(d))) \\
        &\simeq \Pi_{f!}(\yoneda(c \otimes \Delta_f(d))) \simeq \yoneda(\Pi_f(c \otimes \Delta_f(d)))
    \end{align*}
    and
    $$\Pi_{f!}(F)\otimes G \simeq \Pi_{f!}(\yoneda(c))\otimes \yoneda(d) \simeq \yoneda(\Pi_f(c))\otimes \yoneda(d) \simeq \yoneda(\Pi_f(c) \otimes d).$$
    Under these equivalences, the canonical morphism~\eqref{equation: projection formula day} is induced by the canonical equivalence~\eqref{equation: projection local systems}.
\end{proof}

Recall that the adjoint pair $\Pi_{f!}\dashv \Delta_{f!}$ induces the following adjoint pair 
\[
    \begin{tikzcd}
        \func_{\udl{\ast}}(\sC^A,\D)  \ar[rrr,shift left = 2, "\Pi_{f!}\simeq \Delta^*_f"] &&& \func_{\udl{\ast}}(\sC^E,\D)  \ar[lll,shift left = 1, "\crossEffect_f"]
    \end{tikzcd}
\]
between the categories of reduced functors. By \cref{corollary: cross-effect is symmetric monoidal}, $\crossEffect_f$ is symmetric monoidal.

\begin{cor}\label{corollary: projection reduced}
    The canonical morphism $$\Delta_f^*(F\otimes\crossEffect_{f}G)\to \Delta^*_{f}(F)\otimes G $$
    is an equivalence for all $F \in \func_{\udl{\ast}}(\sC^A,\D)$ and $G \in \func_{\udl{\ast}}(\sC^E,\D)$.
\end{cor}

\begin{proof}
    By construction, $\crossEffect_f G \simeq (\Delta_{f!}G)\reduce$. By \cref{lemma: reduction is a smashing localization}, the functor $(-)\reduce$ is a smashing localisation. Therefore, $F\otimes\crossEffect_{f}G \simeq F\otimes\Delta_{f!}G$. We finish the proof by applying \cref{proposition: projection formula day}.
\end{proof}

By \cref{thm:key_pigeonhole_principle} and \cref{prop:symmetric_monoidal_refinement_of_elbow_cross_effects}, the adjoint pair $\Delta_f^* \dashv \crossEffect_f$ restricts to the adjoint pair
\[
    \begin{tikzcd}
        \elbowCat{d,|A|}(\sC,\D)  \ar[rrr,shift left = 2, "\Delta^*_f"] &&& \elbowCat{d,|E|}(\sC,\D)  \ar[lll,shift left = 1, "\crossEffect_f"]
    \end{tikzcd}
\]
such that $\crossEffect_f$ is a symmetric monoidal functor with respect to the symmetric monoidal structure of \cref{cor:symmetric_monoidal_structure_on_elbow_cats}. The next corollary follows immediately from \cref{corollary: projection reduced} and \cref{cor:symmetric_monoidal_structure_on_elbow_cats}.

\begin{cor}\label{corollary: projection_elbow_cats}
    The canonical morphism $$\Delta_f^*(F\otimes\crossEffect_{f}G)\to \Delta^*_{f}(F)\otimes G $$
    is an equivalence for all $F \in \elbowCat{d,|A|}(\sC,\D)$ and $G \in \elbowCat{d,|E|}(\sC,\D)$. \qed
\end{cor}

\begin{cor}\label{cor:generators_are_dualisable}
    For every surjection $f\colon [r]\twoheadrightarrow [k]$, the objects $\Delta^*_f\unit_{d,r}\in\elbowCat{d,k}(\sC,\D)$ are dualisable, and in fact, self--dual objects.
\end{cor}

\begin{proof}
    By the projection formula of \cref{corollary: projection_elbow_cats} and the biadjunction $\Delta^*_f \dashv \crossEffect_f \dashv \Delta^*_f$, the natural morphism
    \begin{equation}\label{equation:dual_proj_formula}
    [\Delta^*_f \unit_{d,r}, F] \to \Delta^*_f \crossEffect_f (F)
    \end{equation}
    is an equivalence for any $F\in \elbowCat{d,k}(\sC,\D)$. By applying \cref{corollary: projection_elbow_cats} once again, we obtain that the natural map
    $$[\Delta^*_f \unit_{d,r}, \unit_{d,k}] \otimes F \to [\Delta^*_f \unit_{d,r}, F] $$
    is an equivalence for any $F \in\elbowCat{d,k}(\sC,\D)$. This implies that $\Delta^*_f \unit_{d,r}$ is dualisable. By \eqref{equation:dual_proj_formula}, we obtained that the dual $[\Delta^*_f \unit_{d,r}, \unit_{d,k}]$ is equivalent to $\Delta^*_f \crossEffect_f(\unit_{d,k}) \simeq \Delta^*_f \unit_{d,r}$, i.e., the object $\Delta^*_f \unit_{d,r}$ is self--dual.
\end{proof}


\subsection{Generators and localisations}
\label{subsec:generators_and_localisations}

Having established the internal logic of subdiagonal functors for a fixed degree $d$, we would like now to relate it to those associated to different degrees. The following result gives the first indication for this possibility, and for it, recall the Bousfield localisation $\excisiveApproximation^{\lrcorner}_{d,r}$ from \cref{prop:reflection_from_all_excisive_into_elbow_excisive}.

\begin{lem}\label{lem:inclusion_of_elbow_cats_of_smaller_degrees}
    Let $r\leq b\leq d$. Then the inclusion $\excisive_{b,r}(\sC,\D)\subseteq \excisive_{d,r}(\sC,\D)$ restrict to an inclusion $\elbowCat{b,r}(\sC,\D)\subseteq \elbowCat{d,r}(\sC,\D)$. Consequently, the Bousfield localisation $\excisiveApproximation^{\lrcorner}_{b,r}\colon \excisive_{d,r}(\sC,\D)\rightleftharpoons \elbowCat{b,r}(\sC,\D) \cocolon \inclusion$ restricts to a Bousfield localisation $\excisiveApproximation^{\lrcorner}_{b,r}\colon \elbowCat{d,r}(\sC,\D)\rightleftharpoons \elbowCat{b,r}(\sC,\D) \cocolon \inclusion$.
\end{lem}
\begin{proof}
    Since the inclusion $\excisive_{b,r}(\sC,\D)\subseteq \excisive_{d,r}(\sC,\D)$ is exact and everything in $\elbowCat{b,r}(\sC,\D)$ may be built as a finite limit of functors $F$ such that $F\simeq \excisiveApproximation_{\udl{\ell}}F$ where $\ell$ is an $r$--tuple with $\ell_1+\cdots+\ell_r\leq b$, we only have to show that such $F$ are also in $\elbowCat{d,r}(\sC,\D)$. But this is clear since $b\leq d$. The final statement is an immediate consequence of the fact that $\elbowCat{b,r}(\sC,\D)\subseteq \excisive_{d,r}(\sC,\D)$ factors through $\elbowCat{d,r}(\sC,\D)$, as supplied by the first part.
\end{proof}

We would like to enumerate the generators for the kernel of the Bousfield localisation above in  the case of $\sC=\spectra^{\omega}$ and $\D=\spectra$. To this end, we  first construct a family of jointly conservative functors from $\elbowCat{d,r}(\sC,\D)$.

\begin{lem}\label{lem:conserv_crosseffects}
    Let $d$ be a positive integer and $r\leq d$. Suppose that $F\in \elbowCat{d,r}(\sC,\D)$ is such that $\crossEffect_f F\simeq 0$ for all surjective maps $f\colon [d]\twoheadrightarrow [r]$. Then $F\in \elbowCat{d-1,r}(\sC,\D)$.
\end{lem}
\begin{proof}
    First, we will show that $F$ is $(d-r)$--excisive in the $i$--th variable. Indeed, let $f\colon [d] \twoheadrightarrow [r]$ be a surjection such that $|f^{-1}(i)|=d-r+1$ and $|f^{-1}(j)|=1$ for $j\neq i$. Then $\crossEffect_fF$ computes the $(d-r+1)$--th cross--effect with respect to the $i$--th variable. By the assumption, $F$ is $(d-r+1)$--excisive in the $i$--th variable and $\crossEffect_f F \simeq 0$, therefore $F$ is also $(d-r)$--excisive.

    Now, we will show that $F\in \elbowCat{d-1,r}(\sC,\D)$. By \cref{lem:pointwise_cubical_induction}, it suffices to argue that for every tuple $(\ell_1,\ldots,\ell_r)$ such that $\ell_1+\cdots+\ell_r\geq d$, writing $\udl{\delta}_i = (0,\ldots,0,1,0,\ldots,0)$ for the tuple with $1$ in component $i$ and zero elsewhere, the $r$--cube
    \[\Big\{P_{\underline{\ell}-\udl{\delta}_{i_1}-\cdots-\udl{\delta}_{i_a}}F\Big\}_{(i_1,\ldots,i_a)\in \mathrm{Subsets \: of\:} [r]}\]
    is a limit diagram. Since $F\in \elbowCat{d,r}(\sC,\D)$, this is true if $\ell_1+\cdots+\ell_r > d$. Therefore, we can assume $\ell_1+\cdots+\ell_r = d$. Then, we note that the total fibre $D$ of this cube is $(\ell_1,\ldots,\ell_r)$--reduced, i.e., $D$ is $(\ell_1,\ldots,\ell_r)$--homogeneous. Let $f\colon [d] \twoheadrightarrow [r]$ be a surjection such that $|f^{-1}(1)|=\ell_1,\ldots, |f^{-1}(r)|=\ell_r$. Since $D$ is homogeneous, $D$ is trivial if and only if $\crossEffect_fD$ is trivial. However, $\crossEffect_f D \simeq \crossEffect_f F \simeq 0$ by the assumption.
\end{proof}

By induction, this lemma and \cref{thm:key_pigeonhole_principle} imply immediately the following:

\begin{cor}\label{cor:jointly conservative}
    Let $d$ be a positive integer and $r\leq d$. Suppose that $F\in \elbowCat{d,r}(\spectra^\omega,\D)$ is such that $\crossEffect_f F(\sphere,\ldots,\sphere)\simeq 0$ for all surjective maps $f\colon [k]\twoheadrightarrow [r]$ and all $k\leq d$. Then $F\simeq 0$. \qed
\end{cor}

Now we are ready to compute the generators for $\elbowCat{d,r}$.

\begin{lem}\label{lem:generators_of_elbow_cats}
    Let $d$ be a positive integer and $r\leq d$. Then the presentable stable category $\elbowCat{d,r}$ is generated by the finite set $\{\Delta^*_f\unit_{d,k}\}_{r\leq k, f\in\surjectiveSet([k],[r])}$ of objects.
\end{lem}
\begin{proof}
    We need to show that this is a jointly conservative set of objects. Firstly, note that by \cref{prop:symmetric_monoidal_refinement_of_elbow_cross_effects}, for $r\leq k$, we have $\unit_{d,k}\simeq \crossEffect_{t_k}\unit_{d,1}$ for $t_k\colon [k]\twoheadrightarrow[1]$ the unique map and $\unit_{d,1}\in\elbowCat{d,1}=\excisive_d$ is the tensor unit. Let $F\in\elbowCat{d,r}$, and suppose $\map_{\elbowCat{d,r}}(\Delta^*_f\unit_{d,k},F)\simeq 0$ for all $r\leq k$ and $f\in\surjectiveSet([k],[r])$. Hence, we see that
    \begin{equation*}
        \begin{split}
            0\simeq \map_{\elbowCat{d,r}}(\Delta^*_f\unit_{d,k},F)&\simeq \map_{\elbowCat{d,r}}(\Delta^*_f\crossEffect_{t_k}\unit_{d,1},F)\\
            &\simeq \map_{\excisive_d}(\unit_{d,1},\Delta^*_{t_k}\crossEffect_fF)\simeq \loops(\crossEffect_fF)(\sphere,\ldots,\sphere)
        \end{split}
    \end{equation*}
    where the third equivalence uses the biadjunctions from \cref{cor:omnibus_multivariable_cross-effects}. But then for any $F\in\elbowCat{d,r}\subseteq \excisive_{d,r}$, we know by \cref{cor:jointly conservative} that the cross--effects evaluated at tuples of spheres are jointly conservative. 
\end{proof}

We may now compute the kernel of $\excisiveApproximation^{\lrcorner}_{b,r}$.

\begin{prop}\label{prop:kernel_of_P_1}
    Let  $r\leq b\leq d$. The  kernel for the  localisation $\excisiveApproximation_{b,r}^{\lrcorner}\colon \elbowCat{d,r} \rightarrow \elbowCat{b,r}$ is generated by the finite set of objects $\{\Delta^*_f\unit_{d,k}\}_{b+1\leq k, f\in\surjectiveSet([k],[r])}$.
\end{prop}
\begin{proof}
    By the description of the generators from \cref{lem:generators_of_elbow_cats}, we just have to show that for $r\leq k$ and $f\colon [k]\twoheadrightarrow [r]$  a surjection, $P^{\lrcorner}_{b,r}\Delta^*_f\unit_{d,k}\simeq 0$ if and only if $b+1\leq k$. For this, we first make the following computation: let $F\in\elbowCat{b,r}\subseteq\elbowCat{d,r}$. Then 
    \[\map_{\elbowCat{b,r}}(\excisiveApproximation_{b,r}^{\lrcorner}\Delta^*_f\unit_{d,k},F)\simeq \map_{\elbowCat{d,r}}(\Delta^*_f\unit_{d,k},F)\simeq \map_{\elbowCat{d,k}}(\unit_{d,k},\crossEffect_fF).\]

    Now, suppose $b+1\leq k$, and suppose without loss of generality that $F\simeq \excisiveApproximation_{\udl{\ell}}F$ where $\ell$ is an 
    $r$--tuple such that $\ell_1+\cdots+\ell_r\leq b$. We write $k=k_1+\cdots +k_r$ where $k_i$ is the size of the preimage of $i\in[r]$ under the surjection $f\colon [k]\twoheadrightarrow [r]$. By pigeonhole,  there exists an $i$ such that $k_i\geq \ell_i+1$. Hence, $\crossEffect_fF\simeq 0$ and so $\map_{\elbowCat{b,r}}(\excisiveApproximation_{b,r}^{\lrcorner}\Delta^*_f\unit_{d,k},F)\simeq 0$ for all $F\in\elbowCat{b,r}$, whence $\excisiveApproximation_{b,r}^{\lrcorner}\Delta^*_f\unit_{d,k}\simeq 0$.

    For the converse, suppose $k\leq b$ and as before write $k=k_1+\cdots+k_r$. We want to show that $\excisiveApproximation_{b,r}^{\lrcorner}\Delta^*_f\unit_{d,k}\not\simeq 0$. For this, just consider for example the functor $F\coloneqq (-)^{\otimes k_1}_{h\Sigma_{k_1}}\otimes (-)^{\otimes k_2}_{h\Sigma_{k_2}}\otimes\cdots \otimes(-)^{\otimes k_r}_{h\Sigma_{k_r}}$ which is $k_i$--homogeneous in the $i$--th variable. This is clearly subdiagonal in $\excisive_{b,r}$ since $k\leq b$, and hence $F\in\elbowCat{b,r}$. In this case, we get $\map_{\elbowCat{d,k}}(\unit_{d,k},\crossEffect_fF)\simeq \loops\sphere$, and so $\map_{\elbowCat{b,r}}(\excisiveApproximation_{b,r}^{\lrcorner}\Delta^*_f\unit_{d,k},F)\not\simeq0$, as required.
\end{proof}

\newpage
\part{Calculus, parametrised} \label{part:parametrised}

\section{Parametrised category theory}\label{section:parametrised_categories}

\subsection{Recollections}\label{subsection:recollection_parametrised}
We will heavily employ the theory of parametrised higher categories as introduced and developed in \cite{expose1Elements,shahThesis,nardinExposeIV,kaifPresentable,kaifNoncommMotives}.  We briefly recall here the basic aspects of this theory sufficient for the purpose of reading this article, and we refer the reader to these sources for more details on the theory. The basic objects are the following:

\begin{recollect}[Basics of parametrised categories]
    Let $\baseCat$ be a small category. A \textit{$\baseCat$--category} is an object $\udl{\sC}$ in $\cat_{\baseCat}\coloneqq \func(\baseCat\op,\cat)$, and a \textit{$\baseCat$--functor} is a morphism in $\cat_{\baseCat}$. When the context is clear, we will often say functor to mean a $\baseCat$--functor. A \textit{$\baseCat$--space} is an object $\udl{X}$ in $\spc_{\baseCat}\coloneqq \func(\baseCat\op,\spc)$. In particular, since we have a fully faithful inclusion $\spc\subseteq \cat$ by viewing a space as an $\infty$--groupoid, we also have the corresponding embedding $\spc_{\baseCat}\subseteq \cat_{\baseCat}$. For an object $v\in\baseCat$, we will write the evaluation of a $\baseCat$--category $\udl{\sC}$ at $v$ as $\udl{\sC}(v)\in\cat$. Using the automorphism $(-)\op\colon \cat\xrightarrow{\simeq}\cat$, we obtain the parametrised version of the opposite category functor $(-)\vop\colon \cat_{\baseCat}\xrightarrow{\simeq}\cat_{\baseCat}$.

    By an \textit{object} in a $\baseCat$--category $\udl{\sC}$, we will mean a $\baseCat$--functor $x\colon \terminalTCat\rightarrow \udl{\sC}$. We have a \textit{global sections functor} given by  $\Gamma_{\baseCat}\coloneqq \lim_{\baseCat\op}\colon \cat_{\baseCat}=\func(\baseCat\op,\cat)\rightarrow \cat$.
    
    For $v\in\baseCat$, writing (by a mild abuse of notation) $v\colon \baseCat_{/v}\rightarrow \baseCat$ for the forgetful functor, we may obtain a $\baseCat_{/v}$-- from a $\baseCat$--category by considering its image under the basechange functor $v^* \colon \cat_{\baseCat}=\func(\baseCat\op,\cat)\rightarrow \func(\baseCat_{/v}\op,\cat)=\cat_{\baseCat_{/v}}$. 
    
    There are natural refinements of $\cat_{\baseCat}$ and $\spc_{\baseCat}$ to become themselves $\baseCat$--categories (in a larger universe), which we write as $\udl{\cat}$ and $\udl{\spc}$ respectively. Concretely, for example, the value of $\udl{\cat}$ at $v\in \baseCat$ is given by $\cat_{\baseCat_{/v}}\in\cat$. There is a natural \textit{mapping $\baseCat$--space functor} associated to any $\udl{\sC}\in\cat_{\baseCat}$ denoted as 
    $\myuline{\map}_{\udl{\sC}}(-,-)\colon \udl{\sC}\vop\times \udl{\sC} \longrightarrow \udl{\spc}$.
    
    The category $\cat_{\baseCat}$ is cartesian closed. In particular, for any two $\baseCat$--categories $\udl{\sC}$ and $\udl{\D}$, we may obtain the \textit{$\baseCat$--category of $\baseCat$--functors} $\udl{\func}(\udl{\sC},\udl{\D})\in\cat_{\baseCat}$ given by taking the internal hom object. Applying global sections  to $\udl{\func}(\udl{\sC},\udl{\D})$ yields the \textit{category $\func_{\baseCat}(\udl{\sC},\udl{\D})\coloneqq \Gamma_{\baseCat}\udl{\func}(\udl{\sC},\udl{\D})$ of $\baseCat$--functors from $\udl{\sC}$ to $\udl{\D}$}.

    There is a good theory of parametrised adjunctions, and one way to obtain this is just to view $\cat_{\baseCat}$ with its natural $(\infty,2)$--categorical refinement. For example, a \textit{$\baseCat$--adjunction $L\colon \udl{\sC}\rightleftharpoons \udl{\D} : R$ } is the datum of functors $L\colon \udl{\sC}\rightarrow \udl{\D}$ and $R\colon \udl{\D}\rightarrow \udl{\sC}$ together with an equivalence $\myuline{\map}_{\udl{\sC}}(-,R-)\simeq \myuline{\map}_{\udl{\D}}(L-,-)\colon \udl{\sC}\vop\times \udl{\D}\longrightarrow \udl{\spc}$.

    The category of  \textit{symmetric monoidal $\baseCat$--categories} is defined as $\cmonoid(\cat_{\baseCat})\simeq\func(\baseCat\op,\cmonoid(\cat))$. Thus,  the datum of a symmetric monoidal $\baseCat$--category is, roughly speaking, given by the datum of a symmetric monoidal category $\udl{\sC}(v)$ for every $v\in\baseCat$, and for every morphism $f\colon v\rightarrow w$ in $\baseCat$, the contravariant functoriality map $f^*\colon \udl{\sC}(w)\rightarrow \udl{\sC}(v)$ is equipped with a symmetric monoidal structure.
\end{recollect}

\begin{rmk}
    Here, we have only mentioned symmetric monoidal $\baseCat$--categories, i.e., objects in $\cmonoid(\cat_{\baseCat})$. Using the adjectives of atomic orbitality below, Nardin and Nardin--Shah \cite{nardinThesis,nardinShah} have studied the more sophisticated notion of $\baseCat$--symmetric monoidality which also captures the so--called Hill--Hopkins--Ravenel norms in equivariant homotopy theory. This will not play a  role in this article.
\end{rmk}

In order to begin considering more algebraic aspects in the parametrised theory, we will need more conditions on the base category $\baseCat$. For this, we will need the following piece of notation:

\begin{nota}
    Let $\baseCat$ be a small category. Let $\finite_{\baseCat}\subseteq \presheaf(\baseCat)=\spc_{\baseCat}$ be the smallest subcategory of $\presheaf(\baseCat)$ containing $\baseCat$ and closed under finite coproducts. This has the universal property of the  finite--coproduct--cocompletion of $\baseCat$.
\end{nota}

\begin{defn}\label{defn:types of orbits}
    Let $\baseCat$ be a small category. We say that it is \textit{orbital} if $\finite_{\baseCat}$ admits pullbacks; \textit{atomic} if for all $f\colon v  \rightarrow w$ and $r\colon w\rightarrow v$ such that $rf$ is an equivalence, then $f$ and $r$ were already equivalences; \textit{inductive orbital} (cf.~\cite[Def. 1.10]{wilsonSlices}) if it is an orbital $1$--category with finitely many isomorphism classes of objects and morphisms such that every endomorphism in $\baseCat$ is an equivalence.
\end{defn}

For every $v\in \baseCat$, if $\baseCat$ is atomic orbital (resp. inductive orbital), then so is $\baseCat_{/v}$. Also, it is immediate to see that inductive orbitality implies atomic orbitality.

\begin{example}\label{example:atomic_orbital}
    The key example of an inductive orbital category for us is that of $\epiCategory_{d}$, the 1--category of finite sets up to size $d$ and surjections. We refer the reader to \cite[Ex. 4.2]{nardinExposeIV}  for more examples.
\end{example}

\begin{recollect}[Parametrised (co)limits]\label{recollect:parametrised_colimits}
    Let $\udl{\sC},\udl{I}\in\cat_{\baseCat}$, and $I\colon\udl{I}\rightarrow\terminalTCat$ be the unique functor. If they exist, we call the left adjoint (resp. right adjoint) to the restriction functor $I^*\colon \udl{\sC}\rightarrow \udl{\func}(\udl{I},\udl{\sC})$ the \textit{$\udl{I}$--shaped parametrised colimit (resp.  limit) in $\udl{\sC}$.} One special class of parametrised (co)limits are those indexed by the \textit{constant $\baseCat$--category}, i.e., those $\udl{I}\in\cat_{\baseCat}$ of the form a constant $\baseCat$--indexed diagram  $\udl{\constant}_{\baseCat}(J)$ with value $J$ for some $J\in\cat$. Such (co)limits are called \textit{fibrewise (co)limits}. Shah provided in \cite[Cor. 5.9]{shahThesis} a criterion for when fibrewise (co)limits exist, namely: $\udl{\sC}$ admits $\udl{\constant}_{\baseCat}(J)$--(co)limits if and only if for each $v\in\baseCat$, the category $\udl{\sC}(v)$ admits $J$--shaped (co)limits, and for every map $f\colon v\rightarrow w$ in $\baseCat$, the functor $f^*\colon \udl{\sC}(w)\rightarrow\udl{\sC}(v)$ preserves $J$--shaped (co)limits. When we say that a $\baseCat$--category admits initial objects/final objects/coproducts/products/pushouts/pullbacks/filtered colimits/geometric realisations etc., we will mean that the $\baseCat$--category admits the fibrewise (co)limits associated to these diagrams.
    
    Another important special class of parametrised (co)limits are the ones coming from objects in $\baseCat$ as follows: let $v\in\baseCat$. We may view it as an object in $\cat_{\baseCat}$ via the inclusion $\baseCat\subseteq \presheaf(\baseCat)=\spc_{\baseCat}\subseteq \cat_{\baseCat}$, which we denote as $\udl{v}$. Writing $v\colon \udl{v}\rightarrow\terminalTCat$ for the unique map, we may speak of $v$--(co)limits. Such (co)limits are called indexed (co)products. It was in fact shown in \cite[\textsection 12]{shahThesis} that all parametrised colimits may be built from fibrewise colimits and finite indexed coproducts.  When $\baseCat$ is orbital, a $\baseCat$--category admits all indexed coproducts if and only if (1) for all morphisms $f\colon v\rightarrow w$ in $\baseCat$, the functor $f^*\colon \udl{\sC}(w)\rightarrow\udl{\sC}(v)$ admits a left adjoint $f_!$ and (2) for all pullbacks in $\finite_{\baseCat}$ as in the left square in 
        \begin{equation}\label{eqn:double_coset_decomposition_parametrised}
            \begin{tikzcd}
                \coprod_iZ_i\rar["\sqcup_i\overline{f}_i"]\dar["\sqcup_i \overline{g}_i"']\ar[dr, phantom,very near start, "\lrcorner"] & U\dar["g"] && \prod_i\udl{\sC}(Z_i) \dar["\prod_i\overline{g}_{i!}"']\ar[dr,phantom, "\Rightarrow"]& \udl{\sC}(U)\dar["g_!"']\lar["\prod_i\overline{f}_i^*"']\\
                V\rar["f"] & W && \udl{\sC}(v) & \udl{\sC}(w)\lar["f^*"] 
            \end{tikzcd}
        \end{equation}
        the adjointed square of categories given on the right commutes. This gives a ``double coset decomposition'' in the setting of orbital base categories. Importantly, when $\baseCat$ is also atomic,  the double coset decomposition for $f^*f_!$ always contains a copy of $\id_V$ which allows one to define a semiadditivity norm map, cf. \cref{recollect:parametrised_stability}. By passing to the opposite category,  analogous statements also hold for parametrised limits. 

    Furthermore, using these fibrewise criteria for the existence of fibrewise (co)limits and indexed (co)products, we may also characterise when a functor $F\colon \udl{\sC}\rightarrow\udl{\D}$ preserves parametrised (co)limits, namely for example, $F\colon \udl{\sC}(v)\rightarrow \udl{\D}(v)$ preserves colimits in the ordinary sense, and for every $f\colon v\rightarrow w$, the  adjointed square 
    \begin{center}
        \begin{tikzcd}
            \udl{\sC}(v) \rar["F"]\ar[dr, phantom, "\Leftarrow"]& \udl{\D}(v)\\
            \udl{\sC}(w) \rar["F"]\uar["f_!"]& \udl{\D}(w) \uar["f_!"']
        \end{tikzcd}
    \end{center}
    coming from the commutation datum of $F$ with $f^*$ commutes.
\end{recollect}

\begin{recollect}[Parametrised presentability]
    Write $\largeCat_{\baseCat}\coloneqq \func(\baseCat\op,\largeCat)$ for the category of large $\baseCat$--categories. Within this, we may find the nonfull subcategory $\presentable^L_{\baseCat}$ of \textit{presentable $\baseCat$--categories} and left adjoint functors. The theory of parametrised presentable categories is much the same as in the ordinary world, for instance, there is also the adjoint functor theorem in this context. Moreover, presentable $\baseCat$--categories may be characterised as a $\baseCat$--category which admits all parametrised colimits and which are fibrewise presentable (cf.~\cite[Thm. 6.1.2 (7)]{kaifPresentable}). There is a symmetric monoidal structure on $\presentable^L_{\baseCat}$ refining the Lurie tensor product, whose tensor unit is the $\baseCat$--category $\udl{\spc}$ of $\baseCat$--spaces. Objects in $\calg(\presentable^L_{\baseCat})\subset \cmonoid(\cat_{\baseCat})$ are called \textit{presentably symmetric monoidal $\baseCat$--categories}, and they may be viewed as the symmetric monoidal $\baseCat$--categories whose underlying $\baseCat$--category is presentable and whose tensor product commutes with parametrised colimits in each variable.
\end{recollect}

Aside from making transparent in what way the category $\epiCategory_{d}$ organises many of the structures in Goodwillie calculus, one of the key reasons for us to take the parametrised point of view is that, while spectral Mackey functors do not admit a nice universal property in categories akin to the one enjoyed by the category of spectra, it \textit{does so} when viewed as a parametrised category. The following set of recollections will explain this.

\begin{recollect}[Parametrised semiadditivity and stability]\label{recollect:parametrised_stability}
    Suppose $\baseCat$ is atomic orbital. Let $\udl{\sC}$ be a pointed $\baseCat$--category (i.e., the fibrewise initial and final objects coincide). Using the hypothesis of atomic orbitality, Nardin \cite{nardinExposeIV} constructed for every morphism $f\colon v\rightarrow w$ in $\baseCat$ a canonical comparison transformation $f_! \longrightarrow f_*$
    of functors $\udl{\sC}(v)\rightarrow \udl{\sC}(w)$. We then say that the pointed $\baseCat$--category $\udl{\sC}$ is \textit{$\baseCat$--semiadditive} if it is fibrewise semiadditive and all these comparison transformations are equivalences. Additionally, we say that $\udl{\sC}$ is \textit{$\baseCat$--stable} if it is $\baseCat$--semiadditive and fibrewise stable. The universal example of a $\baseCat$--stable category is given by $\myuline{\spectra}$, the $\baseCat$--category of \textit{$\baseCat$--spectra} which at $v\in\baseCat$ is given by spectral $\baseCat_{/v}$--Mackey functors $\myuline{\spectra}(v)= \mackey({\baseCat_{/v}},\spectra)\coloneqq \func^{\times}(\spanCat(\finite_{\baseCat_{/v}}),\spectra)$. It turns out that $\myuline{\spectra}$ is an idempotent algebra in $\presentable^L_{\baseCat}$ by \cite[Cor. 3.28]{nardinThesis} (cf. also \cite[Thm. 7.41]{CLLSpans} for a more general and fully worked-out treatment), and the full subcategory $\presentable^L_{\baseCat\mathrm{-st}}\subseteq \presentable^L_{\baseCat}$ of $\baseCat$--stable presentable categories is equivalent to $\myuline{\spectra}$--modules in $\presentable^L_{\baseCat}$ by \cite[Prop. 2.2.19]{kaifNoncommMotives}. In particular, $\myuline{\spectra}$ is the tensor unit in $\presentable^L_{\baseCat\mathrm{-st}}$, and so it is the initial  presentably symmetric monoidal $\baseCat$--stable category.
\end{recollect}

\subsection{Universal spaces and stratifications}\label{subsection:universal_spaces_and_stratifications}

In this subsection, we introduce the universal space associated to ``families'', extending the notion from equivairant homotopy theory to all atomic orbital categories $\baseCat$.  These will be the gadgets we use to organise our stratification arguments later.  Using these universal spaces, we record a categorification of geometric fixed points following \cite{parametrised_Poincare_duality}. Throughout, we will assume that $\baseCat$ is an atomic orbital category, unless otherwise specified.

\begin{cons}\label{cons:universalSpaces}
    Let $f\colon \baseCat_{\family}\subseteq \baseCat$ be a sieve, i.e., if there is a map $y\rightarrow x$ where $x\in\baseCat_{\family}$ and $y\in \baseCat$, then $y\in\baseCat_{\family} $ as well. We then define  $E\family \in \spc_{\baseCat}$ as $f_!\ast$ where $\ast\in\func(\baseCat_{\family}\op,\spc)$ is the final object. We claim now that 
    \begin{equation*}
        E\family(v) \simeq \begin{cases}
            \emptyset & \text{ if } v \notin \baseCat_{\family}\\
            \ast & \text{ if } v\in \baseCat_{\family}.
        \end{cases}
    \end{equation*}
    To see this, since $f$ is fully faithful, the value of $E\family$ at $v\in\baseCat_{\family}$ is clear; if $v\notin\baseCat_{\family}$, then by the pointwise Kan extension formula, we just need to compute $\colim_{(\baseCat_{\family}\op)_{/v}}f$. But since $v\notin\baseCat_{\family}$ and $\baseCat_{\family}$ is a sieve, we get that $(\baseCat_{\family}\op)_{/v}=\emptyset$, and so the colimit is $\emptyset$ as well, as wanted.

    By adding a basepoint to the unique map $E\family\rightarrow\ast$, we obtain a map $E\family_+\rightarrow S^0$ in $\spc_{\baseCat,*}$. We then define $\widetilde{E\family}\in\spc_{\baseCat,*}$ as the cofibre
    \begin{equation}\label{eqn:isotropy_cofibre_sequence}
        E\family_+\longrightarrow S^0\longrightarrow \widetilde{E\family}.
    \end{equation}
    Since evaluations preserve (co)limits, we obtain the expected description
    \begin{equation*}
        \widetilde{E\family}(v) \coloneqq \begin{cases}
            S^0 & \text{ if } v \notin \baseCat_{\family}\\
            \ast & \text{ if } v\in \baseCat_{\family}.
        \end{cases}
    \end{equation*}
\end{cons}

\begin{rmk}
    The object $\widetilde{E\family}\in\spc_{\baseCat,*}$  is easily seen to be an idempotent algebra under the pointwise smash product symmetric monoidal structure on $\spc_{\baseCat,*}$. Hence, it defines a nonunital symmetric monoidal idempotent endofunctor $\widetilde{E\family}\wedge-\colon \spc_{\baseCat,*}\rightarrow\spc_{\baseCat,*}$. Writing $s\colon \baseCat_{\widetilde{\family}} \hookrightarrow \baseCat$ for the inclusion of the cosieve complement to $\baseCat_{{\family}}$, it is easy to see that $\widetilde{E\family}\wedge-$ yields the symmetric monoidal Bousfield localisation
    \begin{center}
        \begin{tikzcd}
            \spc_{\baseCat,*} \rar[shift left = 1, "s^*"] & \func(\baseCat_{\widetilde{\family}}\op,\spc_*)\simeq \module_{\widetilde{E\family}}(\spc_{\baseCat,*}).\lar[shift left = 1, "s_*", hook]
        \end{tikzcd}
    \end{center}
    In particular, $\widetilde{E\family}$ enjoys the following universal property: for  $A, X\in \spc_{\baseCat,*}$ with $X(v)\simeq \ast$ for all $v\in\baseCat_{{\family}}$, the map $A\rightarrow \widetilde{E\family}\wedge A$ induces an equivalence \[\map_{\spc_{\baseCat,*}}(\widetilde{E\family}\wedge A,X)\xrightarrow{\simeq }\map_{\spc_{\baseCat,*}}(A,X).\]
\end{rmk}

\begin{nota}\label{nota:EP_twiddle}
    Suppose that $\baseCat$ is an inductive orbital category with a final object $t\in\baseCat$.    We write $\baseCat_{{\proper}}\subseteq \baseCat$ for the full subcategory away from the final object $t\in\baseCat$. This is a sieve because if $x\rightarrow y$ is a map in $\baseCat$ where $y\neq t$, then $x\neq t$ also because all endomorphisms are equivalences by inductive orbitality. This thus yields a universal space $\widetilde{E\proper}$.
\end{nota}

In equivariant homotopy theory, the cofibre sequence \cref{eqn:isotropy_cofibre_sequence} underlies one of the most important inductive techniques in the subject, namely that of ``isotropy separation''. Such methods have been considered by Glasman  \cite{glasmanStratified} and Wilson \cite{wilsonSlices} in the more general parametrised setting. We will now record some points from this philosophy as a basis for our stratification arguments later, following closely and slightly expanding the ``categorified'' treatment given in \cite[\textsection 2.2]{parametrised_Poincare_duality}. Unless otherwise specified, $\baseCat$ will be an inductive orbital category for the rest of the subsection.

In the first two constructions, we will set the stage by declaring the various categories and functors for presentable categories that we need.

\begin{cons}\label{cons:sym_mon_structure_on_presentable_stable}
    Let $\baseCat$ be an atomic orbital category. Then by \cite[Prop. 2.2.19]{kaifNoncommMotives}, the object $\myuline{\spectra}_{\baseCat}\in\presentable^L_{\baseCat}$ is an idempotent algebra for the $\baseCat$--stable presentable categories. And so we may endow $\presentable^L_{\baseCat-\mathrm{st}}\simeq \module_{\myuline{\spectra}_{\baseCat}}(\presentable^L_{\baseCat})$ with a symmetric monoidal structure making the Bousfield localisation $\myuline{\spectra}_{\baseCat}\otimes-\colon \presentable^L_{\baseCat}\rightarrow \presentable^L_{\baseCat-\mathrm{st}}$ into a smashing localisation.
\end{cons}

\begin{cons}\label{cons:inclusion_of_stables_family}
    Let $\baseCat$ be an inductive orbital category, $ \baseCat_{\family}\subseteq \baseCat$ a sieve, and write $\iota\colon\baseCat_{\widetilde{\family}}\subseteq \baseCat$ be the inclusion of the complement of the sieve. By \cite[Lem. 1.14]{wilsonSlices}, $\baseCat_{\widetilde{\family}}$ is yet again an inductive orbital category. In particular, it makes sense to speak of $\baseCat_{\widetilde{\family}}$--stability. The inclusion $\iota$  induces by the right Kan extension the inclusion $\iota_*\colon \widehat{\cat}_{\baseCat_{\widetilde{\family}}}\hookrightarrow \widehat{\cat}_{\baseCat}$ which is a symmetric monoidal functor with respect to the cartesian symmetric monoidal structures. This identifies the source of $\iota_*$ as the full subcategory of the target consisting of those $\baseCat$--categories $\udl{\sC}$ with the property that for every $v\in \baseCat_{\family}$, $v^*\udl{\sC}\simeq \udl{\ast}\in\widehat{\cat}_{\baseCat_{/v}}$. 

    Note that the inclusion $\iota_*$ induces the fully faithful functor $\iota_*\colon \presentable^L_{\baseCat_{\widetilde{\family}}-\mathrm{st}}\hookrightarrow\presentable^L_{\baseCat-\mathrm{st}}$ because parametrised stability (cf.~\cref{recollect:parametrised_stability}) is  the requirement of being fibrewise stable and for every morphism $f\colon u\rightarrow v$ in the base category, the canonical transformation of functors $f_!\Rightarrow f_*$ of functors $ \udl{\sC}(u)\rightarrow \udl{\sC}(v)$ is an equivalence. 
\end{cons}

Next, we give the standard  formulas for the (co)tensors over pointed spaces.

\begin{cons}\label{cons:tensoring_over_anima}
    For $\udl{\sC}\in\presentable^L_{\baseCat\mathrm{-st}}$ and $X\in\spc_{\baseCat,*}$ (and writing $\udl{X}\in\cat_{\baseCat,*}$ for the corresponding pointed $\baseCat$--category), we define
    \[X\wedge \udl{\sC}\coloneqq \cofib(\udl{\sC}\simeq \colim_{\ast}\udl{\sC}\rightarrow \colim_{X}\udl{\sC})\] \[\udl{\func}_*(\udl{X},\udl{\sC})\coloneqq \fib_0(\udl{\func}(\udl{X},\udl{\sC})\rightarrow\udl{\func}(\ast,\udl{\sC})\simeq \udl{\sC}).\] By \cite[Cons. 2.2.23]{parametrised_Poincare_duality}, $X\wedge \udl{\sC}$ satisfies the universal property of being the tensoring of $\presentableStable{\baseCat}$ over $\spc_{\baseCat,*}$. Importantly, via the equivalence $\presentable^L_{\baseCat}\simeq (\presentable^R_{\baseCat})\op$ from for example \cite[Prop. 2.2.15]{kaifNoncommMotives}, we may compute $X\wedge \udl{\sC}$ as $\udl{\func}_*(\udl{X},\udl{\sC})$, using that $\lim_X\udl{\sC}\simeq \udl{\func}(\udl{X},\udl{\sC})$ since $\udl{\presentable}^R\subseteq \udl{\widehat{\cat}}$ is closed under arbitrary parametrised limits by \cite[Lem. 2.5.5]{kaifNoncommMotives}. In particular, $Y_+\wedge\udl{\sC}\simeq \udl{\func}(\udl{Y},\udl{\sC})$ for any $Y\in\spc_{\baseCat}$.
\end{cons}

\begin{rmk}
    The map $\udl{\func}_*(\udl{X},\udl{\sC})\rightarrow\udl{\func}(\udl{X},\udl{\sC})$ is fully faithful because it is pulled back from the fully faithful inclusion $0\hookrightarrow \udl{\sC}$. As such, pointed functors from groupoids to pointed categories may really be viewed as a property of a functor.
\end{rmk}

With the preliminaries out of the way, we may now obtain the following important categorification of geometric fixed points for arbitrary inductive orbital base categories $\baseCat$.

\begin{prop}\label{prop:smashing_left_adjoint}
    Let $\baseCat$ be an inductive orbital category and $\iota\colon\baseCat_{\family}\subseteq \baseCat$ a sieve. Then the inclusion $\iota_*\colon \presentable^L_{\baseCat_{\widetilde{\family}}-\mathrm{st}}\hookrightarrow\presentable^L_{\baseCat-\mathrm{st}}$  from \cref{cons:inclusion_of_stables_family} participates in  a smashing localisation
    \[
    \begin{tikzcd}
        \presentable^L_{\baseCat-\mathrm{st}}  \ar[rrr,shift left = 2, "\widetilde{\iota}^*=\widetilde{E\family}\wedge-"] &&& \presentable^L_{\baseCat_{\widetilde{\family}}-\mathrm{st}}  \ar[lll,shift left = 1, hook, "\iota_*"]
    \end{tikzcd}
    \]
    using the symmetric monoidal structures from \cref{cons:sym_mon_structure_on_presentable_stable}.
\end{prop}
\begin{proof}
    By \cite[Lem. 2.1.32] {parametrised_Poincare_duality} and the same arguments as in \cite[Thm. 2.2.26]{parametrised_Poincare_duality}, we have the smashing localisation $\widetilde{\iota}^*=\widetilde{E\family}\wedge-\colon \presentable^L_{\baseCat}\rightleftharpoons \presentable^L_{\baseCat_{\widetilde{\family}}}\cocolon \iota_*$ which simultaneously identifies $\presentable^L_{\baseCat_{\widetilde{\family}}} $ as both the full subcategory of $\presentable^L_{\baseCat} \cap \widehat{\cat}_{\baseCat_{\widetilde{\family}}}$ and as $\module_{\widetilde{E\family}\wedge \udl{\spc}_{\baseCat}}(\presentable^L_{\baseCat})$.    On the other hand, for any $\udl{\sC}\in\presentable^L_{\baseCat-\mathrm{st}}$, we also have that $\widetilde{E\family}\wedge\udl{\sC}\simeq \udl{\func}_*(\udl{\widetilde{E\family}},\udl{\sC})$ (cf. \cref{cons:tensoring_over_anima}), and so it is in $\presentable^L_{\baseCat-\mathrm{st}} \cap \widehat{\cat}_{\baseCat_{\widetilde{\family}}}=\presentable^L_{\baseCat_{\widetilde{\family}}-\mathrm{st}}$ since parametrised stability is closed under cotensoring.  Thus, the top Bousfield localisation restricts to the bottom one in 

    \[
    \begin{tikzcd}
        \presentable^L_{\baseCat} \ar[rrr,shift left = 1, "\widetilde{\iota}^*=\widetilde{E\family}\wedge-"] \dar["\myuline{\spectra}_{\baseCat}\otimes-"', shift left =-2,dashed] &&& \presentable^L_{\baseCat_{\widetilde{\family}}} \ar[lll,shift left = 2, hook, "\iota_*"] \dar["\myuline{\spectra}_{\baseCat_{\widetilde{\family}}}\otimes-"', shift left =-2,dashed] \\
        \presentable^L_{\baseCat-\mathrm{st}} \uar[shift right= 2, hook] \ar[rrr,shift left = 2, "\widetilde{\iota}^*=\widetilde{E\family}\wedge-"] &&& \presentable^L_{\baseCat_{\widetilde{\family}}-\mathrm{st}}. \uar[shift right = 2, hook] \ar[lll,shift left = 1, hook, "\iota_*"]
    \end{tikzcd}
    \] In particular, this makes the solid squares commute and so the bottom map $\widetilde{\iota}^*$ inherits a lax symmetric monoidal structure. 
    
    Now, by \cref{cons:sym_mon_structure_on_presentable_stable}, the vertical inclusions admit the dashed left adjoints which are smashing localisations, exhibiting bottom terms as $\module_{\myuline{\spectra}_{\baseCat}}(\presentable^L_{\baseCat})$ and $\module_{\myuline{\spectra}_{\baseCat_{\widetilde{\family}}}}(\presentable^L_{\baseCat_{\widetilde{\family}}})$ respectively. Hence, all inclusions in sight preserve the tensor products and so the bottom $\widetilde{\iota}^*$ map even preserves tensor products. Furthermore, since the square of right adjoints commute and the top map $\widetilde{\iota}^*$ is symmetric monoidal, we also get an equivalence of idempotent algebras $\widetilde{E\family}\wedge \myuline{\spectra}_{\baseCat}\simeq  \myuline{\spectra}_{\baseCat_{\widetilde{\family}}}\in\presentable^L_{\baseCat_{\widetilde{\family}}}$. Since the tensor unit in $\presentable^L_{\baseCat-\mathrm{st}}$ and $\presentable^L_{\baseCat_{\widetilde{\family}}-\mathrm{st}}$ are $\myuline{\spectra}_{\baseCat}$and $\myuline{\spectra}_{\baseCat_{\widetilde{\family}}}$ respectively, we see, all in all, that the bottom map $\widetilde{\iota}^*$ is in fact symmetric monoidal that exhibits $\presentable^L_{\baseCat_{\widetilde{\family}}-\mathrm{st}}$ as $\module_{\widetilde{E\family}\wedge \myuline{\spectra}_{\baseCat}}(\presentable^L_{\baseCat})$, as required.
\end{proof}

The tensoring construction allows us to obtain a stable recollement dictated by~$\family$. Since we will need some standard facts about parametrised stable recollements in a form which is not recorded in the literature, we first justify with the following lemma why we may quote the results from \cite{nineAuthorsII}.

\begin{lem}\label{lem:colimits_reflection_from_nonparametrised}
    The evaluation functor $\prod_{v\in\baseCat}\eval_v\colon\presentableStable{\baseCat}\rightarrow \prod_{v\in\baseCat}\presentable^L_{\mathrm{st}}$ preserves and reflects fibrewise (co)limits.
\end{lem}
\begin{proof}
    The proof proceeds exactly as in \cite{kaifNoncommMotives} where only the case  of $\baseCat$--perfect--stable categories was written, and so we only briefly indicate the steps. One first argues that $\presentable^R_{\baseCat-\mathrm{st}}\subset\widehat{\cat}_{\baseCat}$ is closed under fibrewise limits exactly as in \cite[Lem. 2.5.5]{kaifNoncommMotives}, but now without having to check $\kappa$--accessibility of functors. Next, one shows that $\presentableStable{\baseCat}\subset\widehat{\cat}_{\baseCat}$ preserves fibrewise limits as in \cite[Prop. 2.5.6]{kaifNoncommMotives}: since one already knows that $\presentable^L_{\mathrm{st}}\subset\widehat{\cat}$ preserves limits, the key points now which are handled \textit{mutatis mutandis} as in \textit{loc. cit.} is to show that $\lim_{i\in I}\udl{\sC}_i$ is $\baseCat$--stable and that for any map $\constant_I\udl{\D}\rightarrow \{\udl{\sC}_i\}_{i\in I}$ in $\func(I,\presentable^L_{\baseCat-\mathrm{st}})$, the induced map $\udl{\D}\rightarrow \lim_{i\in I}\udl{\sC}_i$ preserves $\baseCat$--coproducts. All in all, as in \cite[Lem. 2.5.10]{kaifNoncommMotives}, we get that for all $v\in\baseCat$, $\eval_v\colon \presentableStable{\baseCat}\rightarrow\presentable^L_{\mathrm{st}}$ preserves (co)limits by using that $\presentable^L_{\mathrm{st}},\presentable^R_{\mathrm{st}}\subset \widehat{\cat}$ are closed under limits and the identification $\presentableStable{\baseCat}\simeq (\presentable^R_{\baseCat-\mathrm{st}})\op$ from \cite[Prop. 2.2.15]{kaifNoncommMotives}. Given these ingredients, we may now immediately conclude as in \cite[Thm. 2.5.11]{kaifNoncommMotives}, using conservativity of $\prod_{v\in\baseCat}\eval_v$.
\end{proof}

\begin{prop}\label{prop:family_recollement}
    Let $\udl{\sC}\in\presentableStable{\baseCat}$. Then we have a stable recollement
    \begin{center}
        \begin{tikzcd}
             \udl{\func}_*(\udl{\widetilde{E\family}},\udl{\sC})\ar[rr, "p^*", hook] && \udl{\sC} \ar[rr, "j^*", two heads]\ar[ll, "p_*"', two heads, bend left = 40]\ar[ll, "p_!"', two heads, bend right = 40]&& \udl{\func}(\udl{E\family},\udl{\sC})  \ar[ll, "j_!"', hook, bend right = 40]\ar[ll, "j_*"', hook, bend left = 40]
        \end{tikzcd}
    \end{center}
    where the top sequence is identified with ${E\family}_+\wedge\udl{\sC}\rightarrow S^0\wedge\udl{\sC}\rightarrow \widetilde{E\family}\wedge \udl{\sC}$. Furthermore, the endofunctor $p^*p_!$ is given by tensoring with $\widetilde{E\family}\in\spc_{\baseCat,*}$. We thus also write $p_!$ as $\widetilde{E\family}\otimes-$.
\end{prop}
\begin{proof}
    By \cref{cons:tensoring_over_anima}, we already know that ${E\family}_+\wedge\udl{\sC}\rightarrow S^0\wedge\udl{\sC}\rightarrow \widetilde{E\family}\wedge \udl{\sC}$ is a cofibre sequence in $\presentableStable{\baseCat}$ by the universal property of tensorings. On the other hand, it was also explained that these may be identified with the functor categories upon passing to $\presentable^R_{\baseCat-\mathrm{st}}$, which gives the middle sequence in the statement where $j^*$ also admits a fully faithful right adjoint $j_*$; passing to left adjoints gives the top cofibre sequence in $\presentableStable{\baseCat}$ where $j_!$ is also fully faithful. Thus, by  \cite[Cor. A.1.10 (ii)]{nineAuthorsII}, we know that the top sequence is a fibre sequence as well.  Thus, now by \cite[Lem. A.2.8]{nineAuthorsII}, we also get the bottom two fibre and cofibre sequence. Note that we are able to quote these results from \cite{nineAuthorsII} because fully faithfulness and (co)fibre sequences in $\presentable^L_{\baseCat-\mathrm{st}}$ are checked fibrewise by virtue of \cref{lem:colimits_reflection_from_nonparametrised}. Finally, the statement describing $p^*p_!$ is because we have a cofibre sequence $j_!j^*\rightarrow\id_{\udl{\sC}}\rightarrow p^*p_!$, and $E\family_+\otimes-\colon \udl{\sC}\rightarrow\udl{\sC}$ is by definition given by $j_!j^*$.
\end{proof}

To end this subsection, we show that the abstract category $\udl{\func}(\udl{E\family},\udl{\sC})$ admits a concrete description as a subcategory of $\udl{\sC}$.

\begin{prop}\label{prop:descibing_kernel_of_geometric_fixed_points}\label{rmk:smashing_localisation_as_verdier_quotients}
    Let $\baseCat$ be inductive orbital, $\baseCat_{\family}\subseteq \baseCat$ be a sieve, and $\udl{\sC}\in\presentable^L_{\baseCat-\mathrm{st}}$. Let $\langle \family\rangle\subseteq \udl{\sC}$ be the smallest $\baseCat$--stable--presentable full subcategory containing all of $\udl{\sC}(v)$ for all $v\in\family$. 
    
    Then $\im(j_!\colon \udl{\func}(\udl{E\family},\udl{\sC})\rightarrow\udl{\sC})=\langle \family \rangle$. In particular, the map $\udl{\sC}\rightarrow \widetilde{E\family}\wedge \udl{\sC}$ may be identified as the Verdier quotient $\udl{\sC}\rightarrow \udl{\sC}/\langle \family\rangle$ which may be computed fibrewise.
\end{prop}
\begin{proof}
    For $\udl{\D}\in\presentableStable{\baseCat}$, we denote by $\myuline{\map}^{L,\family=0}(\udl{\sC},\udl{\D})$ for the full subspace of $\myuline{\map}^L(\udl{\sC},\udl{\D})$ consisting of those functors $F\colon \udl{\sC}\rightarrow \udl{\D}$ whose restriction $F(v)\colon \langle \family\rangle(v)\rightarrow\udl{\D}(v)$ is the zero morphism in $\presentable^L_{\mathrm{st}}$ for all $v\in\baseCat_{\family}$, i.e., it is defined to fit in the fibre sequence in $\spc_{\baseCat,*}$
    \[\myuline{\map}^{L,\family=0}(\udl{\sC},\udl{\D})\longrightarrow \myuline{\map}^{L}(\udl{\sC},\udl{\D}) \longrightarrow \myuline{\map}^{L}(\langle \family\rangle,\udl{\D}).\] Now, by repeating the proof for the last statement of \cite[Thm. 2.2.26]{parametrised_Poincare_duality} in the generality of inductive orbital categories, we get that $\udl{\sC}\rightarrow \widetilde{E\family}\wedge \udl{\sC}$ induces an equivalence $\myuline{\map}^{L}(\widetilde{E\family}\wedge \udl{\sC},\udl{\D})\xrightarrow{\simeq}\myuline{\map}^{L,\family=0}(\udl{\sC},\udl{\D})$. Therefore, $\widetilde{E\family}\wedge \udl{\sC}$ computes the cofibre in $\presentable^L_{\baseCat-\mathrm{st}}$ of the inclusion $\langle \family\rangle \hookrightarrow \udl{\sC}$. Since we know that, in the nonparametrised setting, the fibres of the  cofibres of such inclusions is the original localising subcategory itself,  we deduce from \cref{lem:colimits_reflection_from_nonparametrised} that $\udl{\func}(\udl{E\family},\udl{\sC})\simeq \fib(\udl{\sC}\rightarrow \widetilde{E\family}\wedge \udl{\sC}) \simeq \langle\family\rangle$, as required.
\end{proof}

We illustrate this principle with an example that we will  use  later in \cref{subsection:mackeYy_model_multiexcisive}.

\begin{example}\label{example:stable_recollement_mackey_functors}
    By \cite{glasmanStratified,wilsonSlices}, we know that the inclusion $\iota\colon \finite_{\baseCat_{\widetilde{\family}}}\subseteq \finite_{\baseCat}$  admits a right adjoint $p$ (note that this right adjoint does not exist before taking the finite coproduct cocompletion) which sends any coproduct factor in $\baseCat_{\family}=\baseCat\backslash\baseCat_{\widetilde{\family}}$ to $\emptyset$. On the other hand, we have the inclusion $j\colon \baseCat_{\family}\subseteq \baseCat$. Then  by \cite[Thm. 2.32]{glasmanStratified} and \cite[Nota. 1.33]{wilsonSlices}, we have the stable recollement
    \begin{equation}\label{eqn:recollment_mackey}
        \begin{tikzcd}
             \mackey(\baseCat_{\widetilde{\family}};\spectra)\ar[rr, "p^*", hook] && \mackey(\baseCat;\spectra)\ar[rr, "j^*", two heads]\ar[ll, "p_*"', two heads, bend left = 40]\ar[ll, "p_!"', two heads, bend right = 40]&& \mackey(\baseCat_{{\family}};\spectra). \ar[ll, "j_!"', hook, bend right = 40]\ar[ll, "j_*"', hook, bend left = 40]
        \end{tikzcd}
    \end{equation}
    Here, the essential image of the inclusion $p^*$ is given precisely by those $\baseCat$--Mackey functors which vanish on $\baseCat_{\family}\subseteq \baseCat$.
    
    Now, recall that $\mackey(\baseCat;\spectra)$ is a compactly rigidly generated category with the set of self--dual compact generators given by $\{\susps\yoneda(v) \: | \: v\in\baseCat\}$ where $\susps\yoneda\colon \spanCat(\baseCat)\rightarrow \func^{\times}(\spanCat(\baseCat),\spectra) = \mackey(\baseCat;\spectra)$ given by the Yoneda map. By the commutative square
    \[
    \begin{tikzcd}
        \spanCat(\baseCat_{\family}) \rar["\susps\yoneda"]\dar["j",hook] & \func^{\times}(\spanCat(\baseCat_{\family}),\spectra) = \mackey(\baseCat_{\family};\spectra) \dar["j_!",hook]\\
        \spanCat(\baseCat) \rar["\susps\yoneda"] & \func^{\times}(\spanCat(\baseCat),\spectra) = \mackey(\baseCat;\spectra)
    \end{tikzcd}
    \]
    coming from the naturality of the Yoneda embedding along left Kan extensions, we thus see that the inclusion $j_!\colon \mackey(\baseCat_{\family};\spectra)\hookrightarrow \mackey(\baseCat;\spectra)$ identifies with the inclusion of the localising subcategory generated by $\{\susps\yoneda(v)\: | \: v\in\baseCat_{\family}\}$, which is also the set coming from the images of the tensor unit under the  induction $v_! \colon \mackey(\baseCat_{/v};\spectra)\rightarrow \mackey(\baseCat;\spectra)$ for $v\in\baseCat_{\family}$.

    Therefore, all in all, if $\baseCat$ has a final object $t$, then $\mackey(\baseCat_{\family};\spectra)=\langle\family\rangle(t)$ for $\langle \family\rangle\subseteq \myuline{\spectra}$ as in \cref{prop:descibing_kernel_of_geometric_fixed_points}. By the same proposition, the stable recollement \cref{prop:family_recollement} evaluates to the stable recollement \cref{eqn:recollment_mackey} at $t\in\baseCat$.
\end{example}

\subsection{A recognition principle for spectra}\label{subsection:recognition_principle}
Let $\baseCat$ be an inductive orbital category. We  give a list of checkable axioms that guarantees that a $\baseCat$--category is symmetric monoidally equivalent to $\myuline{\spectra}$, essentially axiomatising the arguments in \cite{MNN17} and \cite[App. A]{CMMN2}. Throughout this subsection, let $\udl{\A}\in\calg(\presentableStable{\baseCat})$ be a $\baseCat$--stable presentably symmetric monoidal category. 

\begin{nota}\label{rmk:universal_map_from_spectra}
    Since $\myuline{\spectra}$ is the tensor unit in $\presentableStable{\baseCat}$, we obtain a unique symmetric monoidal $\baseCat$--colimit preserving map $\unitMapFromSpectra{\A}\colon\myuline{\spectra}_{\baseCat}\longrightarrow \underline{\A}$. By construction, the symmetric monoidal structure on $\myuline{\spectra}$ is fibrewise given by the Day convolution on Mackey functors.
\end{nota}

Recall from \cref{nota:EP_twiddle}  that when $\baseCat$ has a final object, we have the universal space $\widetilde{E\proper}$ by virtue of inductive orbitality of $\baseCat$.  Furthermore, recall also from \cref{prop:smashing_left_adjoint} that we have a Bousfield localisation $\widetilde{E\proper}\otimes-\colon \udl{\A}\rightarrow \widetilde{E\proper}\wedge\udl{\A}$ associated to the endofunctor $\widetilde{E\proper}\otimes-$ on $\udl{\A}$. In fact, this upgrades even to a smashing localisation since $\widetilde{E\proper}\in\udl{\spc}_*$ is an idempotent algebra.  By \cref{rmk:smashing_localisation_as_verdier_quotients}, this Bousfield localisation may also be identified with the Verdier quotient of $\udl{\A}$ by everything coming from $\udl{\A}(v)$ for $v\neq t$.

We write $\widetilde{E\proper}\in\udl{\A}$ for the object $\widetilde{E\proper}\otimes\unit_{\udl{\A}}\in\udl{\A}$. Moreover, we will also write $\widetilde{E\proper}_v$ for the analogous universal space for $\baseCat_{/v}$. Since $\baseCat_{/v}$ has a final object, $\widetilde{E\proper}_v$ will always make sense even if $\baseCat $ does not have a final object.
\vspace{1mm}

With these preliminaries, we now state the criteria to recognise when $\udl{\A}\simeq \myuline{\spectra}$.

\begin{hypothesis}\label{hypothesis:sphereInvertibleCategories}
    Suppose $\udl{\A}$  satisfies that:
    \begin{enumerate}[label=(\arabic*)]
        \item for every $v\in \baseCat$, the representables $\{f_!\unit_{\udl{\A}(u)}\}_{(f\colon u\rightarrow v)\in\baseCat_{/v}}$ in $\udl{\A}(v)\in\calg(\presentable^L_{\mathrm{st}})$ form a set of compact generators.  Note that  $\baseCat$--stability already implies that the representables are self--dual, and in particular, dualisable;
        \item for every $v\in\baseCat$, we have an equivalence $\mapsp_{\udl{\A}(v)}(\widetilde{E\proper}_{v},\widetilde{E\proper}_v) \simeq \sphere \in \spectra$.
    \end{enumerate}
\end{hypothesis}

\begin{thm}[Abstract spectra recognition]\label{thm:abstractSpectraRecognition}
    Suppose $\underline{\A}\in\calg(\presentableStable{\baseCat})$ satisfies \cref{hypothesis:sphereInvertibleCategories}. Then the symmetric monoidal functor $\unitMapFromSpectra{\A}\colon \myuline{\spectra}\rightarrow \underline{\A}$ from \cref{rmk:universal_map_from_spectra} is an equivalence.
\end{thm}

Since equivalences are checked by basechanging to $\baseCat_{/v}$ for all $v\in\baseCat$, we will assume for the rest of the subsection that $\baseCat$ has a final object $t\in \baseCat$.

\begin{obs}
    Writing $\A_{\baseCat}$ for $\udl{\A}(t)$ where $t\in \baseCat$ is the final object, we obtain a unique symmetric monoidal colimit--preserving functor $\inflation\colon \spectra\rightarrow \A_{\baseCat}$ by virtue of $\A_{\baseCat}$ being a presentably symmetric monoidal stable category. The lax symmetric monoidal right adjoint is  then formally given by $\mapsp_{\A_{\baseCat}}(\unit,-)$. 
\end{obs}

We write $\A_{\baseCat}[\widetilde{E\proper}]\in\presentable^L_{\mathrm{st}}$ for the category $(\widetilde{E\proper}\wedge\udl{\A})(t)$.

\begin{lem}\label{lem:geometricLayerEquivalence}
    Suppose $\udl{\A}$ satisfies \cref{hypothesis:sphereInvertibleCategories}. Then the composite $\spectra\xlongrightarrow{\inflation} \A_{\baseCat} \xlongrightarrow{\widetilde{E\proper}\otimes-} \A_{\baseCat}[\widetilde{E\proper}]$    is an equivalence with inverse given by the right adjoint $\mapsp_{\A_{\baseCat}}(\unit,\inclusion-)$.
\end{lem}
\begin{proof}
    As in \cite[Thm. 6.11]{MNN17}, we first note that $\A_{\baseCat}[\widetilde{E\proper}]$ is  generated by the compact unit $\widetilde{E\proper}$ because $\A_{\baseCat}[\widetilde{E\proper}]$ is the Verdier quotient by the proper representables and $\udl{\A}$ was generated by representables by \cref{hypothesis:sphereInvertibleCategories} (1). By the symmetric monoidal version of Schwede--Shipley, we just need to show that the endomorphism $\mathbb{E}_{\infty}$--ring spectrum of $\widetilde{E\proper}$ is given by $\sphere$. Since there is a unique $\mathbb{E}_{\infty}$--ring structure on $\sphere$, it suffices to show that the endomorphism spectrum is equivalent to $\sphere$, and this is precisely supplied by \cref{hypothesis:sphereInvertibleCategories} (2).
\end{proof}

\begin{defn}[Geometric fixed points]\label{definition: geometric fixed points}
    We define the {\it geometric fixed points functor} $\Phi^t$ to be the following symmetric monoidal colimit preserving functor
    \[\Phi^t\colon \A_{\baseCat} \xlongrightarrow{\widetilde{E\proper}\otimes-} \A_{\baseCat}[\widetilde{E\proper}] \subseteq \A_{\baseCat} \xlongrightarrow{\mapsp(\unit,-)} \spectra \]
    More generally, for any $u\in \baseCat$, we may define similarly the functor $\Phi^u\colon \A_{\baseCat}\rightarrow \spectra$ as the symmetric monoidal colimit preserving composition $\A_{\baseCat}\xlongrightarrow{u^*} \A_{\baseCat_{/u}}\xlongrightarrow{\Phi^u} \spectra$.
\end{defn}

For the next point, recall that $E\proper\in\func(\baseCat\op,\spc)$ is equivalent to the colimit $\colim_{v\in\baseCat_{\proper}}v_!\ast$ where $v_!\colon \func(\baseCat_{/v}\op,\spc)\rightarrow \func(\baseCat\op,\spc)$. Indeed, observe that $\ast\simeq \colim_{\baseCat_{\proper}} \yoneda\in \func(\baseCat_{\proper}\op,\spc)$ where $\yoneda$ is the Yoneda embedding. Then, since $E\proper\simeq f_!\ast$ by \cref{cons:universalSpaces} where $f_!\colon \func(\baseCat_{\proper}\op,\spc)\hookrightarrow \func(\baseCat\op,\spc)$ is the left Kan extension, the equivalence follows.

\begin{prop}\label{prop:jointConservativityGeomFixPoints}
    Let $M\in\A_{\baseCat}$. Then $M\simeq 0$ if and only if $\Phi^uM\simeq 0\in\spectra$ for all $u\in\baseCat$. 
\end{prop}
\begin{proof}
    Since $\baseCat$ is inductive orbital, $\baseCat$ has finitely many objects and morphisms, so we will prove the statement by induction. By induction, suppose that $v^*M\simeq 0$ for all  $v\in\baseCat_{\proper}$. Then $E\proper_+\otimes M \simeq \colim_{u\in \baseCat_{\proper}}(v_!v^*M) \simeq \colim_{v\in \baseCat_{\proper}} 0 \simeq 0$.
    Now the assumption together with \cref{lem:geometricLayerEquivalence} gives us that $\widetilde{E\proper}\otimes M \simeq 0$ since $\Phi^tM\simeq 0$. Hence, by the cofibre sequence
    $E\proper_+\otimes M \rightarrow M \rightarrow \widetilde{E\proper}\otimes M$,
    we get  $M\simeq 0$ as required.
\end{proof}

\begin{proof}[Proof of \cref{thm:abstractSpectraRecognition}.]
    Since $\baseCat$ is inductive orbital, it has finitely many objects and morphisms, and so we will prove the theorem by induction on the number of objects in $\baseCat$, the base case being clear. Suppose therefore that  $\unitMapFromSpectra{u^*\A}\colon u^*\myuline{\spectra}\rightarrow u^*\underline{\A}$ is an equivalence for all  $u\in\baseCat_{\proper}$ . Our goal is to show that symmetric monoidal colimit--preserving functor of nonparametrised categories $\unitMapFromSpectra{\A}\colon {\spectra}_{\baseCat}\rightarrow {\A}_{\baseCat}$
    is an equivalence. Denote by $R$ the right adjoint of $\unitMapFromSpectra{\A}\colon \myuline{\spectra}\rightarrow \underline{\A}$. 
    
    By \cite[Thm. 1.3]{bds16}, because both categories are generated under colimits by dualisable objects by \cref{hypothesis:sphereInvertibleCategories} (1), this right adjoint $R$ itself admits a right adjoint, and $\unitMapFromSpectra{\A}\dashv R$ satisfies the projection formula. Furthermore, $R$ is conservative because its left adjoint $\unitMapFromSpectra{\A}$ sends generators to generators, again by \cref{hypothesis:sphereInvertibleCategories} (1). Thus,  by the monadicity result of  \cite[Prop. 5.29]{MNN17}, we  know that  $\A_{\baseCat}\simeq \module_{\spectra_{\baseCat}}(R(\unit_{\A_{\baseCat}}))$.    Hence we have reduced our problem merely to showing that the adjunction unit
    \[\unit_{\spectra_{\baseCat}}\longrightarrow R\unitMapFromSpectra{\A}(\unit_{\spectra_{\baseCat}})\simeq R(\unit_{\A_{\baseCat}})\]
    is an equivalence.

    By induction, this map becomes an equivalence upon applying $\Phi^u$ for all $u\in\baseCat_{\proper}$. Hence, by \cref{prop:jointConservativityGeomFixPoints}, it suffices to check that it is an equivalence upon applying $\Phi^t$ where $t\in\baseCat$ is the final object. Since $\Phi^t(\unit_{\spectra_{\baseCat}})\simeq \sphere$ by symmetric monoidality, and we are looking at a map of commutative algebras, it suffices to see that there is an equivalence of spectra $\Phi^tR(\unit_{\A_{\baseCat}})\simeq \sphere$. This follows from the following computation:
    \begin{equation*}
        \begin{split}
            \Phi^tR(\unit_{\A_{\baseCat}}) & = \mapsp_{\spectra_{\baseCat}}(\unit_{\spectra_{\baseCat}}, \widetilde{E\proper}\otimes R\unit_{\A_{\baseCat}})\\
            &\simeq \mapsp_{\spectra_{\baseCat}}(\unit_{\spectra_{\baseCat}}, R(\widetilde{E\proper} \otimes\unit_{\A_{\baseCat}}))\\
            & \simeq \mapsp_{\A_{\baseCat}}(\unit_{\A_{\baseCat}}, \widetilde{E\proper})\\
            &\simeq \mapsp_{\A_{\baseCat}}(\widetilde{E\proper}, \widetilde{E\proper}) \simeq \sphere
        \end{split}
    \end{equation*}
    where the second equivalence is by the $\unitMapFromSpectra{\A}\dashv R$ projection formula and the last equivalence is by \cref{hypothesis:sphereInvertibleCategories} (2).
\end{proof}

\section{Parametrised perspective on Goodwillie calculus}\label{section:parametrised_Goodwillie}
\subsection{Subdiagonal  functors as a parametrised category}\label{subsection:subdiagonal_parametrised}
The goal of this subsection is to prove \cref{alphThm:parametrised_excisive_category_is_symmetric_monoidal} that the subdiagonal functors assemble into a presentably symmetric monoidal $\epiCategory_{d}$--stable category. This enhances the ordinary category $\excisive_{d}=\elbowCat{d,1}$ with the requisite parametrised structure, see \cref{thm:basic_properties_of_parametrised_excisive_category}. The restriction maps will be the cross--effects and the (co)inductions will be induced by the diagonal maps, as in \cref{cor:omnibus_multivariable_cross-effects}. 

As a first step, we assemble the subdiagonal functors for arbitrary $\sC$ and $\D$ to an $\epiCategory_d$--category.

\begin{cons}\label{cons:parametrised_category_of_excisive_functors}
    Fix stable categories $\sC$ and $\D$ and suppose that $\sC$ is small. We first construct an object in $\func_*\in\func(\epiCategory_{d},\cat)$ given schematically by 
    \begin{center}
        \begin{tikzcd}
            \func_*^1  & \func_*^2  \lar["\Delta^*"]&  \cdots \lar["\Delta^*"]& \func_*^{d-1} \lar["\Delta^*"] & \func_*^d\lar["\Delta^*"]\\
            {[1]} & {[2]} & \cdots & {[d-1]} & {[d]}
        \end{tikzcd}
    \end{center}
    where $\func_*^r$ denotes the category $\func_*(\sC^{\times r},\D)$. To this end, consider the composition $\udl{\delta}_{\sC}\colon\epiCategory_{d}\op \subset \finite\op \xlongrightarrow{\func(-,\sC)} \cat$.    This  implements the data of diagonals on $\sC$ represented schematically as 
    \begin{center}
        \begin{tikzcd}
            \sC \rar["\Delta"] & \sC^{\times2}  \rar["\Delta"]&  \cdots \rar["\Delta"]& \sC^{\times(d-1)} \rar["\Delta"] & \sC^{\times d}\\
            {[1]} & {[2]} & \cdots & {[d-1]} & {[d]}
        \end{tikzcd}
    \end{center}
    The object $\func_*\in\func(\epiCategory_{d},\cat)$ is then defined to be the composite
    \[\func_*\colon \epiCategory_{d}\xlongrightarrow{\udl{\delta}_{\sC}\op} \cat\op \xlongrightarrow{\func_*(-,\D)} \cat.\]

    Next, \cref{thm:key_pigeonhole_principle} (2) guarantees that we may restrict to the  subcategories $\excisive_{d,r}^{\lrcorner}\subseteq \func_*^{r}$ to get an object $\widetilde{\underline{\excisive}}_{d}^{\lrcorner}\in\func(\epiCategory_{d},\cat)$ given schematically by 
    \begin{center}
        \begin{tikzcd}
            \elbowCat{d,1}  & \excisive_{{d,2}}^{\lrcorner}  \lar["\Delta^*"]&  \cdots \lar["\Delta^*"]& \excisive_{d,d-1}^{\lrcorner} \lar["\Delta^*"] & \elbowCat{d,d}\lar["\Delta^*"]\\
            {[1]} & {[2]} & \cdots & {[d-1]} & {[d]}.
        \end{tikzcd}
    \end{center}
    Finally,  passing to the cross--effect (bi)adjoints from \cref{cor:omnibus_multivariable_cross-effects}, we get the desired $\epiCategory_{d}$\textit{--category  of $d$--excisive functors} $\underline{\excisive}_{d}^{\lrcorner}\in\func(\epiCategory_{d}\op,\cat)$.
\end{cons}

Our next goal is to prove that this object is $\epiCategory_d$--presentable stable.

\begin{lem}\label{corollary: Exc admits indexed colimits}
Let $\sC$ be a small stable category and let $\D$ be a cocomplete stable category. Then $\elbowCatUdl{d}(\sC,\D)$ admits arbitrary $\epiCategory_{d}$--indexed coproducts.
\end{lem}

\begin{proof}
By \cref{recollect:parametrised_colimits}, it suffices to show that the (bi)adjoints of
$$\crossEffect_f \colon \excisive^{\lrcorner}_{d,r}(\sC,\D) \to \excisive^{\lrcorner}_{d,k}(\sC,\D), \;\; f\colon [k] \twoheadrightarrow [r] \in \epiCategory_{d}$$
from \cref{cor:omnibus_multivariable_cross-effects} enjoy the double-coset formula. But this is precisely supplied by \cref{prop:doubcle_coset_truncated_to_d}. 
\end{proof}

With this in place, we will prove the desired presentable stability statement.

\begin{thm}\label{thm:basic_properties_of_parametrised_excisive_category}
Let $\sC$ be a small stable category and let $\D$ be a presentable stable category. Then the $\epiCategory_{d}$--category $\elbowCatUdl{d}(\sC,\D)$ is a presentable $\epiCategory_{d}$--stable category.
\end{thm}
\begin{proof}
    By \cref{prop:reflection_from_all_excisive_into_elbow_excisive}, we know that the $\epiCategory_d$--category is fibrewise presentable and stable. By \cref{corollary: Exc admits indexed colimits},  $\elbowCatUdl{d}(\sC,\D)$ admits parametrised colimits. Therefore, by \cite[Thm. 6.1.2]{kaifPresentable}, we get that it is $\epiCategory_{d}$--presentable. In fact, from \cref{cor:omnibus_multivariable_cross-effects} (1), since the ``induction'' functors $(\Delta\times\id)^*$ are biadjoint, we even get that $\elbowCatUdl{d}$ is $\epiCategory_d$--semiadditive. Thus, by definition (cf.~\cref{recollect:parametrised_stability}), it is an $\epiCategory_d$--stable category.
\end{proof}

In the presence of multiplicative structures on $\sC$ and $\D$, we may endow $\elbowCatUdl{d}(\sC,\D)$ with a symmetric monoidal structure, i.e., it may be enhanced to an object in $\cmonoid(\widehat{\cat}_{\epiCategory_d})$ which is furthermore fibrewise presentably symmetric monoidal.

\begin{cons}\label{cons:functorial_symmetric_monoidal_cross_effects_on_multireduced}
    Let $\sC\in\calg(\cat\exact)$ and $\D\in\calg(\presentable^L_{\mathrm{st}})$.  We first construct  $\udl{\func}_{\udl{\ast}}(\sC^{\times\bullet},\D)^{\otimes}\in\cmonoid(\widehat{\cat}_{\epiCategory_d})$. First, recall Lurie's functorial Dwyer--Kan localisation  from \cite[Cons. 4.1.7.1]{lurieHA} in the left diagram
    \begin{center}
        \begin{tikzcd}
            \wcat \rar["L",shift left = 1] & \cat \lar[hook,shift left= 1, "\mathrm{equiv}"] && \cmonoid(\wcat) \rar["L",shift left = 1] & \cmonoid(\cat) \lar[hook,shift left= 1, "\mathrm{equiv}"]
        \end{tikzcd}
    \end{center}
    where $\mathrm{equiv}$ assigns to a category $\sC$ the pair $(\sC,W)$ where $W\subset \sC$ is the wide subcategory consisting of equivalences. Importantly, the functor $L\colon\wcat\rightarrow\cat$ is product--preserving. Hence, it induces the adjunction on  the right. There is also a finite--product--preserving functor $\mathrm{all}\colon \cat\rightarrow \wcat$ sending a category $\sC$ to the pair $(\sC,\sC)$.     Given these, now consider the composite 
    \begin{equation*}
        \begin{split}
            (\dayConvolution(\sC^{\times\bullet},\D)^{\otimes},\mathrm{all})\colon \epiCategory_d\op\longrightarrow\finite\op &\xlongrightarrow{(\sC^{\times\bullet})^{\otimes}} \cmonoid(\cat)\\
            &\xlongrightarrow{\dayConvolution(-,\D)^{\otimes}} \cmonoid(\widehat{\cat}) \xhookrightarrow{\mathrm{all}}\cmonoid(\mathrm{W}\widehat{\cat})
        \end{split}
    \end{equation*}
    where $\dayConvolution(-,\D)^{\otimes}$ is the Day convolution symmetric monoidal structure on the functor category $\func(-,\D)$ covariantly functorial along left Kan extensions.
    Next, recall from \cref{corollary: cross-effect is symmetric monoidal} that  the Bousfield localisation $(-)^{\mathrm{red}}\colon \func(\sC^{\times r},\D)\rightarrow \func_{\udl{\ast}}(\sC^{\times r},\D)$ canonically refines to a symmetric monoidal Bousfield localisation. In particular, if we write $W_r$ for the collection of morphisms inverted by this Bousfield localisation (so that $\func_{\udl{\ast}}(\sC^{\times r},\D)\simeq \func(\sC^{\times r},\D)[W_r^{-1}]$), we know that it is a $\otimes$--ideal. Moreover, for any surjection $f\colon [k]\twoheadrightarrow [r]$, since the solid square in 
    \begin{center}
        \begin{tikzcd}
        \func(\sC^{\times k},\D) \rar[shift right=1, "\Delta^*_f"'] \dar[shift right = 1, "(-)^{\mathrm{red}}"', dashed]& \func(\sC^{\times r},\D) \lar[shift right=1, "\Delta_{f!}"',dashed]\dar[shift right = 1, "(-)^{\mathrm{red}}"', dashed] \\
        \func_{\udl{\ast}}(\sC^{\times k},\D) \rar[shift right=1, "\Delta^*_f"'] \uar[hook,shift right = 1]& \func_{\udl{\ast}}(\sC^{\times r},\D)\uar[hook,shift right = 1]\lar[shift right=1, "\Delta_{f!}"',dashed]
        \end{tikzcd}
    \end{center}
commutes, so does the dashed square of left adjoints. Thus, we see that the functor $\Delta_{f!}$ from $\func(\sC^{\times r},\D)$ to $\func(\sC^{\times k},\D)$ sends $W_r$ to $W_k$. Observe also that by virtue of the dashed commuting square, the functor $\Delta_{f!}\colon \func_{\udl{\ast}}(\sC^{\times r},\D)\rightarrow \func_{\udl{\ast}}(\sC^{\times k},\D)$ is the cross--effect functor $\crossEffect_f$ by construction. 

All in all, since the $W_{r}$'s are tensor ideals and are preserved by $\Delta_{f!}$, we  may pointwise pass to the subfunctor \[(\dayConvolution(\sC^{\times \bullet},\D)^{\otimes}, W_{\bullet})\subset  (\dayConvolution(\sC^{\times \bullet},\D)^{\otimes},\mathrm{all})\in\func(\epiCategory_d\op,\cmonoid(\mathrm{W}\widehat{\cat})).\] 
Postcomposing with $L\colon \func(\epiCategory_d\op,\cmonoid(\mathrm{W}\widehat{\cat}))\rightarrow \func(\epiCategory_d\op,\cmonoid(\widehat{\cat}))$ yields the desired enhancement of $\func_{\udl{\ast}}(\sC^{\times \bullet},\D)$ together with the cross--effect functors to an object in $\func(\epiCategory_d\op,\cmonoid(\widehat{\cat}))$. 

Now, write $\operadsCat$ for the category of operads, so that we have a nonfull inclusion $\cmonoid(\widehat{\cat})\subset \operadsCat$. This says that it is a property for an operad to be a symmetric monoidal category and it is a property for a map of operads between symmetric monoidal categories, i.e., lax symmetric monoidal functors, to be a symmetric monoidal functor. So far, we have an object $\udl{\func}_{\udl{\ast}}(\sC^{\times\bullet},\D)\in\func(\epiCategory_d\op,\cmonoid(\widehat{\cat}))$ using the cross--effects as the transition maps. By \cref{prop:symmetric_monoidal_refinement_of_elbow_cross_effects}, we obtain the suboperad of $\epiCategory_d$--categories $\elbowCatUdl{d}(\sC,\D)\subseteq \udl{\func}_{\udl{\ast}}(\sC^{\times\bullet},\D) \in \func(\epiCategory_d\op,\operadsCat)$. By the same proposition, this suboperad has the property that it is fibrewise a symmetric monoidal category, and that the transition maps of operads are in fact symmetric monoidal functors. Hence, we get that $\elbowCatUdl{d}(\sC,\D)$ is in fact an object in $\func(\epiCategory_d\op,\cmonoid(\widehat{\cat}))\subset \func(\epiCategory_d\op,\operadsCat)$, as required. By construction, $\elbowCatUdl{d}(\sC,\D)$ is fibrewise presentably symmetric monoidal.
\end{cons}

With these preparations, the main result of this subsection is now rather easy.

\begin{thm}\label{thm:parametrised_excisive_category_is_symmetric_monoidal}
Let $\sC\in\calg(\cat\exact)$ and $\D\in \calg(\presentable^L_{\mathrm{st}})$. The $\epiCategory_{d}$--stable category $\elbowCatUdl{d}(\sC,\D)$ admits an $\epiCategory_{d}$--presentably symmetric monoidal structure which is fibrewise given by the Day convolution. 
\end{thm}
\begin{proof}
    We already know that  it is fibrewise presentably symmetric monoidal by \cref{cons:functorial_symmetric_monoidal_cross_effects_on_multireduced}. It thus suffices now to show that tensoring preserves indexed coproducts. In other words, for every surjection $f\colon [k]\twoheadrightarrow [r]$, $G\in\elbowCat{d,r}(\sC,\D)$, and $F\in\elbowCat{d,k}(\sC,\D)$, the canonical projection formula map $\Delta_f^*(F\otimes \crossEffect_fG)\longrightarrow (\Delta_f^*F)\otimes G$ in $\elbowCat{d,r}$ is an equivalence. But this is precisely \cref{corollary: projection_elbow_cats}.
\end{proof}

\subsection{The main theorem and a summarising dictionary}\label{subsection:main_theorem}\label{subsection:summarising_dictionary}
We now collect all the ingredients and prove the main result of the article. We first need to connect the notion of geometric fixed points functor as defined in \cref{definition: geometric fixed points} to that of the localisations studied in \cref{subsec:generators_and_localisations}, which we do in \cref{lem:geometric_fixed_points_of_elbow_cats_is_spectra}.

\begin{nota}\label{nota:proper_family_excisive_context}
    For $r\leq d$, we write $\epiCategory_{d,r}\coloneqq (\epiCategory_d)_{/[r]}$. In general, for $r\leq b\leq d$, we have inclusions $\iota_{d,b,r}\colon \epiCategory_{b,r}\subseteq \epiCategory_{d,r}$ and $\elbowCat{b,r}\subseteq \elbowCat{d,r}$. By \cref{prop:smashing_left_adjoint}, we get a smashing localisation associated to $\iota_{d,b,r}$
    \begin{center}
        \begin{tikzcd}
            \Phi^{d,b,r}(-)\coloneqq  \widetilde{E\proper}_{d,b,r}\otimes(-) \colon \presentable^L_{\epiCategory_{d,r}\mathrm{-st}} \rar[shift left = 1] & \presentable^L_{\epiCategory_{b,r}\mathrm{-st}}.\lar[hook,shift left = 1]
        \end{tikzcd}
    \end{center}
    where $\proper_{d,b,r}$ is the family in $\epiCategory_{d,r}$ of those objects $[k]\twoheadrightarrow[r]$ such that $k> b$.

    In the case of $b=r$, it will be convenient to adopt the notation $\proper_{d,r}\coloneqq\proper_{d,r,r}$ and $\Phi^{d,r}\coloneqq \Phi^{d,r,r}$. That is, $\proper_{d,r}$ is the family of objects in $\epiCategory_{d,r}$ away from the final object $([r]=[r])$. 
\end{nota}

\begin{nota}
    For $r\leq d$, we write  $\elbowCatUdl{d,r}\in\calg(\presentable^L_{\epiCategory_{d,r}})$ for the image of $\elbowCatUdl{d}$ under the restriction functor $\calg(\presentable^L_{\epiCategory_d})\rightarrow\calg(\presentable^L_{\epiCategory_{d,r}})$ induced by the forgetful map $\epiCategory_{d,r}\rightarrow\epiCategory_d$.
\end{nota}

\begin{lem}\label{lem:identification_of_geometric_fixed_points} \label{lem:geometric_fixed_points_of_elbow_cats_is_spectra}
    Let $r\leq b\leq d$. Evaluating the adjunction unit $\elbowCatUdl{d,r}\rightarrow \Phi^{d,b,r}\elbowCatUdl{d,r}$ at $(f\colon[k]\twoheadrightarrow[r])\in\epiCategory_{d,r}$ yields the map $\excisiveApproximation^{\lrcorner}_{b,k}\colon \elbowCat{d,k}\rightarrow \elbowCat{b,k}$. In particular,  evaluating $ \elbowCatUdl{d,r}\rightarrow \Phi^{{d,r}}\elbowCatUdl{d,r}$ at the final object $([r]=[r])\in\epiCategory_{d,r}$ yields the functor $\excisiveApproximation_{\udl{1}}\colon \elbowCat{d,r}\rightarrow \excisive_{\udl{1}}\simeq \spectra$.
\end{lem}
\begin{proof}
    As in \cref{prop:descibing_kernel_of_geometric_fixed_points}, we write $\langle \proper_{d,b,r}\rangle\subseteq \elbowCatUdl{d,r}$  for the smallest presentable stable subcategory containing $\elbowCatUdl{d,r}(x)$ for all $x\in\epiCategory_{d,r}\backslash\epiCategory_{b,r}$. Then it is easily seen that $\langle \proper_{d,b,r}\rangle([r]=[r]) $ is the localising subcategory generated by the set as in \cref{prop:kernel_of_P_1}. Therefore, by \cref{prop:kernel_of_P_1,prop:descibing_kernel_of_geometric_fixed_points}, we obtain the desired conclusion. 
\end{proof}

We are at last ready to state and prove the main \cref{alphThm:excisive_functors_as_spectral_mackey}. For this, recall from \cref{thm:parametrised_excisive_category_is_symmetric_monoidal} and \cref{rmk:universal_map_from_spectra} that we have a unique symmetric monoidal left adjoint functor $\unitMapFromSpectra{}\colon \myuline{\spectra}\rightarrow\elbowCatUdl{d}$. 

\begin{thm}\label{thm:excisive_functors_as_spectral_mackey}
    Let $d$ be a positive integer. The symmetric monoidal functor $\unitMapFromSpectra{}\colon \myuline{\spectra}_d\rightarrow \elbowCatUdl{d}$ is an equivalence of objects in $\calg(\presentable^L_{\epiCategory_d\mathrm{-st}})$.
\end{thm}
\begin{proof}
    By \cref{lem:generators_of_elbow_cats}, we get that $\elbowCatUdl{d}$ satisfies \cref{hypothesis:sphereInvertibleCategories} (1), and by \cref{lem:geometric_fixed_points_of_elbow_cats_is_spectra}, $\elbowCatUdl{d}$ satisfies \cref{hypothesis:sphereInvertibleCategories} (2). Therefore, by \cref{thm:abstractSpectraRecognition}, we obtain the desired conclusion.
\end{proof}

For example, a straightforward consequence of \cref{thm:excisive_functors_as_spectral_mackey} is the following:

\begin{cor}
    Let $d\geq 1$,  $\sC$ be a small stable category, and $\D$ a presentable stable category. Then $\elbowCatUdl{d}(\sC,\D)$ admits a unique module structure in $\presentable^L_{\epiCategory_d}$ over $\elbowCatUdl{d}$. Moreover, if $\sC$ is  exact symmetric monoidal  and $\D$ is  presentably symmetric monoidal, then  $\elbowCatUdl{d}(\sC,\D)$  attains a unique $\elbowCatUdl{d}$--commutative algebra structure in $\presentable^L_{\epiCategory_d}$. 
\end{cor}
\begin{proof}
    By \cref{thm:basic_properties_of_parametrised_excisive_category}, $\elbowCatUdl{d}(\sC,\D)$ is $\epiCategory_d$--presentable--stable, and therefore, by \cite[Prop. 2.2.19]{kaifNoncommMotives}, it admits a unique module structure over the idempotent algebra $\myuline{\spectra}_d$, which by \cref{thm:excisive_functors_as_spectral_mackey}, is equivalent to $\elbowCatUdl{d}$. The statement about symmetric monoidal structures is now also clear.
\end{proof}

To end this subsection, we offer a picture that contextualises our work by comparing it to the equivariant situation. To this end, let us first recall the situation in the latter setting. 

Let $G$ be a finite group, $N\triangleleft G$ be a normal subgroup, and  $H, K\leq G$ be subgroups such that $N\leq H$, and let $\overline{H}\leq G/N$ denote the image of $H $ under $G\twoheadrightarrow G/N$.  There are then two orthogonal and complementary perspectives on how to present the datum of a genuine $G$--spectrum, namely via its \textit{genuine fixed points} and \textit{geometric fixed points}, respectively. The categories governing these data may be summarised by the following diagram of indexing categories
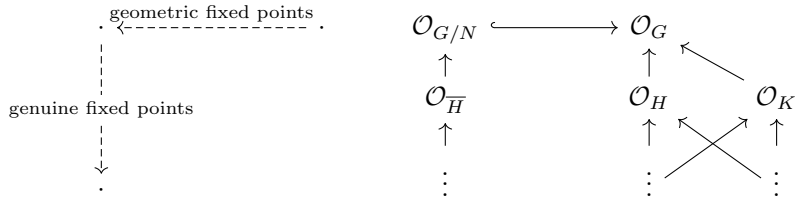
\[
\begin{tikzcd}[row sep = small]
    \cdot\ar[dd, dashed, "\text{genuine fixed points}"' description]&&&\cdot \ar[lll, dashed, "\text{geometric fixed points}"'] & \orbit_{G/N} \ar[rr,hook] && \orbit_G \\
     &&& &\orbit_{\overline{H}}\uar&  &\orbit_H \ar[u]& \orbit_K\ar[ul]\\
     \cdot&&& & \vdots \uar&& \vdots \uar\ar[ur]& \vdots \uar\ar[ul]
\end{tikzcd}
\]
where we have used the standard identification $\orbit_H\simeq (\orbit_G)_{/G/H}$ to obtain the upward--pointing arrows. In this way, we see that the horizontal direction governs geometric fixed point functorialities, and the vertical direction governs the restriction/induction/genuine fixed points functorialities.

Coming back to the setting of  calculus, the analogue is then the following: 

\[
\begin{tikzcd}[row sep =small]
    \cdot \ar[dddd,dashed, "\text{genuine}"description]&  &\ar[ll,dashed, "\text{geometric}"']\epiCategory_2 \rar[hook]&  \epiCategory_3\rar[hook] &  \epiCategory_4 \rar[hook]& \epiCategory_{5}\rar[hook] & \cdots\\
    &&\epiCategory_{2,2}\uar\rar[hook] & \epiCategory_{3,2} \uar\rar[hook]& \epiCategory_{4,2}\uar \rar[hook]& \epiCategory_{5,2}\uar \rar[hook]& \cdots\\
    && & \epiCategory_{3,3} \uar\rar[hook]& \epiCategory_{4,3}\uar \rar[hook]& \epiCategory_{5,3}\uar \rar[hook]& \cdots\\
    &&& & \epiCategory_{4,4}\uar \rar[hook]& \epiCategory_{5,4}\uar \rar[hook]& \cdots\\
    \cdot&&   & & & \epiCategory_{5,5}\uar \rar[hook]& \cdots
\end{tikzcd}
\]
where now the vertical direction involves the cross--effects and restrictions along the diagonal as in \cref{thm:key_pigeonhole_principle}, and  the horizontal direction involves the excisive approximations from \cref{lem:identification_of_geometric_fixed_points} (and hence also the derivatives since these may be computed by a combination of excisive approximations and cross--effects). It is worth emphasising here that unlike in the equivariant situation, we do \textit{not} have an equivalence $\epiCategory_{d,r}\coloneqq (\epiCategory_d)_{/[r]}\simeq \epiCategory_r$, and this accounts for some fundamental qualitative differences between the two settings.

Combining \cref{lem:identification_of_geometric_fixed_points}  with \cref{thm:excisive_functors_as_spectral_mackey} and \cref{example:stable_recollement_mackey_functors}, the stable recollement associated to the inclusion $\epiCategory_{b,r}\subseteq \epiCategory_{d,r}$ for $r\leq b\leq d$ looks like

\begin{equation}\label{eqn:multiexcisive_stable_recollement_with_mackey_fibre}
        \begin{tikzcd}
             \elbowCat{b,r}\ar[rr, hook] && \elbowCat{d,r}\ar[rr,  two heads]\ar[ll, two heads, bend left = 40]\ar[ll, "\excisiveApproximation^{\lrcorner}_{b,r}"' description, two heads, bend right = 40]&& \mackey(\epiCategory_{d\backslash b, r};\spectra). \ar[ll,  hook, bend right = 40]\ar[ll,  hook, bend left = 40]
        \end{tikzcd}
    \end{equation}
where $\excisiveApproximation^{\lrcorner}_{b,r}$ is the ``geometric fixed points'' functor. This specialises to the usual stable recollement 
\begin{equation*}
        \begin{tikzcd}
             \excisive_{d-1}\ar[rr,  hook] && \excisive_{d}\ar[rr, "\partial_d", two heads]\ar[ll, "\excisiveApproximation^{d-1}"' description, two heads, bend left = 40]\ar[ll, "\excisiveApproximation_{d-1}"' description, two heads, bend right = 40]&&\spectra^{B\Sigma_d}. \ar[ll, "((-)\otimes(\bullet)^{\otimes d})_{h\Sigma_d}"', hook, bend right = 40]\ar[ll, "((-)\otimes(\bullet)^{\otimes d})^{h\Sigma_d}", hook, bend left = 40]
        \end{tikzcd}
    \end{equation*}
associated to Kuhn's squares \cite{Kuhn04} by setting $b=d-1$, $r=1$ and $P^{d-1}$ is McCarthy's dual calculus approximation \cite{McCarthy99}.

We now summarise the foregoing discussion in a dictionary translating between genuine equivariant stable homotopy theory and stable Goodwillie calculus, extending the partial one of \cite{ABHS24}. 

\begin{figure}[ht]
    \begin{tabular}{|c|c|}
        \hline
         Equivariant homotopy theory& Goodwillie calculus \\ \hline
         \hline
         $G$ & $d$\\
         \hline
         $e\leq H\leq G$& $[d]\geq [r]\geq [1]$\\
         \hline
         $\orbit_G$ & $\epiCategory_d$\\
         \hline
         $\orbit_H\simeq (\orbit_G)_{/G/H}$ for $H\leq G$ & $(\epiCategory_d)_{/[r]}$ for $r\leq d$\\
         \hline
         $\spectra_H$ for $H\leq G$ & $\elbowCat{d,r}$ for $r\leq d$\\
         \hline
         $\res^G_H\colon \spectra_G\rightarrow \spectra_H$ for $H\leq G$ & \begin{tabular}{c}$\crossEffect_f\colon \excisive_d\rightarrow \elbowCat{d,r}$ along \\the unique surjection $f\colon [r]\twoheadrightarrow [1]$\end{tabular}\\
         \hline
         $\ind^G_H\colon \spectra_H\rightarrow \spectra_G$ for $H\leq G$ & \begin{tabular}{c}$\Delta^*_f\colon \elbowCat{d,r}\rightarrow \excisive_d$ along\\the unique surjection $f\colon [r]\twoheadrightarrow [1]$\end{tabular}\\
         \hline
         $\ind^G_H\dashv \res^G_H \dashv \ind^G_H$ & $\Delta_f^*\dashv \crossEffect_f\dashv \Delta_f^*$\\
         \hline
         $\widetilde{\Phi}^K\colon \spectra_H\rightarrow\spectra_{W_HK}$ for $K\leq H\leq G$ & $\excisiveApproximation^{\lrcorner}_{b,r}\colon \elbowCat{d,r}\rightarrow \elbowCat{d,r}$ for $r\leq b\leq d$ \\
         \hline
         $\Phi^G\colon \spectra_G\rightarrow \spectra$ & $P_1\colon \excisive_d\rightarrow \spectra$\\
         \hline
         $\sphere_G$ & $\unit_{d}=P_d\susps\loops(-)$\\ \hline
         $A(H)$ for $H\leq G$ & \begin{tabular}{c}$A(d,r)$ for $r\leq d$, and $A(d,1)=A(d)$, \\cf. \cref{cons:completion_at_ideals}\end{tabular}\\
         \hline
         $\forget\colon \spectra_G\rightarrow \spectra^{BG}$ & $\partial_d \colon \excisive_d\rightarrow \spectra^{B\Sigma_d}$\\
         \hline
         $(-)^G\colon \spectra_G\rightarrow \spectra$ & $\eval_{\sphere}\colon \excisive_d \rightarrow \spectra$\\
         \hline
         A family $\family$ of subgroups of $G$ & The interval $[d]\geq [r]$ for some $d\geq r$ \\
         \hline
    \end{tabular}
    \caption{A genuine-equivariant-to-Goodwillie-calculus dictionary.}\label{table:dictionary}
\end{figure}

\subsection{Application: Mackey model for all multiexcisive functors}\label{subsection:mackeYy_model_multiexcisive}

Because the equivalence in \cref{thm:excisive_functors_as_spectral_mackey} is that of parametrised categories, it is therefore compatible with the natural stratifications on both sides. In this short subsection, we show how to exploit this observation to deduce a Mackey model for all multiexcisive functors without any cutoff degrees of excisiveness. Of some note will be the disappearance of the subdiagonality conditions when one passes to infinity.

\begin{nota}
    We write $\epiCategory_{\infty,r}$ for $\epiCategory_{/[r]}$, so that $\epiCategory_{\infty,1}=\epiCategory$. Let $\D\in\presentable^L_{\mathrm{st}}$. We write $\mackey_{\mathrm{fs}}(\epiCategory_{\infty,r};\D)\subseteq \mackey(\epiCategory_{\infty,r};\D)$ for the full subcategory of finitely supported Mackey functors. By applying $\D\otimes-$ to the stable recollement from \cref{example:stable_recollement_mackey_functors}, we get inclusions $p^*\colon \mackey(\epiCategory_{d,r};\D)\hookrightarrow \mackey(\epiCategory_{d+1,r};\D)$  with essential image those $\epiCategory_{d+1,r}$--Mackey functors that vanish on any $([d+1]\twoheadrightarrow[r])\in\epiCategory_{d+1,r}$. Taking the sequential colimit along the  inclusions $\mackey(\epiCategory_{d,r};\D)\xhookrightarrow{p^*} \mackey(\epiCategory_{d+1,r};\D)\subseteq \mackey(\epiCategory_{\infty,r};\D)$, we obtain an identification $\mackey_{\mathrm{fs}}(\epiCategory_{\infty,r};\D)\simeq \colim_d\mackey(\epiCategory_{d,r};\D)$.

    We also write $\multiexcisive(\sC^{\times r},\D)\subseteq \func_{\udl{\ast}}(\sC^{\times r},\D)$ for the full subcategory identified as the sequential colimit of $\colim_d\elbowCat{d,r}(\sC,\D)$ where the inclusions $\elbowCat{d,r}(\sC,\D)\subseteq \elbowCat{d+1,r}(\sC,\D)$ are the colimit preserving right adjoints of $\excisiveApproximation^{\lrcorner}_{d,r}$. Because the subdiagonality condition involves larger and larger $d$, we see that $\multiexcisive(\sC^{\times r},\D)$ is just the full subcategory of functors which are reduced and finite excisive in each of the $r$ variables.
\end{nota}

\begin{cor}\label{cor:finitely_supported_equivalence}
    Let $\D\in\presentable^L_{\mathrm{st}}$. There is  an equivalence \[\multiexcisive((\spectra^{\omega})^{\times r},\D)\simeq \mackey_{\mathrm{fs}}(\epiCategory_{\infty,r};\D).\]
\end{cor}
\begin{proof}
    By \cref{thm:excisive_functors_as_spectral_mackey} and \cref{lem:identification_of_geometric_fixed_points}, we see that the fibrewise colimit preserving inclusions $\Phi^{d,b,r}\myuline{\spectra}_{d,r}\subseteq \myuline{\spectra}_{d,r}$ and $\Phi^{d,b,r}\elbowCatUdl{d,r}\subseteq \elbowCatUdl{d,r}$ are compatible. Thus we get  $\multiexcisive((\spectra^{\omega})^{\times r},\D) \simeq \colim_d\elbowCat{d,r}\otimes \D\simeq \colim_d\mackey(\epiCategory_{d,r};\D)\simeq \mackey_{\mathrm{fs}}(\epiCategory_{\infty,r};\D) $ as required.
\end{proof}

\begin{rmk}\label{rmk:multivariable_classification_not_currying}
    Alternatively, one could try to obtain the identifications for $r\geq 2$ from above by applying the identification for $r=1$ iteratively via the equivalence of $\multiexcisive((\spectra^{\omega})^{\times r},\D)$ to the category $ \multiexcisive((\spectra^{\omega})^{\times (r-1)},\excisive(\spectra^{\omega},\D))$. This will then yield the equivalence
    \[\multiexcisive((\spectra^{\omega})^{\times r},\D) \simeq \func^{\mathrm{multi}-\times}_{\mathrm{fs}}(\spanCat(\finite_{\epiCategory})^{\times r},\D)\] where the right hand side are finitely supported functors which preserve finite products in each of the $r$ variables. One might hope that this immediately implies \cref{cor:finitely_supported_equivalence} since $\epiCategory_{\infty,r}\simeq \epiCategory^{\times r}$. However, this is not as direct as one might first surmise because it is not true that $\finite_{\epiCategory_{\infty,r}}$ is equivalent to $(\finite_{\epiCategory})^{\times r}$.
\end{rmk}

\subsection{Application: a special case of the Segal conjecture in calculus}\label{subsection:segal}

One of the virtues of the symmetric monoidal equivalence $\mackey(\epiCategory_d;{\spectra})\simeq \excisive_{d}$ is that we can transport the natural $t$--structure on the left to one that can be fruitfully used to filter object on the right in a way which interacts in known ways with the symmetric monoidal structures on both sides. Furthermore, the full parametrised equivalence $\elbowCatUdl{d}\simeq \myuline{\spectra}_d$ from our main theorem also guarantees that the natural stratifications on boths sides coincide.  The purpose of this final subsection is to take advantage of  these points to begin an exploration of a Segal conjecture type statement in Goodwillie calculus.

Our first order of business is to formulate the statement of a ``Segal conjecture in calculus'', following the example of \cite{greenleesMayMU}. This rests on the following two constructions:

\begin{cons}[Goodwillie--Burnside augmentation ideal]\label{cons:completion_at_ideals}
    As already used in the main body of the article, there is a symmetric monoidal functor $\crossEffect_d\colon \excisive_d\rightarrow\spectra$ admitting a right adjoint $\Delta^*$ with a preferred  lax symmetric monoidal enhancement. This is analogous to the adjunction $\res^G_e\colon\spectra_G\rightleftharpoons \spectra \cocolon \coind^G_e$ from equivariant homotopy theory. We therefore obtain a map in $\calg(\spectra)$
    \[\mapsp_{\excisive_d}(\unit_d,\unit_d)\longrightarrow\mapsp_{\spectra}(\unit,\unit)\simeq \sphere\] taking $\pi_0$ of which induces a ring map $A(d)\rightarrow\mathbb{Z}$ where $A(d)$ is the Goodwillie-Burnside ring of \cite{ABHS24}. We call the kernel of this map $I(d)$ the \textit{Goodwillie-Burnside augmentation ideal}. 
    
    Since $A(d)$ is Noetherian, $I(d)$ is finitely generated, and so let us write $I(d)=(a_1,\ldots,a_m)$. Following \cite{greenleesMayMU}, we may define the $I(d)$-adic completion of an object in $\excisive_d$ as follows: let $K(a_1,\ldots,a_m) \coloneqq \fib(\unit_d\rightarrow\unit_d[a_1^{-1},\ldots,a_m^{-1}])$; for $F\in\excisive_d$, we then define its $I(d)$-adic completion as
    \[F\completion{I(d)} \coloneqq \hhom(K(a_1,\ldots,a_m),F)\] where $\hhom$ is the internal hom in $\excisive_d$. 
\end{cons}

\begin{cons}[Borelifications]
    For $F\in\excisive_d$, we now construct a map
    \begin{equation}\label{eqn:completion_map}
        F(-)\completion{I(d)}\longrightarrow (\partial_dF\otimes(-)^{\otimes d})^{h\Sigma_d}
        \end{equation}
    in $\excisive_d$.     As a notation suggestive of the equivariant analogue, we will write the $d$--th derivative functor $\partial_d\colon \excisive_d\rightarrow\spectra^{B\Sigma_d}$ also as $\beta^*$, so that $\forget\circ\beta^*\simeq \crossEffect_d$ where $\forget\colon \spectra^{B\Sigma_d}\rightarrow\spectra$ is the forgetful functor. This admits a fully faithful right adjoint $\beta_*$ which sends $X\in\spectra^{B\Sigma_d}$ to the functor $(X\otimes(-)^{\otimes d})^{h\Sigma_d}$. The adjunction unit thus gives us a map $F(-)\longrightarrow (\partial_dF\otimes(-)^{\otimes d})^{h\Sigma_d}$.

    On the other hand, for any $X\in\spectra^{B\Sigma_d}$, $\beta_*X\in\excisive_d$ is already $I(d)$--complete in the sense that the map $\beta_*X\rightarrow (\beta_*X)\completion{I(d)}$ is an equivalence. To wit, first note that $\beta^*(\unit_d[a_1^{-1},\ldots,a_m^{-1}])\simeq 0$ since $\forget\colon \spectra^{B\Sigma_d}\rightarrow \spectra$ is conservative and $\crossEffect_d(\unit_d[a_1^{-1},\ldots,a_m^{-1}])\simeq \sphere[0^{-1}]\simeq 0$, by definition of $I(d)$. Hence, $\beta^*K(a_1,\ldots,a_m)\rightarrow\beta^*\unit_d$ is an equivalence. Now, by symmetric monoidality of $\beta^*$, we see that $(\beta_*X)\completion{I(d)}\simeq \beta_*\hhom(\beta^*K(a_1,\ldots,a_m),X)\simeq \beta_*X$, as claimed.

    Therefore, the map $F(-)\rightarrow (\partial_dF\otimes(-)^{\otimes d})^{h\Sigma_d}$ factors through the $I(d)$--adic completion as in \cref{eqn:completion_map}. This completes the construction.
\end{cons}

We are now ready to state the theorem.

\begin{thm}\label{thm:segal_conjecture_calculus}
    Let $d=p$ be a prime number. Then the transformation
    \[P_d\susps\Omega^{\infty}(-)\completion{I(d),p}\longrightarrow \big(((-)^{\otimes d})^{h\Sigma_d}\big)\completion{p}\] of $d$--excisive functors $\spectra^{\omega}\rightarrow\spectra$ is an equivalence.
\end{thm}

We will prove this by reducing it to the usual Segal conjecture in equivariant homotopy theory for the group $C_p$ via a stratification argument and exploiting the idiosyncrasies of the augmentation ideal $I(d)$. The key new input here is an analysis of the interaction between the stratification functors and the $I(d)$--adic completion. This is captured by \cref{prop:geometric_fixed_points_and_completions}, which we work towards now.

We will exploit the fact that $P_1\colon \excisive_d\rightarrow \excisive_1$ is a smashing localisation with a nice description of the idempotent algebra, namely: 
\begin{cons}\label{cons:cellular_formula_P_1}
    Goodwillie gives us a functorial formula for $P_1$, namely $P_1F\simeq \colim_n\Omega^nF\Sigma^n$. Now,  $\Sigma^*\colon \excisive_d\rightarrow \excisive_d$ is an automorphism which in fact is a part of an automorphism of the presentable $\epiCategory_d$--category $\elbowCatUdl{d}\simeq \myuline{\spectra}_d$. By the equivalence $\udl{\func}^L(\myuline{\spectra}_d,\myuline{\spectra}_d)\simeq \myuline{\spectra}_d$, we thus see that there exists an invertible object $T\in \excisive_d$ such that $\Sigma^*(-)\simeq T\otimes-\colon \excisive_d\rightarrow\excisive_d$. By inputting $F$ to be the tensor unit $\unit_d$ in Goodwillie's formula, we thus see that $T\simeq \Sigma^*\unit_d$ and we obtain a map $\unit_d\rightarrow \Omega T$ such that the sequential diagram in the formula is given by $F\simeq \unit_d\otimes F \rightarrow \Omega T\otimes F \rightarrow (\Omega T)^{\otimes 2}\otimes F\rightarrow (\Omega T)^{\otimes 3}\otimes F \rightarrow \cdots$. Importantly, we have also that $T(-)\simeq P_d\susps\loops (\Sigma-)$  and so $T^m\coloneqq T^{\otimes m}\simeq P_d\susps\loops (\Sigma^m-)$. All in all, writing $E_n\coloneqq (\Omega T)^{\otimes n}$ and $E\coloneqq \colim_nE_n$, we learn that $P_1F\simeq E\otimes F$. 
\end{cons}

The strategy is to carry out a connectivity argument to show that $P_1(-)\simeq E\otimes -$ commutes with nice inverse limits. The key input from calculus will enter by way of \cref{prop:connectivity_of_suspended_unit} afforded by our full parametrised identification $\myuline{\spectra}_d\simeq \elbowCatUdl{d}$, which, in particular, guarantees that the natural stratifications on both sides agree. To this end, let us first record the following known fact about $t$--structures on spectral Mackey functors.

\begin{lem}\label{lem:geom_fix_pt_detect_connectivity}
    Let $\baseCat$ be an inductive orbital category. Then an object $X\in\mackey(\baseCat;\spectra)$ is connective if and only if $\Phi^uX\in\spectra$ is connective for all $u\in\baseCat$. 
\end{lem}
\begin{proof}
    For the forward implication, recall that the connective part $\tau_{\geq 0}\mackey(\baseCat;\spectra)$ is generated under colimits by the orbits. But these orbits are suspension spectra, and so since the geometric fixed points commute with suspension spectra and preserve all colimits, we thus see that they send  $\tau_{\geq 0}\mackey(\baseCat;\spectra)$ to the subcategory of $\spectra$ generated under colimits by the sphere, i.e., connective spectra. For the converse implication, since connectivity of $X$ is a pointwise statement, we just need to check that $X(u)\in\spectra$ is connective for all $u\in\baseCat$. Thus, by induction along proper families as in \cref{nota:EP_twiddle}, it suffices to assume that $\baseCat$ has a final object $t$ and that we already know that $X(u)$ is connective for all $u\neq t\in\baseCat$. To see that $X(t)$ is connective, recall that we have the cofibre sequence in $\spectra$
    \[(E\proper_+\otimes X)(t) \longrightarrow X(t) \longrightarrow (\widetilde{E\proper}\otimes X)(t) = \Phi^tX.\]  By assumption, the right hand term is connective. By the paragraph before \cref{prop:jointConservativityGeomFixPoints}, the left hand term is a colimit whose terms are $X(u)$ for $u\neq t$, and so is itself connective. Then so is $X(t)$, as required.
\end{proof}

\begin{prop}\label{prop:connectivity_of_suspended_unit}
    Let $d\geq 1$. Then the cofibre of  $\Omega^mT^m\rightarrow \Omega^{m+1}T^{m+1}$ is  $m$-connective.
\end{prop}
\begin{proof}
    Writing $C$ for the cofibre, it suffices by \cref{lem:geom_fix_pt_detect_connectivity} to show that $\Phi^{[k]}C$ is $m$--connective for all $1\leq k\leq d$. We then get the following equivalences in spectra:
    \begin{equation*}
        \begin{split}
            \Phi^{[k]}C &\simeq \cofib(\Omega^m\Sigma^{m*}\Phi^{[k]}\unit_d\rightarrow \Omega^{(m+1)}\Sigma^{(m+1)*}\Phi^{[k]}\unit_d)\\
            &\simeq \cofib\big(\Omega^m(\Sigma^m\sphere)^{\otimes k}\rightarrow \Omega^{m+1}(\Sigma^{m+1}\sphere)^{\otimes k}\big)
        \end{split}
    \end{equation*}
    where the first equivalence is since $\Sigma^{m*}$ is a parametrised automorphism of $\elbowCatUdl{d}$ and so commutes with $\Phi^{[k]}$, and the second equivalence is since $\Phi^{[k]}\unit_d\simeq \crossEffect_kP_{k}\susps\loops(-)\simeq (-)^{\otimes k}$ by \cref{lem:geometric_fixed_points_of_elbow_cats_is_spectra}. We thus see that $\Phi^{[k]}C$ is a cofibre of $(k-1)m$--connective terms when $k>1$ and is a cofibre of an equivalence when $k=1$. Thus, it is itself $m$--connective, as was to be shown.
\end{proof}

Using this, we roughly follow \cite[Lem. 4.4]{greenleesStableMapsFreeGSpaces} from the equivariant setting in showing that geometric fixed points ``commute'' with completions at the augmentation ideal. For this, we start with the following key classical connectivity argument (cf.~\cite[Thm. 15.2]{AdamsBlueBook} for a variant).

\begin{lem}\label{lem:connectivity_lemma}
    Suppose we have an inverse system $\{Y_n\}_n\in \excisive_d$ of uniformly bounded below objects. Then the map $P_1\lim_nY_n\rightarrow \lim_nP_1 Y_n$ is an equivalence.
\end{lem}
\begin{proof}
    As is standard, since inverse limits may be written as the fibre of a map between infinite products in stable categories, it suffices to deal with the case of infinite products. Using the notations from \cref{cons:cellular_formula_P_1},  let ${C}_k\coloneqq \cofib({E}_k\rightarrow {E})$. We claim that  $C_k$ is $k$--connective. Given this, we consider the cofibre sequence
    \[E_k\otimes \prod_nY_n \simeq \prod_nE_k\otimes Y_n \longrightarrow \prod_nE\otimes Y_n \longrightarrow \prod_nC_k\otimes Y_n\] where we have used that the $E_k$'s are invertible for the first equivalence. Since the $Y_n$'s were uniformly bounded below and the connectivities of the $C_k$'s were increasing strictly, we see that the connectivity of $\prod_nC_k\otimes Y_n$ increases strictly as $k$ increases. This is because the tensor product of a $k$--connective and $m$--connective object under the Day convolution in $\mackey(\epiCategory_d;\spectra)$ is $(m+k)$--connective, and the infinite product $\prod_n$ and the $t$--structure on $\mackey(\epiCategory_d;\spectra)$ is given pointwise. Hence, $\colim_k\prod_nC_k\otimes Y_n\simeq 0$. This thus yields that $E\otimes \prod_nY_n\simeq \colim_kE_k\otimes \prod_nY_n\rightarrow \prod_nE\otimes Y\simeq \colim_k\prod_nE\otimes Y$ is an equivalence, as desired.

    Now, to prove the claim, recall from the preceding construction that $E_m\simeq \Omega^m(P_d\susps\loops(\Sigma-))^{\otimes m}$. By \cref{prop:connectivity_of_suspended_unit}, the maps $E_m\rightarrow E_{m+1}$ are $m$-connective, and therefore the map $E_m\rightarrow E$ is $m$--connective, as claimed.
\end{proof}

\begin{prop}\label{prop:geometric_fixed_points_and_completions}
    Let $I\subseteq A(d)$ be a finitely generated ideal and $F\in\excisive_d$ be bounded below (i.e., $(\crossEffect_kF)(\sphere,\ldots,\sphere)$ is bounded below for all $k$). Then $P_1(F\completion{I})\simeq (P_1X)\completion{P_1I}\in\spectra$. 
\end{prop}
\begin{proof}
    Since the $I$--completion may be computed by taking successive completion with respect to the generators of $I$, it suffices to deal with the case when $I$ is a principal ideal. So let $\alpha\in A(d)$ and $F\in\excisive_d$ be bounded below. Then we need to show that $P_1(F\completion{\alpha})\simeq (P_1F)\completion{P_1\alpha}\in\spectra$. To see this, recall that   $F\completion{\alpha}\coloneqq \lim_nF/\alpha^n$ is an inverse limit of uniformly bounded below objects. Hence, we get equivalences 
    \[P_1(F\completion{\alpha})\simeq P_1\lim_nF/\alpha^n\simeq \lim_nP_1 (F/\alpha^n)\simeq \lim_nP_1F/(P_1\alpha)^n\simeq (P_1F)\completion{P_1\alpha}\]
    where the second equivalence is by \cref{lem:connectivity_lemma} and the third equivalence is since $P_1\colon\excisive_d\rightarrow\spectra$ is exact and symmetric monoidal. This completes the proof.
\end{proof}

Next, we provide the two computational inputs needed in the proof of  \cref{thm:segal_conjecture_calculus}. To state the first of these, recall the Tate construction $(-)^{t\family_{\mathrm{nt}}(\Sigma_p)}$ with respect to the family of nontransitive subgroups from \cite[Ex. 10.4, Defs. 10.16 and 10.21]{ABHS24}.

\begin{lem}\label{lem:p-completion_of_tates}
    There is a canonical equivalence ${\sphere^{t\family_{\mathrm{nt}}(\Sigma_p)}}\completion{p} \xrightarrow{\simeq} {\sphere^{tC_p}}\completion{p}$ of commutative algebras of spectra.
\end{lem}
\begin{proof}
    Consider the unit morphism
        $$\eta\colon\sphere \longrightarrow \mathrm{Coind}_{C_p}^{\Sigma_p}(\sphere) \in \calg(\spectra_{\Sigma_p}).$$
    By~\cite[Lem. ~10.11]{ABHS24}, $\eta$ induces the morphism
    $$\eta^{t\family_{\mathrm{nt}}(\Sigma_p)}\colon {\sphere^{t\family_{\mathrm{nt}}(\Sigma_p)}}\longrightarrow {\sphere^{tC_p}}.$$
    of commutative algebras, see also~\cite[Lem. ~10.18]{ABHS24}. We will show that $\eta^{t\family_{\mathrm{nt}}(\Sigma_p)}$ is an equivalence after $p$-completion.

    Indeed, by the Segal conjecture, the unit morphism $\sphere \longrightarrow \sphere^{tC_p}$ is an equivalence after $p$-completion. Therefore, the homomorphism 
    $$\pi_*\Big({\eta^{t\family_{\mathrm{nt}}(\Sigma_p)}}\completion{p}\Big) \colon \pi_*\Big({\sphere^{t\family_{\mathrm{nt}}(\Sigma_p)}}\completion{p}\Big) \longrightarrow \pi_*\Big({\sphere^{tC_p}}\completion{p}\Big)$$
    is surjective. Finally, $\pi_*\Big({\eta^{t\family_{\mathrm{nt}}(\Sigma_p)}}\completion{p}\Big)$ is also injective by the transfer agrument from~\cite[Prop.~A.8]{glasmanlawson}.
\end{proof}

For the next lemma, let us first establish some useful notations.

\begin{nota}
    The symmetric monoidal functors $\crossEffect_k\colon \excisive_d\rightarrow \elbowCat{d,k}$ and $P_{\udl{1}}\colon \elbowCat{d,k}\rightarrow\spectra$ induce maps in $\calg(\spectra)$
    \[\mapsp_{\excisive_d}(\unit_d,\unit_d)\xlongrightarrow{\crossEffect_k} \mapsp_{\elbowCat{d,k}}(\unit_{d,k},\unit_{d,k})\xlongrightarrow{P_{\udl{1}}} \mapsp_{\spectra}(\sphere,\sphere).\] Upon applying $\pi_0$, we get a map of commutative rings which we denote as
    \[\Phi^{[k]}\colon A(d)\longrightarrow \mathbb{Z}\] following the notation for geometric fixed points in equivariant homotopy theory. When $k=d$, the map $\Phi^{[d]}$ is nothing but the map constructed in \cref{cons:completion_at_ideals} which is used to define the augmentation ideal.
\end{nota}

\begin{lem}\label{lem:description_of_geom_fix_points_ideal}
   Let $d=p$ be a prime number. Then
    \begin{equation*}
        \Big(p,\: \mathrm{Im}\big(\Phi^{[k]}\colon I(d)\longrightarrow \mathbb{Z}\big)\Big) = \begin{cases}
             p\mathbb{Z} & \text{ if } k = 1;\\
            \mathbb{Z} & \text{ if } 1 < k < d.
        \end{cases}
    \end{equation*}
    Moreover, $\Phi^{[d]}(I(d))=0$.
\end{lem}
\begin{proof}
    Recall that the Goodwillie-Burnside ring $A(d)=\pi_0\unit_d$ is generated by the elements $[\lambda_i]$, $1\leq i \leq d$, see \cite[Def.~9.19 and Thm.~9.25]{ABHS24}. Moreover, by \cite[Lem.~9.22]{ABHS24},
    the homomorphism $\Phi^{[k]}\colon A(d) \longrightarrow \mathbb{Z}$ is given by
        $$\Phi^{[k]}([\lambda_i]) = |\surjectiveSet(k,i)|$$
    and the augmentation ideal $I(d)=\mathrm{Ker}(\Phi^{[d]})\subset A(d)$ is generated by the elements 
    $$ [\lambda_i] - |\surjectiveSet(d,i)| \in I(d).$$
    Therefore, we have
        $$\Big(p,\Phi^{[k]} \big(I(d)\big)\Big) = \Big(p, |\surjectiveSet(k,1)|-|\surjectiveSet(d,1)|,\ldots, |\surjectiveSet(k,d)|-|\surjectiveSet(d,d)|\Big) \subset \mathbb{Z}. $$

    The basic key observation is that $C_p$ acts freely on the set $\surjectiveSet(p,i)$ for $i \geq 2$ and so $|\surjectiveSet(p,i)|$ is a multiple of $p$ in these cases. Hence, if $k=1$ or $k=d$, then the assertion follows. Now let $1< k< d$. Again by the key observation, we see that
        $$\Big(p,\Phi^{[k]} \big(I(d)\big)\Big) = \Big(p, |\surjectiveSet(k,2)|,\ldots, |\surjectiveSet(k,d)|\Big) \subset \mathbb{Z}. $$
    Since $|\surjectiveSet(k,k)|=k!$ is coprime to $p$, we obtain that $\Big(p,\Phi^{[k]} \big(I(d)\big)\Big) = \mathbb{Z}$.
\end{proof}

We are finally ready to give the proof of the Segal conjecture in the case $d=p$.

\begin{proof}[Proof of \cref{thm:segal_conjecture_calculus}.]
    That the map of interest is an equivalence may be checked by passing to the symmetric monoidal $[k]$--geometric fixed point functors $\excisive_d(\spectra^{\omega},\spectra\completion{p})\rightarrow \spectra\completion{p}$ for all $1\leq k\leq d$. These are concretely given by $\Phi^{[k]}(-)\simeq \big(P_{\udl{1}}\crossEffect_k(-)(\sphere,\ldots,\sphere)\big)\completion{p}\simeq \big(\crossEffect_kP_k(-)(\sphere,\ldots,\sphere)\big)\completion{p}\simeq \big(\partial_k(-)\big)\completion{p}$. To this end, recall Arone--Ching's formula from \cite[Prop. 10.26]{ABHS24} which gives    \begin{equation}\label{eqn:arone_ching_derivative_formula}
        \Big(\partial_k\big((((-)^{\otimes d})^{h\Sigma_d})\completion{p}\big)\Big)\completion{p}\simeq \Big(\partial_k\big(((-)^{\otimes d})^{h\Sigma_d}\big)\Big)\completion{p}\simeq \prod_{\lambda \dashv d\colon |\lambda|=k}(\sphere^{t\family_{\mathrm{nt}}(\lambda)})\completion{p}.
    \end{equation}

    Almost by construction of the Borelification functor $\beta_*\beta^*$, the map of interest is an equivalence on $\Phi^{[d]}$ even before $p$--completion.  On the other hand, by \cref{prop:geometric_fixed_points_and_completions} and \cref{eqn:arone_ching_derivative_formula}, applying $\Phi^{[1]}$ and invoking \cref{lem:description_of_geom_fix_points_ideal,lem:p-completion_of_tates} gives 
    \[\sphere\completion{p}\longrightarrow {\sphere^{t\family_{nt}(\Sigma_p)}}\completion{p}\simeq {\sphere^{tC_p}}\completion{p},\] which is an equivalence by the usual Segal conjecture in equivariant homotopy theory.
    On the other hand, for $1<k<d$, \cref{eqn:arone_ching_derivative_formula} and~\cite[Prop. ~A.8]{glasmanlawson} gives that $\partial_k\big(((-)^{\otimes d})^{h\Sigma_d}\big)\completion{p}\simeq 0$; but then, also, $\big(p,\:\Phi^{[k]}(I(d))\big)=\mathbb{Z}\subseteq \mathbb{Z}$ by \cref{lem:description_of_geom_fix_points_ideal} and so $\Phi^{[k]}((\unit_d)\completion{I(d),p})\simeq 0$. Hence, the map of interest becomes an equivalence upon applying $\Phi^{[k]}$ for $1<k<d$ also. This completes the proof. 
\end{proof}

\begin{rmk}
    From the computations in the preceding proof, one can easily check  that the theorem would already fail when $d=p=3$  had we  not enforced the completion at the augmentation ideal. The case for $d=p=2$ still works, of course, since this is the case of genuine $C_2$--spectra, as proved by Lin \cite{lin}.
\end{rmk}

\printbibliography

@article{aroneChingChainRule,
  title={{Operads and chain rules for the calculus of functors}},
  author={Arone, G. and Ching, M.},
  series =  {Ast\'erisque },
    VOLUME = {338},
      YEAR = {2011},
     PAGES = {vi+158}
}

@article{aroneChingCrossEffects,
  title={{Cross-effects and the classification of Taylor towers}},
  author={Arone, G. and Ching, M.},
  series =  {Geom. Topol.},
    VOLUME = {20},
      YEAR = {2016},
     PAGES = {1445–1537}
}

@article{aroneChingClassification,
  title={{A classification of Taylor towers of functors of spaces and spectra}},
  author={Arone, G. and Ching, M.},
  series =  {Adv. Math.},
    VOLUME = {272},
      YEAR = {2015},
     PAGES = {471–552}
}

@article{ChingDayConvolution,
  title={{Infinity-operads and Day convolution in Goodwillie calculus}},
  author={Ching, M.},
  series =  {J. Lond. Math. Soc. (2)},
    VOLUME = {104},
    number = {3},
      YEAR = {2021},
     PAGES = {1204–1249}
}

@article{bauesDreckmannFranjouPirashvili,
  title={{Foncteurs polynomiaux et foncteurs de Mackey non lin\'eaires}},
  author={Baues, H.-J. and Dreckmann, W. and Franjou, V. and Pirashvili, T.},
  series =  {Bull. Soc. Math. France},
    VOLUME = {129},
    number = {2},
      YEAR = {2001},
     PAGES = {237–257}
}

@article{lin,
  title={{On Conjectures of Mahowald, Segal, and Sullivan}},
  author={Lin, W.-H.},
  journal = {Math. Proc. Camb. Phil. Soc.},
  year = {1980},
volume = {87},
number = {3},
    pages = {449--458}
}

@article{adamsGunawardenaMiller,
  title={{The Segal Conjecture for Elementary Abelian $p$-Groups - I}},
  author={Adams, J.F. and Gunawardena, J.H. and Miller, H.},
  journal = {Topology},
    volume = {24},
    number = {4},
  year = {1985},
    pages={435--460}
}

@article{carlssonSegalConjecture,
  title={{Equivariant Stable Homotopy and Segal's Burnside Ring Conjecture}},
  author={Carlsson, G.},
  journal = {Ann. of Math.},
  year = {1984},
    volume = {120},
    number ={2},
    pages ={189--224}
}

@book {AdamsBlueBook,
    AUTHOR = {Adams, J. F.},
     TITLE = {Stable homotopy and generalised homology},
    SERIES = {Chicago Lectures in Mathematics},
 PUBLISHER = {University of Chicago Press, Chicago, Ill.-London},
      YEAR = {1974},
     PAGES = {x+373},
   MRCLASS = {55B20 (55E10)},
  MRNUMBER = {402720},
MRREVIEWER = {P.\ S.\ Landweber},
}

@article{greenleesMayMU,
  title={{Localization and Completion Theorems for MU--module Spectra}},
  author={Greenlees, J.P.C. and May, J.P.},
  volume={146},
   journal={Ann. of Math.},
   year={1997},
   pages={509--544}
}

@article{nardinExposeIV,
  title={{Parametrized higher category theory and higher algebra: Expos\'{e} IV – Stability with respect to an
orbital $\infty$--category}},
  author={Nardin, D.},
  year ={2016},
  journal = {\href{https://arxiv.org/abs/1608.07704}{\underline{arXiv:1608.07704}}}
}

@article{glasmanStratified,
  title={{Stratified categories, geometric fixed points and a generalized Arone-Ching theorem}},
  author={Glasman, S.},
  year ={2017},
  journal = {\href{https://arxiv.org/abs/1507.01976}{\underline{arXiv:1507.01976}}}
}

@article{glasmanlawson,
  title={{Stable power operations}},
  author={Glasman, S. and Lawson, T.},
  year ={2020},
  journal = {\href{https://arxiv.org/abs/2002.02035}{\underline{arXiv:2002.02035}}}
}

@article {MNN17,
    AUTHOR = {Mathew, Akhil and Naumann, Niko and Noel, Justin},
     TITLE = {Nilpotence and descent in equivariant stable homotopy theory},
   JOURNAL = {Adv. Math.},
  FJOURNAL = {Advances in Mathematics},
    VOLUME = {305},
      YEAR = {2017},
     PAGES = {994--1084},
}

@article{heutsCategorifiedGoodwillie,
  title={{Goodwillie approximations to higher categories}},
  author={Heuts, G.},
  JOURNAL = {Mem. Amer. Math. Soc.},
    VOLUME = {272},
    number ={1333},
      YEAR = {2021},
     PAGES = {1--108}
}

@article{glasmanGoodwillie,
  title={{Goodwillie calculus and Mackey functors
}},
  author={Glasman, S.},
      YEAR = {2018},
      series={\href{https://arxiv.org/abs/1610.03127}{arXiv:1610.03127}}
}

@article{CMMN2,
  title={{Descent and Vanishing in Chromatic Algebraic K-Theory via Group Actions}},
  author={Clausen, D. and Mathew, A. and Naumann, N. and Noel, J.},
  JOURNAL = {Ann. Sci. Éc. Norm. Supér.},
    VOLUME = {57},
      YEAR = {2024},
        number = {4},
     PAGES = {1135–-1190}
}

@article{greenleesStableMapsFreeGSpaces,
  title={{Stable maps into free G-spaces}},
  author={Greenlees, J.P.C.},
  series =  {Trans. Amer. Math. Soc.},
    VOLUME = {310},
    number = {1},
      YEAR = {1988},
     PAGES = {199--215}
}

@article{bds16,
  title={{Grothendieck-Neeman duality and the
Wirthm¨uller isomorphism}},
  author={Balmer, P. and Dell'Ambrogio, I. and Sanders, B.},
  series =  {Compos. Math.},
    VOLUME = {152 (8)},
      YEAR = {2016},
     PAGES = {1740--1776}
}

@article{nardinShah,
  title={{Parametrized and Equivariant Higher Algebra}},
  author={Nardin, D. and Shah, J.},
  year = {2022},
  series =  {\href{https://arxiv.org/abs/2203.00072}{arXiv:2203.00072}}
}

@article{lurieHA,
  title={{Higher Algebra}},
  author={Lurie, J.},
  series =  {\url{https://www.math.ias.edu/~lurie/}},
  year = {2017}
}

@article{hopkinslurie,
  title={{Ambidexterity in $K(n)$-Local Stable Homotopy Theory}},
  author={Hopkins, M. and Lurie, J.},
  series =  {\url{https://www.math.ias.edu/~lurie/}},
  year = {2013}
}

@article{expose1Elements,
  title={{Parametrized Higher Category Theory and Higher Algebra: Exposé I -- Elements of Parametrized Higher Category Theory}},
  author={Barwick, C. and Dotto, E. and Glasman, S. and Nardin, D. and Shah, J.},
  year ={2016},
  series =  {\href{https://arxiv.org/abs/1608.03657}{arXiv:1608.03657}}
}

@article{parametrisedIntroduction,
      title={{Parametrized Higher Category Theory and Higher Algebra: A General Introduction}},
        author={Barwick, C. and Dotto, E. and Glasman, S. and Nardin, D. and Shah, J.},
      year={2016},
      series={\href{https://arxiv.org/abs/1608.03654}{arXiv:1608.03654}}
}

@article{wilsonSlices,
  title={{On categories of slices}},
  author={Wilson, D.},
  year = {2017},
    series={\href{https://arxiv.org/abs/1711.03472}{arXiv:1711.03472}}
}

@article{nardinThesis,
  title={{Stability and Distributivity
over Orbital $\infty$-categories}},
  author={Nardin, D.},
  series =  {PhD Thesis, MIT},
  year = {2017}
}

@article{shahThesis,
  title={{Parametrized Higher Category Theory}},
  author={Shah, J.},
  journal = {Alg. Geom. Top.},
  year = {2023}, 
    volume = {23},
    number = {2},
    pages = {509--644}
}

@article{kaifPresentable,
  title={{Parametrised Presentability over orbital categories}},
  author={Hilman, K.},
  year = {2024},
  JOURNAL = {Appl. Cat. Str.},
    VOLUME = {32}
}

@article{parametrised_Poincare_duality,
  title={{Parametrised Poincar\'e duality and equivariant fixed points methods}},
  author={Hilman, K. and Kirstein, D. and Kremer, C.},
  year = {2024},
  series =  {\href{https://arxiv.org/abs/2405.17641}{arXiv:2405.17641}}
}

@article{kaifNoncommMotives,
  title={{Parametrised Noncommutative Motives and Cubical Descent in Equivariant Algebraic K-theory}},
  author={Hilman, K.},
  journal = {\href{https://arxiv.org/submit/4151077/pdf}{arXiv}},
  year = {2024}
}

@article{CLLSpans,
  title={{Parametrized (higher) semiadditivity and the universality of spans}},
  author={Cnossen, B. and Haugseng, R. and Lenz, T. and Linskens, S.},
  year ={2025},
  journal = {\href{https://arxiv.org/abs/2403.07676}{arXiv:2403.07676}}
}

@article {nineAuthorsI,
    AUTHOR = {Calm\`es, Baptiste and Dotto, Emanuele and Harpaz, Yonatan and
              Hebestreit, Fabian and Land, Markus and Moi, Kristian and
              Nardin, Denis and Nikolaus, Thomas and Steimle, Wolfgang},
     TITLE = {Hermitian {K}-theory for stable {$\infty$}-categories {I}:
              {F}oundations},
   JOURNAL = {Selecta Math. (N.S.)},
  FJOURNAL = {Selecta Mathematica. New Series},
    VOLUME = {29},
      YEAR = {2023},
    NUMBER = {1},
     PAGES = {Paper No. 10, 269},
      ISSN = {1022-1824,1420-9020},
   MRCLASS = {19G38 (18N99 55U35)},
  MRNUMBER = {4514986},
MRREVIEWER = {Mohamed\ Elhamdadi},
       DOI = {10.1007/s00029-022-00758-2},
       URL = {https://doi.org/10.1007/s00029-022-00758-2},
}

@article {nineAuthorsII,
    AUTHOR = {Calm\`es, B. and Dotto, E. and Harpaz, Y.
              and Hebestreit, F. and Land, M. and Moi, K.
              and Nardin, D. and Nikolaus, T. 
              and Steimle, W.},
    TITLE = {Hermitian {K}-theory for stable {$\infty$}-categories {II}:
              {Cobordism categories and additivity}},
    JOURNAL = {Acta Math.},
    VOLUME = {235},
      YEAR = {2025},
    NUMBER = {2},
     PAGES = {149 - 400}
}

@article{ABHS24,
      title={The spectrum of excisive functors}, 
      author={Arone, G. and  Barthel, T. and  Heard, D. and  Sanders, B.},
      year={2025},
      journal = {Invent. Math.},
    volume = {241},
    pages = {363–-464}
}

@incollection {McCarthy99,
    AUTHOR = {McCarthy, Randy},
     TITLE = {Dual calculus for functors to spectra},
 BOOKTITLE = {Homotopy methods in algebraic topology ({B}oulder, {CO},
              1999)},
    SERIES = {Contemp. Math.},
    VOLUME = {271},
     PAGES = {183--215},
 PUBLISHER = {Amer. Math. Soc., Providence, RI},
      YEAR = {2001},
      ISBN = {0-8218-2621-2},
   MRCLASS = {18F25 (19D55 55P42)},
  MRNUMBER = {1831354},
       DOI = {10.1090/conm/271/04357},
       URL = {https://doi.org/10.1090/conm/271/04357},
}

@article {Kuhn04,
    AUTHOR = {Kuhn, Nicholas J.},
     TITLE = {Tate cohomology and periodic localization of polynomial
              functors},
   JOURNAL = {Invent. Math.},
  FJOURNAL = {Inventiones Mathematicae},
    VOLUME = {157},
      YEAR = {2004},
    NUMBER = {2},
     PAGES = {345--370},
      ISSN = {0020-9910,1432-1297},
   MRCLASS = {55P60 (55N22 55P65 55P91 55P92)},
  MRNUMBER = {2076926},
MRREVIEWER = {J.\ P. C. Greenlees},
       DOI = {10.1007/s00222-003-0354-z},
       URL = {https://doi.org/10.1007/s00222-003-0354-z},
}

@article{barwick1,
  title={{Spectral Mackey Functors and Equivariant Algebraic K-theory (I)}},
  author={Barwick, C.},
  JOURNAL = {Adv. Math.},
  FJOURNAL = {Advances in Mathematics},
    VOLUME = {304},
      YEAR = {2017},
     PAGES = {646--727}
}

@article{guillouMay,
  title={{Models of G-spectra as presheaves of spectra}},
  author={Guillou, B. and May, J.P.},
  year = {2024},
  journal = {Alg. Geom. Topol.},
  VOLUME = {24},
     PAGES = {1225–1275}
}

@article {GoodwillieCalculus1,
    AUTHOR = {Goodwillie, Thomas G.},
     TITLE = {Calculus. {I}. {T}he first derivative of pseudoisotopy theory},
   JOURNAL = {$K$-Theory},
  FJOURNAL = {$K$-Theory. An Interdisciplinary Journal for the Development,
              Application, and Influence of $K$-Theory in the Mathematical
              Sciences},
    VOLUME = {4},
      YEAR = {1990},
    NUMBER = {1},
     PAGES = {1--27},
      ISSN = {0920-3036,1573-0514},
   MRCLASS = {57Q10 (19M05 55P99)},
  MRNUMBER = {1076523},
       DOI = {10.1007/BF00534191},
       URL = {https://doi.org/10.1007/BF00534191},
}

@article {GoodwillieCalculus2,
    AUTHOR = {Goodwillie, Thomas G.},
     TITLE = {Calculus. {II}. {A}nalytic functors},
   JOURNAL = {$K$-Theory},
  FJOURNAL = {$K$-Theory. An Interdisciplinary Journal for the Development,
              Application, and Influence of $K$-Theory in the Mathematical
              Sciences},
    VOLUME = {5},
      YEAR = {1992},
    NUMBER = {4},
     PAGES = {295--332}
}

@article {GoodwillieCalculus3,
    AUTHOR = {Goodwillie, Thomas G.},
     TITLE = {Calculus. {III}. {T}aylor series},
   JOURNAL = {Geom. Topol.},
  FJOURNAL = {Geometry and Topology},
    VOLUME = {7},
      YEAR = {2003},
     PAGES = {645--711},
      ISSN = {1465-3060,1364-0380},
   MRCLASS = {55P65 (19D10 55P42 55U35)},
  MRNUMBER = {2026544},
MRREVIEWER = {Daniel\ C.\ Isaksen},
       DOI = {10.2140/gt.2003.7.645},
       URL = {https://doi.org/10.2140/gt.2003.7.645},
}
\end{document}